\documentclass[12pt]{article}

\usepackage{amsmath, amssymb, amsthm, physics, url, mathrsfs, cite, xcolor, mathtools, dsfont, enumitem, upgreek, subfig}

\mathtoolsset{showonlyrefs} 

\definecolor{lblue}{RGB}{55, 126, 168}

\usepackage[colorlinks, citecolor=lblue, linkcolor=lblue, urlcolor=lblue, bookmarks=false,
	pdfauthor={Aobo Chen}, 
	pdftitle={Stability of heat kernel bounds under pointed Gromov--Hausdorff convergence},
	pdfkeywords={Pointed Gromov--Hausdorff convergence, heat kernel estimates, metric geometry, Dirichlet form}
	]{hyperref}

\usepackage[margin=1 in]{geometry}

\numberwithin{equation}{section}
\numberwithin{figure}{section}

\newtheorem{theorem}{Theorem}[section]
\newtheorem{lemma}[theorem]{Lemma}
\newtheorem{proposition}[theorem]{Proposition}

\theoremstyle{definition}
\newtheorem{definition}[theorem]{Definition}
\newtheorem{example}[theorem]{Example}
\newtheorem{remark}[theorem]{Remark}
\newtheorem{notation}[theorem]{Notation}

\newcommand{\supp}{\mathrm{supp}}
\newcommand{\diam}{\mathrm{diam}}
\newcommand{\dist}{\mathrm{dist}}
\newcommand*{\dif}{\mathop{}\!\mathrm{d}}
\newcommand{\one}{\mathds{1}}
\newcommand{\set}[1]{\left\{ #1 \right\}}
\newcommand{\Sett}[2]{\left\{ #1  : \, #2 \right\}}
\newcommand\restr[2]{{\left.\kern-\nulldelimiterspace 
		#1 
		\vphantom{\big|}
		\right|_{#2}}} 

\font\titlefont=cmbx12 scaled 1400
\title{\titlefont Stability of heat kernel bounds under pointed Gromov--Hausdorff convergence}

\author{Aobo Chen\thanks{The author is partially supported by the National Natural Science Foundation of China (No.~12271282).}
}

\date{\today}

\begin{document}

\maketitle

\vspace{-0.7cm}

\begin{abstract}

We construct a conservative and strongly local regular symmetric Dirichlet form on the pointed Gromov--Hausdorff limit space and demonstrate the stability of heat kernel estimates under this convergence. Furthermore, we establish the Mosco convergence of the associated energy forms along a subsequence.


\end{abstract}

\section{Introduction}\label{s.intro}

Heat kernels, also known as transition densities, play a central role in the analysis of manifolds, metric spaces, and broader mathematical structures. In diverse geometric contexts, such as Riemannian manifolds, Alexandrov spaces, and metric measure spaces, heat kernels give valuable information about the geometric and analytical properties of the underlying space. Understanding their behavior becomes particularly relevant when studying the convergence of a sequence of these spaces.

The convergence of spaces can occur in multiple frameworks, such as the construction of fractals or the study of the tangent and asymptotic cones of general metric measure spaces, which can be formulated in the language of Gromov--Hausdorff topology. In \cite{KK94, KK96, Din02, Xu14}, the convergence of heat kernels, Green functions, and the Harnack inequality for Riemannian manifolds under Gromov--Hausdorff convergence is studied. In \cite{CHK17}, it is shown that if a sequence of measurable spaces equipped with resistance metric satisfies uniform volume doubling (equivalent to certain heat kernel bounds when the resistance forms are local, as shown in \cite{Kum04}), and also Gromov--Hausdorff converges, then the corresponding stochastic processes converge in distribution. As a result, the same type of heat kernel estimates holds for the limiting process.

In this work, we focus on the convergence of the following type of heat kernel bounds in the context of strongly local Dirichlet forms on metric measure spaces. Recall that we say a triple $(X,d,m)$ is a metric measure space, if $(X,d)$ is a locally compact separable metric space and $m$ is a nontrivial locally finite Borel-regular measure on $(X,d)$. For detailed definitions and notation on Dirichlet forms, see Section \ref{s.heat} or \cite{FOT11}.

\begin{definition}\label{d:HKE}
Let $(X,d,m)$ be a metric measure space and $\Psi:[0,\infty)\to[0,\infty)$ be a continuous increasing bijection of $[0, \infty)$ onto itself. We say that a Dirichlet form $(\mathcal{E},\mathcal{F})$ on $L^{2}(X,m)$ satisfies the \emph{heat kernel estimates} \hypertarget{HKE} $\hyperlink{HKE}{\mathrm{HKE}(\Psi)}$, if there exists
		$C_{1},c_{2},c_{3}, \delta\in(0,\infty)$ and a heat kernel $\set{p_t}_{t>0}$ of its semigroup $\{P_{t}\}_{t>0}$ such that for any $t>0$,
  		\begin{align}
		p_{t}(x,y) &\leq \frac{C_{1}}{V(x,\Psi^{-1}(t))} \exp\left(-c_{2}t\Phi\left(c_{3}\frac{d(x, y)}{t}\right)\right)
		\qquad \mbox{for $m$-a.e.\ $x,y \in X$}\\
		\text{and }\ p_{t}(x,y) &\geq \frac{C_{1}^{-1}}{V(x,\Psi^{-1}(t))}
		\qquad \mbox{for $m$-a.e.\ $x,y\in X$ with $d(x,y) \leq \delta \Psi^{-1}(t)$},
		\end{align}
		where $V(x,r):=m(B(x,r))$ and \begin{equation}\label{e.phi}
		\Phi(s):=\Phi_{\Psi}(s):=\sup_{r>0}\left({\frac{s}{r}-\frac{1}{\Psi(r)}}\right),\ s\in[0,\infty).
		\end{equation}
\end{definition}
The heat kernel estimate $\hyperlink{HKE}{\mathrm{HKE}(\Psi_{\beta})}$, where $\Psi_{\beta}(r):=r^{\beta}$ arises in various contexts. For $\beta=2$, it corresponds to the \emph{Gaussian heat kernel estimate}, which holds for Riemannian manifolds with non-negative Ricci curvature \cite{LY86}. For $\beta>2$, it corresponds to the \emph{sub-Gaussian heat kernel estimate}, which is valid for diffusions on many fractals, including the Sierpi\'{n}ski gasket \cite{BP88} and the Sierpi\'{n}ski carpet \cite{BB92}. 

The main result of this paper is as follows. See Sections \ref{s.pgh} and \ref{s.mosco} for the definitions of pointed Gromov--Hausdorff convergence and Mosco convergence, respectively.

\begin{theorem}\label{t.main}
Let $\{(X_{j},d_{j},m_{j}):1\leq j<\infty\}$ be a sequence of metric measure spaces. Let $(\mathcal{E}_{j},\mathcal{F}_{j})$ be a conservative, regular and strongly local symmetric Dirichlet form on $L^{2}(X_{j},m_{j})$ for each $1\leq j<\infty$. Fix functions $\Sett{\Psi_{j}:[0,\infty) \rightarrow [0,\infty)}{1\leq j\leq\infty}$ that are continuous, increasing bijections of $[0, \infty)$ onto itself. Assume that
\begin{enumerate}[label=\textup{(A{\arabic*})},align=right,leftmargin=*,topsep=5pt,parsep=0pt,itemsep=2pt]
\item\label{lb.as1} The Dirichlet form $(\mathcal{E}_{j},\mathcal{F}_{j})$ satisfies $\hyperlink{HKE}{\mathrm{HKE}(\Psi_{j})}$ for each $1\leq j<\infty$, with uniform constant $C_{1}$, $c_{2}$,  $c_{3}$ and $\delta$. 
\item\label{lb.as2} There exist constants $1<\beta\leq\beta^{\prime}$ and $C_{0}>1$ such that
\begin{equation}\label{e.psi}
C^{-1}_{0}\left({\frac{R}{r}}\right)^{\beta}\leq\frac{\Psi_{j}(R)}{\Psi_{j}(r)}\leq C_{0}\left({\frac{R}{r}}\right)^{\beta^{\prime}}\text{for all $1\leq j\leq \infty$ and all $0 < r \leq R<\infty$}.
\end{equation}
\item\label{lb.as3}  $\Psi_{j}\to \Psi_{\infty}$ uniformly on every compact subset of $[0,\infty)$.
\item\label{lb.as4}  There exists $\{V_{j}:[0,\infty)\to [0,\infty):1\leq j\leq \infty\}$, $V_{l}:[0,\infty)\to [0,\infty)$ and $V_{u}:[0,\infty)\to [0,\infty)$ that are increasing functions, and there exist constants $\alpha,\alpha^{\prime}>0$, $C_{\mathrm{v}}\geq1$ such that 

\begin{equation}
C_{\mathrm{v}}^{-1}\left({\frac{R}{r}}\right)^{\alpha}\leq\frac{V_{j}(R)}{V_{j}(r)}\leq C_{\mathrm{v}}\left({\frac{R}{r}}\right)^{\alpha^{\prime}},\quad \forall 0<r\leq R<\infty, j\in \mathbb{N}\cup\{u,l\},
\end{equation}
\begin{equation}\label{e.uesm}
C_{\mathrm{v}}^{-1}V_{j}(r)\leq m_{j}(B_{j}(x,r))\leq C_{\mathrm{v}} V_{j}(r),\ \forall x\in X_{j},\ 0< r<\diam(X_{j},d_{j}),\ 1\leq j<\infty,
\end{equation}and
\begin{equation}\label{e.ueslu}
C_{\mathrm{v}}^{-1}V_{l}(r)\leq V_{j}(r)\leq C_{\mathrm{v}} V_{u}(r),\ \forall r>0,\ 1\leq j\leq\infty.
\end{equation}
\item\label{lb.as5}  $V_{j}\to V_{\infty}$ uniformly on every compact subset of $[0,\infty)$.
\item\label{lb.as6} Assume further that $\{(X_{j},d_{j})\}_{1\leq j<\infty}$ are proper length spaces, $\inf_{1\leq j<\infty}\diam(X_{j},d_{j})>0$, and for some points $p_{j}\in X_{j}$, $1\leq j<\infty$, there holds 
\begin{equation}\label{e.gh}
(X_{j},d_{j},p_{j})\to (X_{\infty},d_{\infty},p_{\infty})\text{ in the pointed Gromov--Hausdorff topology.}
\end{equation}
for some complete metric space $(X_{\infty},d_{\infty})$ and some $p_{\infty}\in X_{\infty}$.
\end{enumerate}
Then the following holds:
\begin{enumerate}[label=\textup{(\arabic*)},align=right,leftmargin=*,topsep=5pt,parsep=0pt,itemsep=2pt]
\item\label{lb.meas} There exists a Radon measure $m_{\infty}$ on the Borel  $\sigma$-algebra of $(X_{\infty}, d_{\infty})$ such that for some $C\geq1$, \begin{equation}\label{e.minfty}
C^{-1}V_{\infty}(r)\leq m_{\infty}(B_{\infty}(x,r))\leq C V_{\infty}(r),\ \forall x\in X_{\infty},\ 0<r<\diam(X_{\infty},d_{\infty}).
\end{equation}
\item\label{lb.hke} There exists a conservative, regular and strongly local symmetric Dirichlet form $(\mathcal{E}_{\infty},\mathcal{F}_{\infty})$ on $L^{2}(X_{\infty}, m_{\infty})$ satisfying $\hyperlink{HKE}{\mathrm{HKE}(\Psi_{\infty})}$. 
\item\label{lb.mosco} Along a subsequence, $\mathcal{E}_{j}\rightarrow \mathcal{E}_{\infty}$ with respect to the Mosco topology, as $j\to\infty$.
\end{enumerate}
\end{theorem}
\begin{remark}\label{r.main}
\begin{enumerate}[label=\textup{({\roman*})},align=right,leftmargin=*,topsep=5pt,parsep=0pt,itemsep=2pt]
\item[]
\item By \cite[Lemma 5.7]{GK17}, under the assumption \ref{lb.as2}, for each $j$, the functions $\Phi_{\Psi_{j}}$ defined by \eqref{e.phi} are $[0, \infty)$-valued on $[0,\infty)$ and are strictly positive on $(0,\infty)$.
\item The assumption $\inf_{1\leq j<\infty}\diam(X_{j},d_{j}) > 0$ in \ref{lb.as6} is made to exclude the degenerate case where $\diam(X_{\infty}, d_{\infty}) = 0$, i.e., when $X_{\infty}$ is a singleton set (see Lemma \ref{l.diam}). If $X_{\infty} = \{p_{\infty}\}$ is a singleton, then any non-trivial measure $m$ on $X_{\infty}$ must be a multiple of the Dirac measure, and $p_{\infty}$ becomes a \emph{trap} for the diffusion process on $X_{\infty}$. In this case, the semigroup of the diffusion is the identity map, the corresponding regular Dirichlet form is $\mathcal{E} \equiv 0$ with domain $\mathcal{F} = \mathbb{R} \one_{\{p_{\infty}\}}=C(X_{\infty})$, and the heat kernel $p_{t}(x,x)=m(X_{\infty})^{-1}$ for all $(t,x)\in(0,\infty)\times \{p_{\infty}\}$.

\item\label{lb.Feller}Any Markov process on a volume doubling metric measure space that admits the heat kernel bound $\hyperlink{HKE}{\mathrm{HKE}(\Psi)}$ in Definition~\ref{d:HKE} is in fact a \emph{Feller} and \emph{strong Feller process}; see \cite[Proposition~3.2]{Lie15}. In particular, the Hunt process associated with $(\mathcal{E}_{\infty},\mathcal{F}_{\infty})$ in Theorem \ref{t.main}-\ref{lb.hke} enjoys the Feller and strong Feller properties. 

\item\label{lb.rm1} If a metric space $(X,d)$ satisfies the \emph{chain condition} \cite[Definition 6.9]{GT12} (for example, if $(X,d)$ is a proper length space \cite[Proposition A.1]{KM20}), and a Dirichlet form $(\mathcal{E},\mathcal{F})$ on $L^{2}(X,m)$ satisfies $\hyperlink{HKE}{\mathrm{HKE}(\Psi)}$, then the heat kernel estimates can be improved to the \emph{full heat kernel bound} \hypertarget{HKEf} $\hyperlink{HKEf}{\mathrm{HKE}_{\mathrm{f}}(\Psi)}$ \cite[Theorem 6.5]{GT12}. Specifically, the upper bound of $p_{t}$ in $\hyperlink{HKE}{\mathrm{HKE}(\Psi)}$ remains valid, while the lower bound of $p_{t}$ in $\hyperlink{HKE}{\mathrm{HKE}(\Psi)}$ is replaced by \[p_{t}(x,y)\geq \frac{C_{1}^{-1}}{V(x,\Psi^{-1}(t))} \exp\left(-c_{4}t\Phi\left(c_{5}\frac{d(x, y)}{t}\right)\right)
		\qquad \mbox{for $m$-a.e.\ $x,y \in X$}\]
		for some $C_{1},c_{4},c_{5}\in(0,\infty)$. Conversely, if the full heat kernel bound $\hyperlink{HKEf}{\mathrm{HKE}_{\mathrm{f}}(\Psi)}$ holds, then the chain condition of the metric space holds \cite[Theorem 2.11]{Mur20}.
		
By \cite[Proposition 11.3.12, Proposition 11.3.14]{HKST15}, \ref{lb.as6} implies that $(X_{\infty},d_{\infty})$ is also a proper length space. Therefore, the full heat kernel bound $\hyperlink{HKEf}{\mathrm{HKE}_{\mathrm{f}}(\Psi_{j})}$ holds for $(\mathcal{E}_{j},\mathcal{F}_{j})$ for all $1\leq j\leq \infty$.
\item\label{lb.rm2} 
Since the assumptions and conclusions in Theorem \ref{t.main} are stable under a bi-Lipschitz change of metric, we may instead assume that \emph{each $d_{j}$ is (uniformly) bi-Lipschitz equivalent to a length metric}, which can be ensured by the quasi-convexity \cite[Proposition 8.3.12]{HKST15} or by the chain condition \cite[Proposition A.1]{KM20}.
\item\label{lb.rm3} It is possible to prove similar results by assuming beforehand the existence of a measure $m_{\infty}$ on $X_{\infty}$ and that the sequence $(X_{n},d_{n},p_{n},m_{n})$ \emph{pointed measured Gromov--Hausdorff converges}\footnote{see \cite[Definition 11.4.5]{HKST15} or \cite[Definition 2.2]{KS03} for definitions.} to $(X_{\infty},d_{\infty},p_{\infty},m_{\infty})$. However, it appears that directly verifying pointed measured Gromov--Hausdorff convergence is more challenging.
\end{enumerate}
\end{remark}

We briefly outline the proof of Theorem \ref{t.main}. The Radon measure $m_{\infty}$ on limit space is derived from the tightness of the push-forward measures. To obtain the limit Dirichlet form $(\mathcal{E}_{\infty},\mathcal{F}_{\infty})$, we first expand the heat kernel of the part of $(\mathcal{E}_{n},\mathcal{F}_{n})$ on the balls $B_{n}(p_{n},R)$ with Dirichlet boundary conditions, using their eigenvalues and eigenfunctions (named \emph{Hilbert--Schmidt expansion}, see \eqref{e.hs(n,r)}). We then apply a version of Arzel\`{a}--Ascoli theorem (Lemma \ref{l.a-a}) to take the limit and obtain the limits of eigenvalues and eigenfunctions, which gives a heat kernel on $B_{\infty}(p_{\infty},R)$. Letting $R\to\infty$, we obtain a heat kernel on $X_{\infty}$ (Theorem \ref{t.qehke}). This heat kernel defines a semigroup (Theorem \ref{t.proQ}), corresponding to a Dirichlet form that satisfies the desired properties. The verification of Mosco convergence relies on the strong convergence of semigroups, as stated in the Mosco--Kuwae--Shioya's theorem (Theorem \ref{t.mks}). 

The construction of limit heat kernels using the Hilbert--Schmidt expansion was first introduced by Ding \cite{Din02} and later studied by Xu \cite{Xu14}, in the context of complete Riemannian manifolds with nonnegative Ricci curvature. Their approach relied on gradient estimates and Harnack's convergence theorem specific to manifolds, which are not applicable in the settings of fractal spaces and general metric measure Dirichlet spaces. In this work, we estimate the eigenvalues and eigenfunctions using only information derived from heat kernel estimates.

In \cite[Section 5]{KS03}, the authors show that under the uniform Poincar\'{e} inequality and other geometric conditions, the Dirichlet forms converge in the Mosco topology. However, it remains unknown whether the Poincar\'{e} inequality $\mathrm{(PI)}$ is stable, i.e., whether it holds for the limit Dirichlet form. In \cite[Section 11.6]{HKST15}, the stability of the $p$-Poincar\'{e} inequality for \emph{upper gradients} is proved. Our results indicate that the uniform heat kernel bounds on varying metric measure Dirichlet spaces imply a heat kernel bound on the limit metric measure Dirichlet space. Building on the well-known results on heat kernel estimates \cite{GHL15}, Theorem \ref{t.main} in this paper establishes that the package of conditions $\mathrm{(V)}+\mathrm{(PI)}+\mathrm{(CSA)}$ should be stable under Gromov--Hausdorff convergence, where $\mathrm{(V)}$ represents the homogeneous volume growth estimate and $\mathrm{(CSA)}$ represents the \emph{cutoff Sobolev inequality} (see \cite{AB15} or \cite[p.1492]{GHL15} for definitions).

This paper is organized as follows. In Section \ref{s.pgh}, we review the definition of pointed Gromov--Hausdorff convergence and prove that the approximate isometries can be chosen to be measurable. In Section \ref{s.doub}, we construct a doubling measure on the limit space and verify Theorem \ref{t.main}-\ref{lb.meas}. In Section \ref{s.heat}, we utilize the Hilbert--Schmidt expansion of the heat kernel and estimate the eigenvalues and eigenfunctions. In Section \ref{s.heatlim}, we construct the heat kernel with Dirichlet boundary conditions in the balls of the limit space. In Section \ref{s.dflim}, we obtain a heat kernel in the limit space and prove Theorem \ref{t.main}-\ref{lb.hke}. In Section \ref{s.mosco}, the Mosco convergence of Dirichlet forms, stated in Theorem \ref{t.main}-\ref{lb.mosco}, is established. In Section \ref{s.example}, we apply Theorem \ref{t.main} to reconstruct the diffusions on the Sierpi\'{n}ski carpet and Sierpi\'{n}ski gasket. We also provide an alternative proof in Theorem \ref{t.inverse} that for any $\alpha\geq1$ and $2\leq\beta\leq\alpha+1$, there exists a metric measure Dirichlet space that satisfies volume growth $r^{\alpha}$ and heat kernel bounds $\hyperlink{HKEf}{\mathrm{HKE}_{\mathrm{f}}(r^{\beta})}$.
\begin{notation}
In this paper, we use the following notation and conventions.
\begin{enumerate}[label=\textup{(\roman*)},align=right,leftmargin=*,topsep=5pt,parsep=0pt,itemsep=2pt]
\item Let $\{(X_{j},d_{j})\}_{1\leq j\leq\infty}$ be metric spaces. For $1\leq j\leq\infty$, $(x,r)\in X_{j}\times (0,\infty)$, we write the open ball in $X_{j}$ centered at $x$ with radii $r$ \[B_{j}(x,r):=\{z\in X_{j}:d_{j}(x,z)<r\}.\]
 		\item Let $X$ be a non-empty set. We define $\one_{A}=\one_{A}^{X}\in\mathbb{R}^{X}$ for $A\subset X$ by
	 \[\one_{A}(x):=\one_{A}^{X}(x):= \begin{cases}
	 	1 & \mbox{if $x \in A$,}\\
	 	0 & \mbox{if $x \notin A$.}
	 \end{cases} \]
	 
		\item Let $X$ be a topological space. We set
		$C(X):=\{f\mid\textrm{$f:X\to\mathbb{R}$, $f$ is continuous}\}$, $C_{b}(X):=\{f\in C(X)\mid\textrm{$f$ is bounded}\}$, $C_{0}(X):=\{f\in C(X)\mid\textrm{$f$ vanishes at infinity}\}$ and
		\[C_c(X):=\{f\in C(X)\mid\textrm{$X\setminus f^{-1}(0)$ has compact closure in $X$}\}.\]
\item For a measure $\mu$ on $(X,\mathcal{B}(X))$ and a measurable function $f:X\rightarrow Y$, we write $f_{\#}\mu$ for the push-forward measure defined by $f_{\#}\mu(A):=\mu(f^{-1}(A))$ for all measurable subsets of $Y$.

\item $\mathbb{Q}_{+}:=\mathbb{Q}\cap(0,\infty)$ is the set of all positive rational numbers. $\mathbb{N}:=\{1,2,\ldots\}$, that is $0\notin \mathbb{N}$.

\item We use the letters $C$, $c$ etc. to denote positive constant whose value is inessential and may change at each occurrence. 
\item For two extended real numbers $A,B\in\mathbb{R}\cup\{-\infty,\infty\}$, let $A\wedge B:=\min(A,B)$ and $A\vee B:=\max(A,B)$.

\item For real valued quantities $f$ and $g$, if there exists an implicit constant $C\geq1$ that depends on inessential parameters such that $f\leq Cg$ then we write $f\lesssim g$.  We write $f\asymp g$ whenever $f\lesssim g$ and $g\lesssim f$.
\end{enumerate}
\end{notation}

\section{Gromov--Hausdorff convergence}\label{s.pgh}

In this section, we review the basic theory of Gromov--Hausdorff distance and Gromov--Hausdorff convergence. The main reference is \cite[Chapter 11]{HKST15}.

Let $(Z,d_{Z})$ be a metric space. We say $(Z,d_{Z})$ is \emph{proper} if the closed balls $\{y\in Z:d_{Z}(z,y)\leq r\}$ are compact for all $(z,r)\in Z\times (0,\infty)$. We say $(Z,d_{Z})$ is a \emph{length} space, if the distance between any two points equals the smallest possible length of curves connecting them. By \cite[Lemma 8.3.11]{HKST15}, every proper length space is geodesic, i.e., any two points can be connected by a curve whose length matches the exact distance between them. 

For any $\epsilon>0$ and any nonempty subset $A\subset Z$, the $\epsilon$-neighborhood of $A$ is defined as the set
\[N_{\epsilon}(A):=\left\{z\in Z:\dist_{Z}(z,A)<\epsilon\right\},\]
where $\dist_{Z}$ is the distance under the metric $d_{Z}$. The Hausdorff distance between two nonempty subsets $A,B\subset Z$ is
\[d_{\mathrm{H}}(A,B):=\inf\left\{\epsilon>0:A\subset N_{\epsilon}(B)\text{ and }B\subset N_{\epsilon}(A)\right\}.\]

We define for each $x=\{x_{j}\}_{j=1}^{\infty}, y=\{y_{j}\}_{j=1}^{\infty}\in\mathbb{R}^{\mathbb{N}}$ the quantities $\lVert x\rVert_{\infty}:=\sup_{1\leq j<\infty}|x_{j}|$ and \[d_{l^{\infty}}(x,y):=\lVert x-y\rVert_{\infty}=\sup_{1\leq j<\infty}|x_{j}-y_{j}|.\] Let \[l^{\infty}:=\left\{ x=\{x_{j}\}_{j=1}^{\infty}\in\mathbb{R}^{\mathbb{N}}:\lVert x\rVert_{\infty}<\infty\right\},\]so that $(l^{\infty},d_{l^{\infty}})$ is a complete metric space. Recall that proper metric spaces are always separable. By the Fr\'{e}chet embedding theorem \cite[p.101]{HKST15}, every separable metric space $(X,d_{X})$ allows an isometric embedding $i:(X,d_{X})\to (l^{\infty},d_{l^{\infty}})$, i.e., $d_{l^{\infty}}( i(x_{1}),i(x_{2}))=d_{X}(x_{1},x_{2})$ for all $x_{1}, x_{2}\in X$. The \emph{Gromov--Hausdorff distance} between two separable metric spaces $(X,d_{X})$ and $(Y,d_{Y})$ is defined as \[d_{\mathrm{GH}}(X,Y):=\inf d^{\infty}_{\mathrm{H}}(i(X),j(Y)),\] where $d^{\infty}_{\mathrm{H}}$ is the Hausdorff distance in $(l^{\infty},d_{l^{\infty}})$, and the infimum is taken over all isometric embeddings $i:(X,d_{X})\to (l^{\infty},d_{l^{\infty}})$ and $j:(Y,d_{Y})\to (l^{\infty},d_{l^{\infty}})$. For a sequence of separable metric spaces $\{(X_{n},d_{n}):1\leq n\leq\infty\}$, we write $(X_{n},d_{n})\xrightarrow{\mathrm{GH}}(X_{\infty},d_{\infty})$ if $d_{\mathrm{GH}}(X_{n},X_{\infty})\to 0$ as $n\to\infty$, and we say that the sequence $\{(X_{n},d_{n}):1\leq n<\infty\}$ Gromov--Hausdorff converges to $(X_{\infty},d_{\infty})$. It is known that the Gromov--Hausdorff distance $d_{\mathrm{GH}}$ gives a complete, separable, and contractible metric on the collection of all isometry classes of compact metric spaces \cite[Theorem 11.1.15]{HKST15}.

As pointed out in \cite[Section 11.3]{HKST15}, the Gromov--Hausdorff convergence is quite restrictive for the class of non-compact metric spaces. Therefore, the following concept of \emph{pointed Gromov--Hausdorff convergence} is developed to deal with (possibly) non-compact spaces.

\begin{definition}[Pointed Gromov--Hausdorff convergence]\label{d.pgh}
A sequence of pointed separable metric spaces $\{(X_{j},d_{j},p_{j}):1\leq j<\infty\}$ is said to \emph{pointed Gromov--Hausdorff converge} to a pointed separable metric space $(X_{\infty},d_{\infty},p_{\infty})$ if, for each $r>0$ and $0<\epsilon<r$, there exists an $i_{0}$ such that for each $i\geq i_{0}$ there is a map $f_{i}^{\epsilon}: B_{i}(p_{i},r)\rightarrow X_{\infty}$ satisfying:
\begin{enumerate}[label=\textup{(\arabic*)},align=right,leftmargin=*,topsep=5pt,parsep=0pt,itemsep=2pt]
    \item $f_{i}^{\epsilon}(p_{i})=p_{\infty}$;
    \item $|d(f_{i}^{\epsilon}(x),f_{i}^{\epsilon}(y))-d_{i}(x,y)|<\epsilon$ for all $x,y\in B_{i}(p_{i},r)$;
    \item $B_{\infty}(p_{\infty},r-\epsilon)\subset  N_{\epsilon}(f_{i}^{\epsilon}(B_{i}(p_{i},r)))$, where $N_{\epsilon}(A)$ is the $\epsilon$-neighborhood of $A$.
\end{enumerate}
In this case, we write $(X_{n},d_{n},p_{n})\xrightarrow{\mathrm{p-GH}}(X_{\infty},d_{\infty},p_{\infty})$.
\end{definition}
By \cite[Proposition 11.3.5]{HKST15}, for bounded separable spaces, the notions of Gromov--Hausdorff convergence and pointed Gromov--Hausdorff convergence coincide. Gromov's compactness theorem in this context can be stated as follows: for every \emph{eventually proper} sequence in a \emph{pointed totally bounded} family of pointed length metric spaces, there exists a subsequence that pointed Gromov--Hausdorff converges to a proper length pointed metric space \cite[Proposition 11.3.12, Theorem 11.3.16]{HKST15}.

In general, the Gromov--Hausdorff approximations $f_{i}^{\epsilon}$ introduced in Definition \ref{d.pgh} may not be measurable with respect to the Borel $\sigma$-algebra. This lack of measurability  brings difficulties when analyzing such approximations. However, as we will see later, under the volume doubling property, we can carefully construct $f_{i}^{\epsilon}$ to be measurable functions. To begin, we recall the following covering theorem, which is due to David \cite{Dav88} and Christ \cite{Chr90}.
\begin{theorem}[{\cite[Theorem 11]{Chr90}}]\label{t.cubes}
Let $(Z,d,\mu)$ be a separable, volume doubling metric measure space, that is, there exists $C_{Z}>0$ such that $\mu(B(z,2r))\leq C_{Z}\mu(B(z,r))$ for all $z\in Z$ and all $r>0$. Then there exists some index set $\{I_{k}\}_{k\in\mathbb{Z}}$, a countable collection of open subsets $\left\{Q_{k,j}\subset Z\;\big|\;j\in I_{k},\ k\in\mathbb{Z}\right\}$, a countable set of points $\left\{z_{k,j}\in Z\;\big|\;j\in I_{k},\ k\in\mathbb{Z}\right\}$, and constants $\delta\in(0,1)$, $0<a_{0}\leq a_{1}<\infty$, such that 
\begin{enumerate}[label=\textup{(\arabic*)},align=right,leftmargin=*,topsep=5pt,parsep=0pt,itemsep=2pt]
    \item\label{lb.cov1} For all $k\in\mathbb{Z}$, 
    \begin{equation}
        \mu\left(Z{\bigg\backslash}\bigcup_{j\in I_{k}}Q_{k,j}\right)=0.
    \end{equation}
    \item\label{lb.cov2}  $Q_{k,i}\cap Q_{k,j}=\emptyset$ for each $k\in\mathbb{Z}$ and $i\neq j\in I_{k}$.
    \item\label{lb.cov3} For each $k\in\mathbb{Z}$ and $j\in I_{k}$, $B(z_{k,j},a_{0}\delta^{k})\subset  Q_{k,j}\subset  B(z_{k,j},a_{1}\delta^{k})$.
\end{enumerate}
Moreover, the constants $\delta\in(0,1)$ and $0<a_{0}\leq a_{1}<\infty$ depend only on $C_{Z}$.
\end{theorem}
\begin{proof}
The volume doubling property is equivalent to the definition of being \emph{homogeneous type} in \cite[Definition 1]{Chr90}. By \cite[Theorem 11 (3.1), (3.4) and (3.5)]{Chr90}, we can find cubes that satisfy \ref{lb.cov1} and \ref{lb.cov3}. \ref{lb.cov2} follows from \cite[Lemma 15]{Chr90}.
\end{proof}

The following proposition is the main result of this section.
\begin{proposition}\label{p.isom}
Assume \ref{lb.as4}, \ref{lb.as5} and \ref{lb.as6}. Then there exist two sequences of positive numbers $\{\epsilon_{n}\}_{n=1}^{\infty}$ and $\{R_{n}\}_{n=1}^{\infty}$ with $\epsilon_{n}\downarrow0$ and $R_{n}\uparrow\infty$ as $n\rightarrow\infty$, and two sequences of maps $f_{n}:\overline{B_{{n}}(p_{{n}},R_{n})}\rightarrow{B_{\infty}(p_{\infty},R_{n})}$ and $g_{n}:\overline{B_{\infty}(p_{\infty},R_{n})}\rightarrow{B_{{n}}(p_{{n}},R_{n})}$, $n\geq1$ such that the following are true:
\begin{enumerate}[label=\textup{(\arabic*)},align=right,leftmargin=*,topsep=5pt,parsep=0pt,itemsep=2pt]
    \item\label{lb.isom1} For all $n\geq1$, the maps $f_{n}$ and $g_{n}$ are measurable with respect to the completion of Borel $\sigma$-algebra on $X_{n}$ and $X_{\infty}$, respectively.
    \item\label{lb.isom2} For all $n\geq1$, \begin{equation}\label{e.iso1}
\text{$f_{n}(\overline{B_{n}(p_{n},R_{n})})$ is a finite $(2\epsilon_{n})$-net in $B_{\infty}(p_{\infty},R_{n})$.}    
\end{equation}
    \item\label{lb.isom3} For all $x,y\in\overline{B_{n}(p_{n},R_{n})}$, 
    \begin{equation}\label{e.iso2}
        \left|d_{\infty}(f_{n}(x),f_{n}(y))-d_{n}(x,y)\right|\leq\epsilon_{n},
    \end{equation}
    \begin{equation}    \label{e.iso3}
    d_{n}(x,g_{n}(f_{n}(x)))\leq\epsilon_{n}.
    \end{equation}
    \item\label{lb.isom4} For all $x,y\in \overline{B_{\infty}(p_{\infty},R_{n})}$, 
    \begin{equation}\label{e.iso4}
      \left|d_{\infty}(x,y)-d_{n}(g_{n}(x),g_{n}(y))\right|\leq\epsilon_{n},
    \end{equation}
    \begin{equation} \label{e.iso5}
    d_{\infty}(x,f_{n}(g_{n}(x)))\leq\epsilon_{n}.  
    \end{equation} 
    \item\label{lb.isom5} For all $n\geq1$, \begin{align}
        f_{n}(p_{n})&=p_{\infty}\label{e.iso6},\\
        g_{n}(p_{\infty})&=p_{n}\label{e.iso7}.
    \end{align} 
\end{enumerate}
\end{proposition}
We call $f_{n}$ and $g_{n}$ a family of \emph{$\epsilon_{n}$-approximate isometries}.
\begin{proof}
Let $\{\epsilon_{n}\}_{n=1}^{\infty}$ and $\{R_{n}\}_{n=1}^{\infty}$ be two positive sequences that will be chosen later, with $\epsilon_{n}\downarrow0$ and $R_{n}\uparrow\infty$ as $n\rightarrow\infty$.

Recall that a closed ball in a length space is the closure of the respective open ball. By the proof of \cite[Proposition 11.3.12]{HKST15}, for each $R>1$, we have \[\overline{B_{n}(p_{n},R)}\xrightarrow{\mathrm{GH}}\overline{B_{\infty}(p_{\infty},R)} \text{ as $n\rightarrow\infty$},\]i.e., \[d_{\text{GH}}(\overline{B_{n}(p_{n},R)},\overline{B_{\infty}(p_{\infty},R)} )\rightarrow0\text{ as $n\rightarrow\infty$.}\] 

By \cite[Lemma 2.8, Proposition 2.9]{Jan17}, there are monotone positive sequences $R_{n}\uparrow \infty$ and $s_{n}\downarrow 0$ such that $d_{\text{GH}}(B_{{n}}(p_{{n}},R_{n}),B_{\infty}(p_{\infty},R_{n}))\leq s_{n}\rightarrow0$ as $n\rightarrow\infty$.
For each $n\geq1$, we apply \cite[Proposition 11.1.9]{HKST15} and assume that $(B_{n}(p_{n},R_{n}),d_{n})$ and $(B_{\infty}(p_{\infty},R_{n}),d_{\infty})$ are isometrically embedded into a common metric space $(Z,d_{Z})$. We do not distinguish $(B_{n}(p_{n},R_{n}),d_{n})$ and $(B_{\infty}(p_{\infty},R_{n}),d_{\infty})$ with their isometric images, respectively. From \cite[Proposition 11.1.11 and (11.1.12)]{HKST15} we have \begin{equation}\label{e.hsdist}
d^{Z}_{\text{H}}(\overline{B_{n}(p_{n},R_{n})},\overline{B_{\infty}(p_{\infty},R_{n})})+d_{Z}(p_{n},p_{\infty})\leq 5s_{n}.
\end{equation} 
We first define the following:
\begin{enumerate}
[label=\textup{(\roman*)},align=right,leftmargin=*,topsep=5pt,parsep=0pt,itemsep=2pt]
\item For each $n\geq1$, we decompose the volume doubling space $(X_{n},d_{n},m_{n})$ according to Theorem \ref{t.cubes} and obtain an index set $\{I^{(n)}_{k}\}_{k\in\mathbb{Z}}$, a family of open sets \[\left\{Q_{k,j}^{(n)}\subset X_{n}\,\big|\, j\in I^{(n)}_{k},k\in\mathbb{Z}\right\},\] a set of points $\left\{z_{k,j}^{(n)}\in X_{n}\,\big|\, j\in I^{(n)}_{k},k\in\mathbb{Z}\right\}$, and constants $\delta\in(0,1)$ and $0<a_{0}\leq a_{1}$ that are independent of $n$. Define \[J_{k}^{(n)}:=\{j\in I^{(n)}_{k}:Q_{k,j}^{(n)}\cap \overline{B_{n}(p_{n},R_{n})}\neq\emptyset\}.\] For each $j\in J_{k}^{(n)}$, if $z_{k,j}^{(n)}\in Q_{k,j}^{(n)}\cap \overline{B_{n}(p_{n},R_{n})}$, we define $\widetilde{z}_{k,j}^{(n)}$ to be $z_{k,j}^{(n)}$, and if otherwise we define $\widetilde{z}_{k,j}^{(n)}$ to be any arbitrary (but fixed) point in $Q_{k,j}^{(n)}\cap \overline{B_{n}(p_{n},R_{n})}$.
\item \label{lb.constfn}Since $\overline{B_{\infty}(p_{\infty},R_{n})}$ is compact, thus totally bounded, we can find a finite $s_{n}$-net of $\overline{B_{\infty}(p_{\infty},R_{n})}$ with cardinality $N_{n}$, say $\{w_{j}\}_{j=1}^{N_{n}}$. We slightly modify this net so that $\{w_{j}\}_{j=1}^{N_{n}}\subset  {B_{\infty}(p_{\infty},R_{n})}$.  For each $z\in \overline{B_{n}(p_{n},R_{n})}$, by \eqref{e.hsdist} we are able to find $w_{z}\in \{w_{j}\}_{j=1}^{N_{n}}$ such that $d_{Z}(z,w_{z})\leq 6s_{n}$.
\end{enumerate}
We now define $f_{k,n}: \overline{B_{n}(p_{n},R_{n})}\to \{w_{j}\}_{j=1}^{N_{n}}\cup \{p_{\infty}\}\subset B_{\infty}(p_{\infty},R_{n})$ by 
\begin{equation}
f_{k,n}(z):=
\begin{cases}
p_{\infty} \; &\text{ if }z=p_{n}\\
w_{\widetilde{z}_{k,j}^{(n)}}\; &\text{ if }z\in \left(Q_{k,j}^{(n)}\cap \overline{B_{n}(p_{n},R_{n})}\right)\setminus\{p_{n}\},\ j\in J_{k}^{(n)}\\
w_{z} \quad &\text{ if }z\in \left(\overline{B_{n}(p_{n},R_{n})}\setminus\bigcup_{j\in J_{k}^{(n)}} Q_{k,j}^{(n)}\right)\setminus\{p_{n}\}.
\end{cases}
\end{equation}
It is easy to verify, using \eqref{e.hsdist}, the definition of $z\mapsto w_{z}$ in \ref{lb.constfn} and triangle inequality that \begin{equation}\label{e.disfzz}
d_{Z}(f_{n,k}(z),z)\leq d_{Z}(f_{n,k}(z),w_{z})+d_{Z}(w_{z},z)\leq 2a_{1}\delta^{k}+6s_{n} \text{ for all $z\in \overline{B_{n}(p_{n},R_{n})}$}.
\end{equation} Therefore for $z_{1},z_{2}\in  \overline{B_{n}(p_{n},R_{n})}$, 
\begin{align}
&\phantom{\leq \ }\abs{d_{\infty}(f_{n,k}(z_{1}),f_{n,k}(z_{2}))-d_{n}(z_{1},z_{2})}
\\
&=\abs{d_{Z}(f_{n,k}(z_{1}),f_{n,k}(z_{2}))-d_{Z}(z_{1},z_{2})}
\\
&\leq d_{Z}(f_{n,k}(z_{1}),z_{1})+d_{Z}(f_{n,k}(z_{2}),z_{2})\leq 12(a_{1}\delta^{k}+s_{n}).\label{e.fkniso}
\end{align}

For each $w\in \overline{B_{\infty}(p_{\infty},R_{n})}$, we can find a $v_{w}\in \overline{B_{n}(p_{n},R_{n})}$ such that $d_{Z}(w,v_{w})\leq 6s_{n}$. Since the ball $B_{n}(v_{w},s_{n})$ must intersect one of the sets $\{Q_{k,j}^{(n)}\cap \overline{B_{n}(p_{n},R_{n})}\}_{j\in J_{k}^{(n)}}$ as the measure $m_{n}$ has full support, we can choose $\widetilde{v}_{w}\in \bigcup_{j\in J_{k}^{(n)}}\left(Q_{k,j}^{(n)}\cap \overline{B_{n}(p_{n},R_{n})}\right)\setminus\{p_{n}\}$ so that $d_{Z}(w,\widetilde{v}_{w})\leq 7s_{n}$. Therefore \begin{align}
d_{\infty}(f_{k,n}(\widetilde{v}_{w}),w)&=d_{Z}(f_{k,n}(\widetilde{v}_{w}),w)\\
&\leq d_{Z}(f_{k,n}(\widetilde{v}_{w}),\widetilde{v}_{w})+d_{Z}(\widetilde{v}_{w},w)\\
&\overset{\eqref{e.disfzz}}{\leq} 12(a_{1}\delta^{k}+s_{n})+s_{n}\leq 24(a_{1}\delta^{k}+s_{n})
\end{align}
so \begin{equation}\label{e.fnnet}
\text{$f_{k,n}(\overline{B_{n}(p_{n},R_{n})})$ is a finite $24(a_{1}\delta^{k}+s_{n})$-net in $B_{\infty}(p_{\infty},R_{n})$.}
\end{equation}

Now let $f_{n}:=f_{n,n}$ and $\epsilon_{n}=12(a_{1}\delta^{n}+s_{n})$. Since every subset of  $\overline{B_{n}(p_{n},R_{n})}\setminus\bigcup_{j\in J_{k}^{(n)}} Q_{k,j}^{(n)}$ is in the completion of Borel $\sigma$-algebra on $\overline{B_{n}(p_{n},R_{n})}$, we see that \ref{lb.isom1} holds for $f_{n}$. \eqref{e.iso1}, \eqref{e.iso2} and \eqref{e.iso6} follows from \eqref{e.fnnet}, \eqref{e.fkniso} and the definition of $f_{n}$, respectively. 

From Proposition \ref{p.mea.tan}, whose proof depends only on the existence of $f_{n}$, \eqref{e.iso2}, \eqref{e.iso1} and \eqref{e.iso6}, we see that $(X_{\infty},d_{\infty})$ admits a volume doubling measure. Thus we can apply Theorem \ref{t.cubes} again to the limit space $(X_{\infty},d_{\infty})$ and construct $g_{n}$, in the same way as we construct $f_{n}$, so that \ref{lb.isom1}, \eqref{e.iso4} and \eqref{e.iso7} hold. By \eqref{e.iso2}, \eqref{e.iso4} and triangle inequality, \[d_{n}(x,g_{n}(f_{n}(x)))\leq d_{Z}(x,f_{n}(x))+d_{Z}(f_{n}(x),g_{n}(f_{n}(x)))\leq \epsilon_{n},\]which gives \eqref{e.iso3}. \eqref{e.iso5} holds by the same argument.
\end{proof}

The following lemma, concerning the convergence of diameters, will be used in subsequent estimates. Its proof is a direct adaptation of \cite[Proposition 2.13]{Jan17}.
\begin{lemma}\label{l.diam}
If $(X_{n},d_{n},p_{n})$ are proper length spaces that pointed Gromov--Hausdorff converge to $ (X_{\infty},d_{\infty},p_{\infty})$, then \[\diam(X_{j},d_{j})\to \diam(X_{\infty},d_{\infty})\text{ as $j\to\infty$}.\]
\end{lemma}
Since we assumed in \ref{lb.as6} that \begin{equation}
\ell:=\inf_{1\leq j<\infty}\diam(X_{j},d_{j})>0\label{e.infdiam}
\end{equation}
By Lemma \ref{l.diam}, we know that $\diam(X_{\infty},d_{\infty})\geq\ell>0$ so $X_{\infty}$ is not a singleton.

\section{Doubling measure on limit space}\label{s.doub}
In this section, we assume $(X_{j},d_{j},p_{j})\to (X_{\infty},d_{\infty},p_{\infty})$ under the pointed Gromov--Hausdorff convergence. Let $\{f_{n}\}$ and $\{g_{n}\}$ be the family of $\epsilon_{n}$-approximate isometries in Proposition \ref{p.isom}. The goal in this section is to construct a Radon measure $m_{\infty}$ on $X_{\infty}$ that satisfies \eqref{e.minfty}.

\begin{lemma}[{\cite[Proposition 11.5.3]{HKST15}}]\label{l.annu}
    Let $(Z,d)$ be a length space and $\mu$ be a volume doubling measure on $Z$ with constant $C_{\mathrm{v}}$. Then for every $z\in Z$ and $0\leq r<R$, \[\mu\left(B(z,R)\setminus\overline{B(z,r)}\right)\leq C\left(\frac{R-r}{R}\right)^{\gamma}\mu(B(z,R)),\]where \[\gamma=\frac{\log\left({1+C_{\mathrm{v}}^{-4}}\right)}{\log3}\text{ and } C=6^{\gamma}.\]
\end{lemma}

The following proposition includes the proof of Theorem \ref{t.main}-\ref{lb.meas}.
\begin{proposition}\label{p.mea.tan}
   Assume \ref{lb.as4}, \ref{lb.as5}, and \ref{lb.as6}. There exists a Radon measure $m_{\infty}$ on $(X_{\infty},d_{\infty})$ such that
\begin{enumerate}[label=\textup{(\arabic*)},align=right,leftmargin=*,topsep=5pt,parsep=0pt,itemsep=2pt]
\item the volume growth satisfies \eqref{e.minfty}.
\item along some subsequence of $\{(X_{n},d_{n},p_{n})\}_{n=1}^{\infty}$, 
     \[ \left(\left(f_{n}\right)_{\#}\left(m_{n}|_{\overline{B_{n}(p_{n},R_{n})}}\right)\right)\Big|_{B_{\infty}(p_{\infty},R)}\rightarrow m_{\infty}|_{B_{\infty}(p_{\infty},R)}\quad\text{weakly as }n\rightarrow\infty,\]
    for every $R\in\mathbb{Q}_{+}\cup\{R_{n}\}_{n=1}^{\infty}$, in the sense that, for all $\phi\in C(B_{\infty}(p_{\infty},R))$ that is bounded and uniformly continuous,
     \begin{equation}\label{e.vague}
     \int_{B_{\infty}(p_{\infty},R)}\phi \dif\left(f_{n}\right)_{\#}\left(m_{n}|_{\overline{B_{n}(p_{n},R_{n})}}\right)\rightarrow\int_{B_{\infty}(p_{\infty},R)}\phi\dif m_{\infty}\text{ as }n\rightarrow\infty.
     \end{equation}
\end{enumerate}
\end{proposition}

\begin{proof}
    The proof consists of four steps. For each $R>0$, we will construct a measure on the ball $B_{\infty}(p_{\infty},R)$, and then unify these measures to get the desired globally defined measure.
    \begin{enumerate}[label=\textit{Step {\arabic*}.},align=right,leftmargin=*,topsep=5pt,parsep=0pt,itemsep=2pt]%
\item For the first step, we fix the $0<R<\infty$. Let $\{\epsilon_{n}\}$ and $\{R_{n}\}$ be in Proposition \ref{p.isom}. Since $R_{n}\rightarrow\infty$, we assume that $R_{n}>R+1$. By \eqref{e.iso2} and \eqref{e.iso6}, \[
        B_{n}(p_{n},R-\epsilon_{n})\subset  \left\{x\in \overline{B_{n}(p_{n},R_{n})}\;\Big|\; f_{n}(x)\in B_{\infty}(p_{\infty},R)\right\}\subset  B_{n}(p_{n},R+\epsilon_{n}).\] 
    Thus 
        \[\begin{aligned}C_{\mathrm{v}}^{-1}V_{n}(R-\epsilon_{n})&\leq m_{n}(B_{n}(p_{n},R-\epsilon_{n}))\\ &\leq\left(f_{n}\right)_{\#}\left(m_{n}|_{\overline{B_{n}(p_{n},R_{n})}}\right)(B_{\infty}(p_{\infty},R))\\
        &\leq m_{n}(B_{n}(p_{n},R+\epsilon_{n}))\leq C_{\mathrm{v}}V_{n}(R+\epsilon_{n})
    \end{aligned}\]
    We may choose a subsequence such that along this subsequence, the limit\[\displaystyle{\lim_{n\rightarrow\infty}\left(f_{n}\right)_{\#}\left(m_{n}|_{\overline{B_{n}(p_{n},R_{n})}}\right)(B_{\infty}(p_{\infty},R))}\] exists and is positive. 
   \item For this $R$ and along this subsequence, we claim that \[\left\{\left(\left(f_{n}\right)_{\#}\left(m_{n}|_{\overline{B_{n}(p_{n},R_{n})}}\right)\right)\Big|_{B_{\infty}(p_{\infty},R)}\right\}\] as a set of measures on $B_{\infty}(p_{\infty},R)$, is tight, that is, for every $\epsilon>0$, there exists a compact subset $K\subset  B_{\infty}(p_{\infty},R)$ such that \[\left(f_{n}\right)_{\#}\left(m_{n}|_{\overline{B_{n}(p_{n},R_{n})}}\right)\left(B_{\infty}(p_{\infty},R)\setminus K\right)\leq\epsilon\quad\text{for all } n\geq1.\] In fact, for any given $\epsilon>0$, we can choose $r_{\epsilon}<R$  and $N_{\epsilon}\in\mathbb{N}_{+}$ such that, for all $r\in(r_{\epsilon},R)$ and $n\geq N_{\epsilon}$,
    \[\begin{aligned} &\left(f_{n}\right)_{\#}\left(m_{n}|_{\overline{B_{n}(p_{n},R_{n})}}\right)\left(B_{\infty}(p_{\infty},R)\setminus \overline{B_{\infty}(p_{\infty},r)}\right)\\ =\ &m_{n}\left(\left\{x\in \overline{B_{n}(p_{n},R_{n})}\;\Big|\; f_{n}(x)\in B_{\infty}(p_{\infty},R)\setminus \overline{B_{\infty}(p_{\infty},r)} \right\}\right)\\ \leq\ &m_{n}\left(B_{n}(p_{n},R+2\epsilon_{n})\setminus \overline{B_{n}(p_{n},r-2\epsilon_{n})}\right)\\ \leq\ & 6^{\gamma}C_{\mathrm{v}}(R-r+4\epsilon_{n})^{\gamma}(R+2\epsilon_{n})^{-\gamma}V_{u}(R+2\epsilon_{n})\\ \leq\ & 6^{\gamma}C_{\mathrm{v}}(R-r_{\epsilon}+4\epsilon_{n})^{\gamma}(R+2\epsilon_{n})^{-\gamma}V_{u}(R+2\epsilon_{n})<\epsilon
    \end{aligned}\]
    Since $\left\{\left(\left(f_{n}\right)_{\#}\left(m_{n}|_{\overline{B_{n}(p_{n},R_{n})}}\right)\right)\Big|_{B_{\infty}(p_{\infty},R)}\right\}_{n=1}^{N_{\epsilon}}$ is a finite set and by \eqref{e.iso1} the elements in the image of $f_{n}$ are also finite for each $n$, we can choose a $r^{\prime}_{\epsilon}<R$ such that, for all $r\in(r^{\prime}_{\epsilon},R)$, $f_{n}^{-1}\left(B_{\infty}(p_{\infty},R)\setminus \overline{B_{\infty}(p_{\infty},r)}\right)=\emptyset$, $1\leq n\leq N_{\epsilon}$. Hence, for such $r$,
    \begin{equation*}
        \left(f_{n}\right)_{\#}\left(m_{n}|_{\overline{B_{n}(p_{n},R_{n})}}\right)\left(B_{\infty}(p_{\infty},R)\setminus \overline{B_{\infty}(p_{\infty},r)}\right)=0,\quad 1\leq n\leq N_{\epsilon}
    \end{equation*}
    Thus for any $r\in(\max(r^{\prime}_{\epsilon},r_{\epsilon}),R)$, \begin{equation*}
        \left(f_{n}\right)_{\#}\left(m_{n}|_{\overline{B_{n}(p_{n},R_{n})}}\right)\left(B_{\infty}(p_{\infty},R)\setminus \overline{B_{\infty}(p_{\infty},r)}\right)<\epsilon, \text{ for all }n\geq1.
    \end{equation*}
which means that $\left\{\left(\left(f_{n}\right)_{\#}\left(m_{n}|_{\overline{B_{n}(p_{n},R_{n})}}\right)\right)\Big|_{B_{\infty}(p_{\infty},R)}\right\}_{n=1}^{\infty}$ is tight. We then apply the non-compact version of Prokhorov's theorem \cite[Exercise 1.10.29]{Tao10} and a well-known fact that, if the vague convergent limit is a probability measure then the convergence holds in the weak sense \cite[Problem 5.10]{Bil99}, there exists a non-zero probability measure $\widetilde{\mu_{R}}$ on $B_{\infty}(p_{\infty},R)$ such that, for some subsequence, \[\frac{\left(\left(f_{n}\right)_{\#}\left(m_{n}|_{\overline{B_{n}(p_{n},R_{n})}}\right)\right)\Big|_{B_{\infty}(p_{\infty},R)}}{\left(\left(f_{n}\right)_{\#}\left(m_{n}|_{\overline{B_{n}(p_{n},R_{n})}}\right)(B_{\infty}(p_{\infty},R))\right)}\rightarrow \widetilde{\mu_{R}}\text{ weakly as }n\rightarrow\infty.\]
    We now write $\mu_{R}=\lim_{n\rightarrow\infty}\left(f_{n}\right)_{\#}\left(m_{n}|_{\overline{B_{n}(p_{n},R_{n})}}\right)(B_{\infty}(p_{\infty},R))\cdot\widetilde{\mu_{R}}.$ By a standard diagonal argument, we can extract a common subsequence of $\mathbb{N}$ such that \begin{equation}\label{e.weakpf}
    \left(\left(f_{n}\right)_{\#}\left(m_{n}|_{\overline{B_{n}(p_{n},R_{n})}}\right)\right)\Big|_{B_{\infty}(p_{\infty},R)}\rightarrow \mu_{R} \quad\text{weakly, for all }R\in\mathbb{Q}_{+}\cup\{R_{n}\}_{n=1}^{\infty}.
    \end{equation}

\item If $R_{1}\neq R_{2}$, let us show that $m_{R_{1}}$ and $m_{R_{2}}$ agree on their common domain. Without loss of generality, assume $R_{1}<R_{2}$. It is straightforward to see that \[\left(\left(f_{n}\right)_{\#}\left(m_{n}|_{\overline{B_{n}(p_{n},R_{n})}}\right)\right)\Big|_{B_{\infty}(p_{\infty},R_{1})}\rightarrow m_{R_{2}}|_{B_{\infty}(p_{\infty},R_{1})}\quad \text{weakly},\]
    by the uniqueness of weak convergence, we conclude that $m_{R_{2}}|_{B_{\infty}(p_{\infty},R_{1})}=m_{R_{1}}.$ We can then define a set function $m_{\infty}$ on $(X_{\infty},d_{\infty})$ by \[m_{\infty}(E)=\lim_{\mathbb{Q}_{+}\ni R\rightarrow\infty} \mu_{R}\left(E\cap B_{\infty}(p_{\infty},R)\right), \quad\text{for all Borel subsets of }(X_{\infty},d_{\infty}).\]The limit exists, as the sequence is monotone increasing. It is straightforward to verify by \cite[Proposition 3.3.44]{HKST15} that $m_{\infty}$ is a Radon measure.

\item The final step is to verify that $m_{\infty}$ satisfies \eqref{e.minfty}. In fact, for any $x\in X_{\infty}$ and $0<r<\diam(X_{\infty},d_{\infty})$, we may choose $n$ large enough such that $B_{\infty}(x,r)\subset  B_{\infty}(p_{\infty},R_{n})$ and $r+{4}\epsilon_{n}<\diam(X_{n},d_{n})$ by Lemma \ref{l.diam}. Since every metric space is normal, by the Urysohn's lemma, there exist functions $\phi_{1}\in C_{c}(B_{\infty}(x,r);[0,1])$ with $\phi_{1}=1$ on $B_{\infty}(x,r-\epsilon_{n})$, and $\phi_{2}\in C_{c}(B_{\infty}(x,r+\epsilon_{n});[0,1])$ with $\phi_{2}=1$ on $B_{\infty}(x,r)$. By \eqref{e.iso1}, we can find a $x_{n}\in B_{n}(p_{n},R_{n})$ such that $d_{\infty}(f_{n}(x_{n}),x)\leq 2\epsilon_{n}$. Therefore 
\[\begin{aligned}
C_{\mathrm{v}}^{-1}V_{n}(r-4\epsilon_{n})&\leq m_{n}\left(B_{n}\left(x_{n},r-4\epsilon_{n}\right)\right)\\ &\leq\left(f_{n}\right)_{\#}\left(m_{n}|_{\overline{B_{n}(p_{n},R_{n})}}\right)\left(B_{\infty}\left(x,r-\epsilon_{n}\right)\right)\\ &\leq\left(f_{n}\right)_{\#}\left(m_{n}|_{\overline{B_{n}(p_{n},R_{n})}}\right)(B_{\infty}(x,r+\epsilon_{n}))\\
    &\leq m_{n}(B_{n}(x_{n},r+4\epsilon_{n}))\leq C_{\mathrm{v}}V_{n}(r+4\epsilon_{n}).
    \end{aligned}\]
    Hence by the convergence of $V_{n}$ and using the uniform volume doubling property \ref{lb.as4} when $\epsilon_{n}\leq r/8$, we have \[\begin{aligned}
        2^{-\alpha^{\prime}}C_{\mathrm{v}}^{-1}V_{\infty}(r)&\leq\liminf_{n\rightarrow\infty}\left(f_{n}\right)_{\#}\left(m_{n}|_{\overline{B_{n}(p_{n},R_{n})}}\right)(B_{\infty}(x,r-\epsilon_{n}))\\ &\leq\lim_{n\rightarrow\infty}\int_{B_{\infty}(x,r)}\phi_{1}\dif \left(f_{n}\right)_{\#}\left(m_{n}|_{\overline{B_{n}(p_{n},R_{n})}}\right)\\&\leq \mu_{R}(B_{\infty}(x,r))=m_{\infty}(B_{\infty}(x,r))\\&\leq \lim_{n\rightarrow\infty}\int_{B_{\infty}(x,r)}\phi_{2}\dif \left(f_{n}\right)_{\#}\left(m_{n}|_{\overline{B_{n}(p_{n},R_{n})}}\right)\leq C_{\mathrm{v}}{2^{\alpha^{\prime}}}V_{\infty}(r),
    \end{aligned}\]
    which gives \eqref{e.minfty}. The weak convergence of measures is given by \eqref{e.weakpf}.
\end{enumerate}
\end{proof}

The following concept of convergence of the $L^{2}$ space is a special case of convergence of Hilbert spaces considered in \cite[Section 2.2]{KS03}.
\begin{definition}\label{d.conhil}
Let $\left\{(X_{n},m_{n}):1\leq n \leq \infty\right\}$ be a sequence of measurable spaces. We say that the Hilbert spaces \emph{$L^{2}(X_{n},m_{n})$ converge to $L^{2}(X_{\infty},m_{\infty})$} if there exists a dense subspace $\mathcal{C}\subset L^{2}(X_{\infty},m_{\infty})$ and a sequence of linear operators $\upphi_{n}:\mathcal{C}\to L^{2}(X_{n},m_{n})$ with the following property:
\begin{equation}
\lim_{n\to\infty}\lVert \upphi_{n} u\rVert_{L^{2}(X_{n},m_{n})}=\lVert u\rVert_{L^{2}(X_{\infty},m_{\infty})},\ \forall u\in\mathcal{C}.
\end{equation}
\end{definition}
\begin{theorem}\label{t.conv-Hil}
Assume \ref{lb.as4}, \ref{lb.as5}, \ref{lb.as6} and let $m_{\infty}$ be given by Proposition \ref{p.mea.tan}. Then $L^{2}(X_{n},m_{n})$ converges to $L^{2}(X_{\infty},m_{\infty})$. Moreover, we can take $\mathcal{C}:=C_{c}(X_{\infty})$ and $\upphi_{n}: \mathcal{C}\to L^{2}(X_{n},m_{n})$ defined by \begin{equation}\label{e.upphi}
\upphi_{n} u(x):=\begin{dcases}
(u\circ f_{n})(x)\ \quad & x\in \overline{B_{n}(p_{n},R_{n})}\\
0\ \quad & x\in X_{n}\setminus \overline{B_{n}(p_{n},R_{n})},
\end{dcases}
\end{equation}
in Definition \ref{d.conhil}.
\end{theorem}
\begin{proof}
For $u\in \mathcal{C}$, we assume that $\supp(u)\subset B_{\infty}(p_{\infty},R)$ for some $R\in\mathbb{Q}_{+}$. Then by \eqref{e.vague} \begin{align}
\lVert \upphi_{n} u\rVert_{L^{2}(X_{n},m_{n})}^{2}&=\int_{X_{n}}(\upphi_{n} u)^{2}\dif m_{n}=\int_{\overline{B_{n}(p_{n},R_{n})}}(u\circ f_{n})(x)^{2}\dif m_{n}(x)\\
&=\int_{B_{\infty}(p_{\infty},R)}u^{2}\dif \left(f_{n}\right)_{\#}\left(m_{n}|_{\overline{B_{n}(p_{n},R_{n})}}\right)\\
&\to \int_{B_{\infty}(p_{\infty},R)}u^{2}\dif m_{\infty}=\lVert u\rVert_{L^{2}(X_{\infty},m_{\infty})}^{2}
\end{align}
as $n\to\infty$.
\end{proof}
\section{Estimates on heat kernel and eigenfunctions}\label{s.heat}
In this section, we will estimate the growth of eigenvalues and the eigenfunctions of Dirichlet forms on balls with Dirichlet boundary conditions. Throughout this section, we assume that all the conditions in Theorem \ref{t.main} are satisfied.

We begin by briefly reviewing some standard notation related to Dirichlet forms and refer the reader to \cite{FOT11} for a more detailed discussion. Given a metric measure space $(X,d,m)$, we say that $(\mathcal{E},\mathcal{F})$ is a \emph{regular, symmetric Dirichlet form} on $L^{2}(X,m)$, if it satisfies the following conditions:
\begin{enumerate}[label=\textup{({\roman*})},align=right,leftmargin=*,topsep=5pt,parsep=0pt,itemsep=2pt]
    \item $\mathcal{F}$ is a dense linear subspace of $L^{2}(X,m)$;
    \item $\mathcal{E}$ is a non-negative definite symmetric bilinear form on $\mathcal{F}\times\mathcal{F}$;
    \item $\mathcal{F}$ is a Hilbert space with inner product $\mathcal{E}_{1}:=\mathcal{E}+\langle\cdot,\cdot\rangle_{L^{2}(X,m)}$;
    \item for every $u\in \mathcal{F}$, $v:=(0\vee u)\wedge1\in\mathcal{F}$ and $\mathcal{E}(v, v)\leq\mathcal{E}(u, u)$;
    \item $\mathcal{F}\cap C_{c}(X)$ is dense both in $(\mathcal{F},\sqrt{\mathcal{E}_{1}})$ and in $(C_{c}(X),\lVert\cdot\rVert_{\sup})$.
\end{enumerate}If, in addition, $(\mathcal{E},\mathcal{F})$ satisfies
\begin{enumerate}
  \item[(vi)] $\mathcal{E}(f, g)=0$ for any $f, g\in\mathcal{F}$ with $\supp_{m}(f)$, $\supp_{m}(g)$ compact, and $\supp_{m}(f-a\one_{X})\cap \supp_{m}(g)=\emptyset$ for some $a\in\mathbb{R}$. Here, $\supp_{m}(f)$ denotes the support of the measure $|f|\dif m$.
\end{enumerate}
then we call $(\mathcal{E},\mathcal{F})$ a \emph{regular, strongly local and symmetric Dirichlet form}. Let $\{P_{t}:L^{2}(X,m)\to L^{2}(X,m)\}_{t>0}$ denote the semigroup corresponding to $(\mathcal{E},\mathcal{F})$. We say $(\mathcal{E},\mathcal{F})$ is \emph{conservative} if $P_{t}\one_{X}=\one_{X}$ for every $t>0$.

Given an open set $D\subset X$, we define another Dirichlet form $(\mathcal{E}^{D},\mathcal{F}^{D})$ by letting $\mathcal{F}^{D}$ be the closure of $\mathcal{F}\cap C_{c}(D)$ under the norm of $\mathcal{E}_{1}$, which ensures that $\mathcal{F}^{D}$ is a linear space of $\mathcal{F}$. For $u,v\in \mathcal{F}^{D}\subset \mathcal{F}$, we define $\mathcal{E}^{D}(u,v):=\mathcal{E}(u,v)$. It is known that $(\mathcal{E}^{D},\mathcal{F}^{D})$ is also a regular symmetric Dirichlet form on $L^{2}(D,\restr{m}{D})$ \cite[Theorem 3.3.9]{CF12}, and is called the \emph{part} of $(\mathcal{E},\mathcal{F})$ on $D$. Let $\{P^{D}_{t}\}$ be the semigroup corresponding to $(\mathcal{E}^{D},\mathcal{F}^{D})$ and let $p^{D}:(0,\infty)\times D\times D\to[0,\infty)$ be its heat kernel.

By \cite[Theorem 3.1]{BGK12}, if $(\mathcal{E},\mathcal{F})$ satisfies the heat kernel estimates stated in Definition \ref{d:HKE}, and if $(X,d,m)$ satisfies the volume doubling property, then the heat kernels of $(\mathcal{E},\mathcal{F})$ and of $(\mathcal{E}^{D},\mathcal{F}^{D})$ are continuous.

Returning to our setting, consider a sequence of metric measure spaces $(X_{n},d_{n},m_{n})$ and, for each $n\geq1$, a regular, strongly local symmetric Dirichlet form $(\mathcal{E}_{n},\mathcal{F}_{n})$ on $L^{2}(X_{n},m_{n})$. For each $R>0$, we write $(\mathcal{E}_{n}^{(R)},\mathcal{F}_{n}^{(R)}):=(\mathcal{E}_{n}^{B_{n}(p_{n},R)},\mathcal{F}_{n}^{B_{n}(p_{n},R)})$ as the part of $(\mathcal{E}_{n},\mathcal{F}_{n})$ on $B_{n}(p_{n},R)$ for brevity. The semigroup corresponding to $(\mathcal{E}_{n}^{(R)},\mathcal{F}_{n}^{(R)})$ is denoted by $\{P^{(n,R)}_{t}\}_{t>0}$, and its heat kernel by $p^{(n,R)}(t,x,y):=p^{B_{n}(p_{n},R)}(t,x,y)$ for $t>0$ and $x,y\in B_{n}(p_{n},R)$. For convenience, we set $p^{(n,R)}(t,x,y)=0$ if $x$ or $y$ is outside $B_{n}(p_{n},R)$.

By \cite[Theorem 2.1.4]{Dav89}, the generator of $(\mathcal{E}_{n}^{(R)},\mathcal{F}_{n}^{(R)})$ has a discrete spectrum $\{\lambda_{n,j}^{(R)}\}_{j=1}^{\infty}$ (written in increasing order and repeated according to multiplicity). Let $\{\phi^{(R)}_{n,j}\}_{j=1}^{\infty}$ be the corresponding eigenfunction normalized by $\lVert\phi^{(R)}_{n,j}\rVert_{L^{2}(B_{n}(p_{n},R),m_{n})}=1$. Moreover, the following Hilbert--Schmidt expansion holds:
\begin{equation}\label{e.hs(n,r)}
    p^{(n,R)}(t,x,y)=\sum_{j=1}^{\infty}\exp(-\lambda_{n,j}^{(R)}t)\phi^{(R)}_{n,j}(x)\phi^{(R)}_{n,j}(y)\quad\text{for all }t>0\text{ and }x,y\in B_{n}(p_{n},R),
\end{equation}
where this series converges uniformly on $[\delta,\infty)\times B_{n}(p_{n},R)\times B_{n}(p_{n},R)$ for any $\delta>0$ and $n\geq1$. We first estimate the growth of eigenvalues, which requires the following \emph{capacity condition}:
\begin{lemma}\label{l.cap}
    There exist constants $\kappa_{\mathrm{cap}}\in (0,1)$ and $C_{\mathrm{cap}}>1$ such that, for any $n\geq1$ and any ball $B_{n}=B_{n}(x,r)$ in $(X_{n},d_{n})$ of radius $r>0$, there exists a continuous cut-off function $\phi\in\mathcal{F}_{n}$ with $0\leq\phi\leq1$, $\restr{\phi}{\kappa_{\mathrm{cap}} B_{n}}=1$, $\supp(\phi)\subset  B_{n}$, and \begin{equation}\label{e.cap<}
    \mathcal{E}_{n}(\phi,\phi)\leq C_{\mathrm{cap}}\frac{V_{n}(r)}{\Psi_{n}(r)}.
    \end{equation}
\end{lemma}
\begin{proof}
As the proof of \cite[Theorem 1.2]{GHL15} is quantitative with the constant, we see that the capacity condition $(\text{cap})_{\Psi}$ in \cite{GHL15} holds for all $(\mathcal{E}_{n},\mathcal{F}_{n})$ with the same constant $\kappa$ and $C$, which implies \eqref{e.cap<}.
\end{proof}
\begin{proposition}\label{p.weyl}
There exists a constant $C_{1}\geq1$, independent of $n$ and $R$, such that 

\begin{equation}\label{e.weyl}
\lambda_{n,j}^{(R)}\leq C_{1}\Psi_{n}\left(\frac{R\wedge \diam(X_{n},d_{n})}{j}\right)^{-1},\ {\forall R\in(0,\infty)},\ 1\leq n,j<\infty.
\end{equation}Therefore \[\frac{R\wedge \diam(X_{n},d_{n})}{j}\lesssim\Psi_{n}^{-1}((\lambda_{n,j}^{(R)})^{-1}), \ {\forall R\in(0,\infty)},\ 1\leq n,j<\infty.\]
\end{proposition}
\begin{proof}
We use the max-min principle \cite[Exercise 12.4.2]{Sch12}.

Let $R\in(0,\frac{3}{2}\diam(X_{n},d_{n}))$. Given an integer $j$ sufficiently large, we can find a disjoint family of $j$ balls $B_{n}(x_{i},r)\subset B_{n}(p_{n},R)$, where $r=R/(100j)$. In fact, since $\frac{R}{3}<\frac{1}{2}\diam(X_{n},d_{n})$ and $\frac{1}{2}\diam(X_{n},d_{n}) \leq\sup_{q \in X_{n}}d_{n}(p_{n},q)\leq \diam(X_{n},d_{n})$ by triangle inequality, there exists a point $q\in X_{n}$ such that $d_{n}(p_{n},q)\geq R/3$. Choose one of the geodesics that connect $p_{n}$ and $q$. We can pick points along this geodesic such that each two points have a distance greater than $R/(50j)$, then choose the balls centered at these points with radius $r=R/(200j)$. Applying Lemma \ref{l.cap} on these balls, there exists $\phi_{i}$, $1\leq i\leq j$ such that \[\mathcal{E}_{n}^{(R)}(\phi_{i},\phi_{i})\lesssim\frac{V_{n}(r)}{\Psi_{n}(r)}.\]
Thus by the variational characterization of the $j$-th eigenvalue $\lambda_{n,j}^{(R)}$, we see that  such that for $\phi=\sum_{i=1}^{j}a_{i}\phi_{i}\in\mathrm{span}(\phi_{i},1\leq i\leq j)$ where $(a_{1},\ldots,a_{j})\in\mathbb{R}^{j}\setminus\{0\}$,
\[\begin{aligned}
\frac{\mathcal{E}^{(R)}_{n}(\phi,\phi)}{\lVert \phi\rVert_{L^{2}(B_{n}(p_{n},R))}^{2}}{=}\frac{\sum_{i=1}^{j}a_{i}^{2}\mathcal{E}^{(R)}_{n}(\phi_{i},\phi_{i})}{\sum_{i=1}^{j}a_{i}^{2}m_{n}(B_{n}(x_{i},\kappa_{\mathrm{cap}} r))}\lesssim\frac{V_{n}(r)/\Psi_{n}(r)}{V_{n}(\kappa_{\mathrm{cap}}r)}\lesssim \frac{1}{\Psi_{n}(R/j)}
\end{aligned}\]
where the locality of the Dirichlet form is used in the first equality. Taking supremum over all $(a_{1},\ldots,a_{j})\in\mathbb{R}^{j}\setminus\{0\}$, we see by the max-min principle that \[\lambda_{n,j}^{(R)}=\inf_{H\subset \mathcal{F}_{n}: \dim H=j}\sup_{\phi\in H}\frac{\mathcal{E}^{(R)}_{n}(\phi,\phi)}{\lVert \phi\rVert_{L^{2}(B_{n}(p_{n},R))}^{2}}\leq \sup_{\phi\in\mathrm{span}(\phi_{i},1\leq i\leq j)}\frac{\mathcal{E}^{(R)}_{n}(\phi,\phi)}{\lVert \phi\rVert_{L^{2}(B_{n}(p_{n},R))}^{2}}\lesssim\frac{1}{\Psi_{n}(R/j)}.\]

If $\diam(X_{n},d_{n})<\infty$ and if $R\in[\frac{3}{2}\diam(X_{n},d_{n}),\infty)$, then $B_{n}(p_{n},R)=X_{n}=B_{n}(p_{n},\frac{5}{4}\diam(X_{n},d_{n}))$, therefore $\lambda_{n,j}^{(R)}=\lambda_{n,j}^{(\frac{5}{4}\diam(X_{n},d_{n}))}\lesssim \Psi_{n}(\diam(X_{n},d_{n})/j)^{-1}$ by the discussion in the above paragraph. Therefore we conclude that \eqref{e.weyl} holds for all $R\in(0,\infty)$.
\end{proof}
\begin{remark}
If the metric space $(X_{n},d_{n})$ is bounded, that is, $\diam(X_{n},d_{n})<\infty$, and if $R>\diam(X_{n},d_{n})$, then $\lambda_{n,1}^{(R)}=0$, so the lower bound of $\lambda_{n,j}^{(R)}$ in \eqref{e.weyl} cannot be expected for $R$ very large. On the other hand, by the Faber--Krahn inequality (see \cite{GHL15} for a definition), there exists $C_{2},C_{3}>1$ such that $\lambda_{n,1}^{(R)}\geq {C_{3}^{-1}(\Psi_{n}(R))^{-1}}$ for all $0<R<C^{-1}_{2}\diam(X_{n},d_{n})$.
\end{remark}

\begin{remark}
If $\Psi(r)=r^{\beta}$ for some $\beta>1$ and if the space is $\alpha$-regular, namely $V(r)\asymp r^{\alpha}$, then \eqref{e.weyl} gives $\lambda_{j}\lesssim j^{\beta}$, which is weaker than the classic Weyl's law $\lambda_{j}\lesssim j^{\beta/\alpha}$ as $\alpha\geq1$ and $2\leq \beta\leq\alpha+1$ whenever the full heat kernel estimate $\hyperlink{HKEf}{\mathrm{HKE}_{\mathrm{f}}}$ holds, see \cite[Theorem 3.20]{Bar98}, or \cite[Theorem 4.8 and Corollary 3.5]{GHL03}, or a recent result \cite[Theorem 2.1]{Mur25}.
\end{remark}

We also have the following $L^{\infty}$ estimates of eigenfunctions.
\begin{proposition}
   There exists a constant $C>0$, such that for each $n,j\geq1$,
    \begin{equation}\label{e.eigfunbd}
      \left\lVert \phi_{n,j}^{(R)}\right\rVert_{L^{\infty}\left(B_{n}(p_{n},R)\right)}\leq C\inf_{t\in(0,\infty)}\frac{e^{\lambda_{n,j}^{(R)}t}}{V_{n}(\Psi_{n}^{-1}(t)\wedge \diam(X_{n},d_{n}))^{1/2}}.
    \end{equation}
\end{proposition}
\begin{proof}
    Since $\phi_{n,j}^{(R)}$ is also the eigenfunction of the bounded operator $P^{(n,R)}_{t}$ with eigenvalue $\exp(-\lambda_{n,j}^{(R)}t)$, we have by the Cauchy--Schwarz inequality that \begin{align*}
        \left|\phi_{n,j}^{(R)}(x)\right|&=\left|e^{\lambda_{n,j}^{(R)}t}P^{(n,R)}_{t}\phi_{n,j}^{(R)}(x)\right|=\left|e^{\lambda_{n,j}^{(R)}t}\int_{B_{n}(p_{n},R)}p^{(n,R)}(t,x,y)\phi_{n,j}^{(R)}(y)\dif m_{n}(y)\right|\\&\leq {e^{\lambda_{n,j}^{(R)}t}} p^{(n,R)}(2t,x,x)^{1/2}\left\lVert \phi_{n,j}^{(R)}\right\rVert_{L^{2}\left(B_{n}(p_{n},R)\right)}\\
        &\lesssim \frac{e^{\lambda_{n,j}^{(R)}t}}{V_{n}(\Psi_{n}^{-1}(2t)\wedge \diam(X_{n},d_{n}))^{1/2}}\lesssim \frac{e^{\lambda_{n,j}^{(R)}t}}{V_{n}(\Psi_{n}^{-1}(t){\wedge \diam(X_{n},d_{n})})^{1/2}}.
    \end{align*}
By taking the infimum over $t\in(0,\infty)$ we obtain \eqref{e.eigfunbd}.
\end{proof}

We need the H\"older estimates of the heat kernel and eigenfunctions to pass to the limit. To formulate it, we first define a metric on $X_{n}\times X_{n}$, with some abuse of notation, by \[d_{n}\left((x_{1},y_{1}),(x_{2},y_{2})\right):=\max(d_{n}(x_{1},x_{2}),d_{n}(y_{1},y_{2})).\] Under this metric, the ball in $X_{\infty}\times X_{\infty}$ centered at $(x,y)$ with radius $r$ equals $B_{\infty}(x,r)\times B_{\infty}(y,r)$, which is totally bounded.
\begin{proposition}\label{p.holder}
    For all $n\geq1$, $j\geq1$, $R>0$ and $t>0$, the Dirichlet heat kernels $p^{(n,R)}(t,\cdot,\cdot): B_{n}(p_{n},R)\times B_{n}(p_{n},R)\rightarrow[0,\infty)$ and $\phi^{(R)}_{n,j}:B_{n}(p_{n},R)\rightarrow\mathbb{R}$ are H\"older continuous. To be precise, there exist constants $C$ and $\Theta$, such that
\begin{enumerate}[label=\textup{(\arabic*)},align=right,leftmargin=*,topsep=5pt,parsep=0pt,itemsep=2pt]        \item\label{it.phder1} for all $n\geq1$ and $x_{1},x_{2},y_{1},y_{2}\in B _{n}(p_{n},R)$,
    \begin{equation}\label{e.hold1}
    \left|p^{(n,R)}(t,x_{1},y_{1})-p^{(n,R)}(t,x_{2},y_{2})\right|\leq \frac{Cd_{n}((x_{1},y_{1}),(x_{2},y_{2}))^{\Theta}}{(\Psi_{n}^{-1}(t){\wedge \diam(X_{n},d_{n})})^{\Theta}V_{n}(\Psi_{n}^{-1}(t){\wedge \diam(X_{n},d_{n})})}.
    \end{equation}
    \item\label{it.phder2} for all $n\geq1$ and $x,y\in B _{n}(p_{n},R)$,
     \begin{equation}\label{e.hold2}
        \left|\phi^{(R)}_{n,j}(x)-\phi^{(R)}_{n,j}(y)\right|\leq C\frac{V_{n}({R\wedge \diam(X_{n},d_{n})})^{1/2}}{V_{n}({(R\wedge \diam(X_{n},d_{n}))}/j)}\left({\frac{j}{{R\wedge \diam(X_{n},d_{n})}}}\right)^{\Theta}d_{n}(x,y)^{\Theta}.
    \end{equation}
    \end{enumerate}
\end{proposition}
\begin{proof}
\begin{enumerate}[label=\textup{(\arabic*)},align=right,leftmargin=*,topsep=5pt,parsep=0pt,itemsep=2pt]
    \item[\ref{it.phder1}] Under the heat kernel estimate $\hyperlink{HKE}{\mathrm{HKE}(\Psi)}$ in Definition~\ref{d:HKE}, the \emph{elliptic Harnack inequality} $\mathrm{(H)}$ follows from the proof of \cite[Theorem~7.4]{GT12}, noting that the unboundedness of the space is not used in deriving $\mathrm{(H)}$. Moreover, the \emph{mean exit time estimate} $\mathrm{(E_{\Psi})}$ (see \cite[Definition~3.10]{GT12}) holds for all $r \in (0, \diam(X,d)/A)$, where $A>1$ depends only on the constants appearing in $\hyperlink{HKE}{\mathrm{HKE}(\Psi)}$. This can be proved by the same argument as in \cite[Theorem~9.4]{GHL15} (see also \cite[p.~725]{Lie15}, \cite[Remark~4.16]{GHH23} and \cite[Lemma~12.2]{GHH24} for the proof of the mean exit time upper bound). According to the proof of \cite[Lemma 5.10]{GT12} (see also \cite[Theorem 3.1 and Corollary 4.2]{BGK12}), there exists a constant $\Theta$ that depends only on the constants appearing in the two-sided heat kernel estimates, such that for all $(x_{1},y_{1})$ and $(x_{2},y_{2})$ belong to $B_{n}(p_{n},R)\times B_{n}(p_{n},R)$, if $d_{n}((x_{1},y_{1}),(x_{2},y_{2}))\leq \Psi_{n}^{-1}(t)$, there holds
\[\begin{aligned}
& \left|p^{(n,R)}(t,x_{1},y_{1})-p^{(n,R)}(t,x_{2},y_{2})\right|\\ \leq\ &\left|p^{(n,R)}(t,x_{1},y_{1})-p^{(n,R)}(t,x_{1},y_{2})\right|+\left|p^{(n,R)}(t,x_{1},y_{2})-p^{(n,R)}(t,x_{2},y_{2})\right|\\
\lesssim\ &\frac{1}{(\Psi_{n}^{-1}(t){\wedge \diam(X_{n},d_{n})})^{\Theta}V_{n}(\Psi_{n}^{-1}(t){\wedge \diam(X_{n},d_{n})})}\left(d_{n}(x_{1},x_{2})^{\Theta}+d_{n}(y_{1},y_{2})^{\Theta}\right)\\
\lesssim\ &\frac{1}{(\Psi_{n}^{-1}(t){\wedge \diam(X_{n},d_{n})})^{\Theta}V_{n}(\Psi_{n}^{-1}(t){\wedge \diam(X_{n},d_{n})})} d_{n}((x_{1},y_{1}),(x_{2},y_{2}))^{\Theta}.
\end{aligned}\]
We also have, by the on-diagonal estimates of heat kernel, that if $d_{n}((x_{1},y_{1}),(x_{2},y_{2}))\geq\Psi_{n}^{-1}(t)$, then $\Psi_{n}^{-1}(t)\leq \diam(X_{n},d_{n})$ and \begin{align}
\left|p^{(n,R)}(t,x_{1},y_{1})-p^{(n,R)}(t,x_{2},y_{2})\right|&\lesssim\frac{1}{V_{n}(\Psi_{n}^{-1}(t))}=\frac{1}{(\Psi_{n}^{-1}(t))^{\Theta}V_{n}(\Psi_{n}^{-1}(t))}\cdot (\Psi_{n}^{-1}(t))^{\Theta} \\ & \lesssim \frac{1}{(\Psi_{n}^{-1}(t))^{\Theta}V_{n}(\Psi_{n}^{-1}(t))} d_{n}((x_{1},y_{1}),(x_{2},y_{2}))^{\Theta}.
\end{align}
Combining these two inequalities, we see that for all $n\geq1$ and $x_{1},x_{2},y_{1},y_{2}\in B _{n}(p_{n},R)$, the estimate \eqref{e.hold1} holds.
\item[\ref{it.phder2}] Using \eqref{e.hold1} and the fact that $\phi_{n,j}^{(R)}$ is also a eigenfunction of the operator $P^{(n,R)}_{t}$ with eigenvalue $\exp(-\lambda_{n,j}^{(R)}t)$, we have that for all $t>0$,
\begin{align*}
    &\phantom{\ \leq}\left|\phi_{n,j}^{(R)}(x)-\phi_{n,j}^{(R)}(y)\right|\\
    &=\left|e^{\lambda_{n,j}^{(R)}t}\int_{B_{n}(p_{n},R)}\left(p^{(n,R)}(t,x,z)-p^{(n,R)}(t,y,z)\right)\phi_{n,j}^{(R)}(z)\dif m_{n}(z)\right|\\
    &\lesssim  \frac{e^{\lambda_{n,j}^{(R)}t}d_{n}(x,y)^{\Theta}m_{n}(B_{n}(p_{n},R))^{1/2}}{(\Psi_{n}^{-1}(t){\wedge \diam(X_{n},d_{n})})^{\Theta}V_{n}(\Psi_{n}^{-1}(t){\wedge \diam(X_{n},d_{n})})}\left\lVert \phi_{n,j}^{(R)}\right\rVert_{L^{2}\left(B_{n}(p_{n},R)\right)}\\ &\lesssim\frac{e^{\lambda_{n,j}^{(R)}t}d_{n}(x,y)^{\Theta} }{(\Psi_{n}^{-1}(t){\wedge \diam(X_{n},d_{n})})^{\Theta}V_{n}(\Psi_{n}^{-1}(t){\wedge \diam(X_{n},d_{n})})} V_{n}({R\wedge \diam(X_{n},d_{n})})^{1/2}\\
   &\overset{\eqref{e.weyl}}{\lesssim}\frac{{\exp({C\Psi_{n}((R\wedge \diam(X_{n},d_{n}))/j)^{-1}t})}d_{n}(x,y)^{\Theta}}{(\Psi_{n}^{-1}(t){\wedge \diam(X_{n},d_{n})})^{\Theta}V_{n}(\Psi_{n}^{-1}(t){\wedge \diam(X_{n},d_{n})})}  V_{n}({R\wedge \diam(X_{n},d_{n})})^{1/2}
\end{align*}
By choosing $t={\Psi_{n}((R\wedge \diam(X_{n},d_{n}))/j)}$, we complete the proof of \eqref{e.hold2}.
\end{enumerate}

\end{proof}

\section{Heat kernel on limit space}\label{s.heatlim}
In this section, we construct the heat kernel on the balls of limit spaces. We first need the following Arzel\`{a}--Ascoli type theorem for (possibly) discontinuous functions.
\begin{lemma}[{\cite[Lemma D.1]{Kig23}}]\label{l.a-a}
    Let $(X,d)$ be a totally bounded metric space. Let $u_{n}: X \rightarrow \mathbb{R}$ for any $n\in\mathbb{N}$. Assume that there exists a nondecreasing function $\eta:[0, \infty) \rightarrow [0,\infty)$ and a sequence $\{\delta_{n}\}_{n\in\mathbb{N}}$ of non-negative numbers such that $\lim_{t\downarrow0} \eta(t) = 0$, $\lim_{n\rightarrow\infty} \delta_{n} = 0$, $\sup_{n\in\mathbb{N},x\in X} |u_{n}(x)| < \infty$ and \[|u_{n}(x)-u_{n}(y)|\leq \eta(d(x,y))+\delta_{n}\quad\text{for all }x,y\in X\text{ and }n\in\mathbb{N}.\]
    Then there exists a subsequence $\{u_{n_{k}}\}_{k\in\mathbb{N}}$ and $u \in C(X)$ with \[|u(x)-u(y)|\leq\eta(d(x,y))\quad\text{for all }x,y\in X,\]
    such that $\sup_{x\in X}|u_{n_{k}}(x)-u(x)|\rightarrow0$ as $k\rightarrow\infty$.
\end{lemma}
\begin{definition}
    For any fixed $n\geq1$, $R>0$, define the family of approximate heat kernels as functions on $(0,\infty)\times B_{\infty}(p_{\infty},R)\times B_{\infty}(p_{\infty},R)$ by \[q^{(n,R)}(t,x,y):=p^{(n,R)}(t,g_{n}(x),g_{n}(y)), \quad x,y\in B_{\infty}(p_{\infty},R),\ t>0.\]
    In the same way, we define the family of approximate $j$-th eigenfunctions as functions on $B_{\infty}(p_{\infty},R)$ by \[\varphi_{n,j}^{(R)}(x):=\phi_{n,j}^{(R)}(g_{n}(x)),\quad x\in B_{\infty}(p_{\infty},R).\]
\end{definition}

\begin{proposition}\label{p.rational}
Fix $R>0$. There exist positive numbers $\{\lambda^{(R)}_{j}\}_{j=1}^{\infty}$, H\"older continuous functions $\{\varphi^{(R)}_{j}\}_{j=1}^{\infty}$ on $B_{\infty}(p_{\infty},R)$, H\"older continuous functions $\{\widetilde{q}^{(R)}(t,\cdot,\cdot)\}_{t\in\mathbb{Q}_{+}}$ on $B_{\infty}(p_{\infty},R)\times B_{\infty}(p_{\infty},R)$, and a common subsequence $\{n_{k}\}_{k\in\mathbb{N}}$ such that 
\begin{enumerate}[label=\textup{(\arabic*)},align=right,leftmargin=*,topsep=5pt,parsep=0pt,itemsep=2pt]
    \item\label{lb.ra1} for each $t\in\mathbb{Q}_{+}$, \begin{equation}\label{e.ration1}
    \begin{aligned}
    	&\sup_{x,y\in B_{\infty}(p_{\infty},R)}|\widetilde{q}^{(R)}(t,x,y)-q^{(n_{k},R)}(t,x,y)|\rightarrow0\text{ as }k\rightarrow\infty,\\
    	&{\left|\widetilde{q}^{(R)}(t,x_{1},y_{1})-\widetilde{q}^{(R)}(t,x_{2},y_{2})\right|\lesssim  \frac{1}{(t^{1/\beta^{\prime}}{\wedge \ell})^{\Theta}V_{l}(t^{1/\beta^{\prime}}{\wedge \ell})}d_{\infty}((x_{1},y_{1}),(x_{2},y_{2}))^{\Theta}.}
    \end{aligned}
         \end{equation}
        { where $\Theta$ is the constant appearing in Proposition \ref{p.holder} and $\ell$ is defined in \eqref{e.infdiam}.}
    \item \label{lb.ra2} for each $j\geq1$, \begin{equation}\label{e.ration2}
        |\lambda^{(R)}_{n_{k},j}-\lambda^{(R)}_{j}|\rightarrow0\text{ as }k\rightarrow\infty\text{ and }\lambda^{(R)}_{j}\lesssim {\Psi_{\infty}\left(\frac{R\wedge \diam(X_{\infty},d_{\infty})}{j}\right)^{-1}}.
    \end{equation}
    \item\label{lb.ra3} for each $j\geq1$, 
   \begin{equation}\label{e.ration3} \begin{aligned}
        \sup_{x\in B_{\infty}(p_{\infty},R)}|\varphi_{j}^{(R)}(x)-\varphi^{(R)}_{n_{k},j}(x)|&\rightarrow0\text{ as }k\rightarrow\infty\\ \left\lVert\varphi_{j}^{(R)}\right\rVert_{L^{\infty}\left(B_{\infty}(p_{\infty},R)\right)}&\lesssim \inf_{t\in(0,\infty)}\frac{e^{\lambda_{j}^{(R)}t}}{V_{\infty}(\Psi_{\infty}^{-1}(t){\wedge\diam(X_{\infty},d_{\infty}) })^{1/2}}.
    \end{aligned}
    \end{equation}
    \item\label{lb.ra4} $\text{for all }t\in\mathbb{Q}_{+},\ x,y\in B_{\infty}(p_{\infty},R)$ and $\delta>0$, 
    \begin{equation}
   \begin{aligned}\label{e.ration4}
      & \text{the series } \widetilde{q}^{(R)}(t,x,y)=\sum_{j=1}^{\infty}\exp(-\lambda_{j}^{(R)}t)\varphi_{j}^{(R)}(x)\varphi_{j}^{(R)}(y)\quad\\ 
       &\text{converges uniformly on $\left(\mathbb{Q}\cap [\delta,\infty)\right)\times B_{\infty}(p_{\infty},R)\times B_{\infty}(p_{\infty},R)$}
    \end{aligned}
    \end{equation}
\end{enumerate}
\end{proposition}
\begin{proof}
\begin{enumerate}[label=\textup{(\arabic*)},align=right,leftmargin=*,topsep=5pt,parsep=0pt,itemsep=2pt]
    \item[\ref{lb.ra1}] We first construct $\widetilde{q}^{(R)}$. By Proposition \ref{p.holder}, for each fixed $t>0$, {since $\diam(X_{n},d_{n})\to \diam(X_{\infty},d_{\infty})$ by Lemma \ref{l.diam} and $\ell:=\inf_{1\leq n<\infty}\diam(X_{n},d_{n})>0$}, we have
    \begin{align}
       &\phantom{\leq\ } \left|q^{(n,R)}(t,x_{1},y_{1})-q^{(n,R)}(t,x_{2},y_{2})\right| \\ &=\left|p^{(n,R)}(t,g_{n}(x_{1}),g_{n}(y_{1}))-p^{(n,R)}(t,g_{n}(x_{2}),g_{n}(y_{2})) \right|\\ 
       &\lesssim \frac{d_{n}((g_{n}(x_{1}),g_{n}(y_{1})),(g_{n}(x_{2}),g_{n}(y_{2})))^{\Theta}}{(\Psi_{n}^{-1}(t){\wedge \diam(X_{n},d_{n})})^{\Theta}V_{n}(\Psi_{n}^{-1}(t){\wedge \diam(X_{n},d_{n})})}\\
        &\lesssim \frac{1}{(\Psi_{n}^{-1}(t){\wedge \ell})^{\Theta}V_{n}(\Psi_{n}^{-1}(t){\wedge \ell})}\left(d_{\infty}((x_{1},y_{1}),(x_{2},y_{2}))+2\epsilon_{n} \right)^{\Theta}\\ 
        & \lesssim \frac{1}{(t^{1/\beta^{\prime}}{\wedge \ell})^{\Theta}V_{l}(t^{1/\beta^{\prime}}{\wedge \ell})}\left(d_{\infty}((x_{1},y_{1}),(x_{2},y_{2}))^{\Theta}+\epsilon^{\Theta}_{n}\right),\label{e.raeq1}
    \end{align}
combine which and the uniform heat kernel bounds, we can apply Lemma \ref{l.a-a} with \[\eta(s)= \frac{C}{(t^{1/\beta^{\prime}}{\wedge \ell})^{\Theta}V_{l}(t^{1/\beta^{\prime}}{\wedge \ell})}s^{\Theta}\text{ and }\delta_{n}= \frac{C}{(t^{1/\beta^{\prime}}{\wedge \ell})^{\Theta}V_{l}(t^{1/\beta^{\prime}}{\wedge \ell})}\epsilon^{\Theta}_{n},\] and conclude that there is a {H\"{o}lder} continuous function, denoted by $\widetilde{q}^{(R)}(t,x,y)$, defined on $B_{\infty}(p_{\infty},R)\times B_{\infty}(p_{\infty},R)$ such that, along a subsequence (which may depend on $t$ and $R$), \[\sup_{x,y\in B_{\infty}(p_{\infty},R)}|\widetilde{q}^{(R)}(t,x,y)-p^{(n,R)}(t,g_{n}(x),g_{n}(y))|\rightarrow0\text{ as }n\rightarrow\infty.\]By a standard diagonal argument, we can find a common subsequence such that the convergence holds for all $t\in\mathbb{Q}_{+}$ along that subsequence.
    \item[\ref{lb.ra2}] By Proposition \ref{p.weyl}, the sequence $\{\lambda^{(R)}_{n,j}\}_{n=1}^{\infty}$ is bounded for each $j$, we can find a subsequence and a $\lambda^{(R)}_{j}$ such that $\lambda^{(R)}_{n,j}\rightarrow \lambda^{(R)}_{j}$ as $n\rightarrow\infty$, which have the given upper bound, {using the convergence $\Psi_{n}\to \Psi_{\infty}$ and the convergence of diameters in Lemma \ref{l.diam}}. 
    \item[\ref{lb.ra3}] By \eqref{e.eigfunbd}, \eqref{e.hold2} and using the same argument as \ref{lb.ra1}, we can apply Lemma \ref{l.a-a} and find continuous functions $\varphi^{(R)}_{j}$ on $B_{\infty}(p_{\infty},R)$ such that along some subsequence (which may depend on $j$), \[\sup_{x\in B_{\infty}(p_{\infty},R)}|\varphi_{j}^{(R)}(x)-\phi^{(R)}_{n,j}(g_{n}(x))|\rightarrow0\text{ as }n\rightarrow\infty.\] By a standard diagonal argument again, we can find a common subsequence $\{n_{k}\}$ such that the convergence in  \eqref{e.ration1}, \eqref{e.ration2} and \eqref{e.ration3} holds along this subsequence for all $j\geq1$ and all $t\in\mathbb{Q}_{+}$.
    \item[\ref{lb.ra4}] From \eqref{e.hs(n,r)} we know that for $t\in\mathbb{Q}_{+}$,
    \begin{equation*}
       {q}^{(n_{k},R)}(t,x,y)=\sum_{j=1}^{\infty}\exp(-\lambda_{n_{k},j}^{(R)}t)\varphi_{n_{k},j}^{(R)}(x)\varphi_{n_{k},j}^{(R)}(y)\quad\text{for all }x,y\in B_{\infty}(p_{\infty},R).
    \end{equation*} By \eqref{e.ration1}, \eqref{e.ration2}, \eqref{e.ration3} and the uniform convergence of \eqref{e.hs(n,r)}, we can take the limit on both side and conclude that the series in \eqref{e.ration4} convergent. Fix a small $\delta\in(0,\Psi_{\infty}(\ell))$\footnote{{We adopt the convention $\Psi_{\infty}(\infty):=\infty$.}}. Let us show that the series \eqref{e.ration4} converges uniformly on $\left(\mathbb{Q}\cap [\delta,\infty)\right)\times B_{\infty}(p_{\infty},R)\times B_{\infty}(p_{\infty},R)$. In fact, by \eqref{e.ration3}, \begin{equation}\label{e.unif-1}
    \left\lVert\varphi_{j}^{(R)}\right\rVert_{L^{\infty}(B_{\infty}(p_{\infty},R))}\lesssim \frac{\exp({\lambda_{j}^{(R)}\delta/4})}{V_{\infty}(\Psi_{\infty}^{-1}(\delta/4))^{1/2}}.
    \end{equation}  
    Therefore for $\delta\leq t\in\mathbb{Q}_{+}$, 
    \begin{equation}\label{e.unif}
        \left|\exp(-\lambda_{j}^{(R)}t)\varphi_{j}^{(R)}(x)\varphi_{j}^{(R)}(y)\right|\lesssim \exp(-\lambda_{j}^{(R)}\delta)\frac{\exp({\lambda_{j}^{(R)}\delta/2})}{V_{\infty}(\Psi_{\infty}^{-1}(\delta/4))} \lesssim \frac{\exp(-{\lambda_{j}^{(R)}\delta/2})}{V_{\infty}(\Psi_{\infty}^{-1}(\delta/4))}.
    \end{equation}
By Mercer's theorem on trace \cite[Proposition 5.6.9]{Dav07}, \begin{align}
        \sum_{j=1}^{\infty}\exp(-\lambda_{j}^{(R)}\delta/2)&=\sum_{j=1}^{\infty}\left(\lim_{k\rightarrow\infty}\exp(-\lambda_{n_{k},j}^{(R)}\delta/2)\right)\\ &\leq \liminf_{k\rightarrow\infty}\sum_{j=1}^{\infty}\exp(-\lambda_{n_{k},j}^{(R)}\delta/2)\ \text{(Fatou's lemma)}\\
        &=\liminf_{k\rightarrow\infty}\tr(P^{(n_{k},R)}_{\delta/2})\lesssim \frac{1}{V_{l}(\delta^{1/\beta^{\prime}})}<\infty.\label{e.Mercer}
    \end{align}
Therefore the Weierstrass M-test is applicable to show the uniform convergence in \eqref{e.ration4}.
\end{enumerate}
\end{proof}
Now we extend the convergence in \eqref{e.ration1} and the expansion in \eqref{e.ration4} from $t\in\mathbb{Q}_{+}$ to all $t>0$.
\begin{proposition}\label{p.hs-r}
Fix $R>0$. 
\begin{enumerate}[label=\textup{({\arabic*})},align=right,leftmargin=*,topsep=5pt,parsep=0pt,itemsep=2pt]
\item\label{lb.hsr1} For $x,y\in B_{\infty}(p_{\infty},R)$ and $t>0$, the limit $\lim_{\mathbb{Q}_{+}\ni s\rightarrow t}\widetilde{q}^{(R)}(s,x,y)$ exists and
 \begin{equation}\label{e.hs-r}
       {q}^{(R)}(t,x,y):=\lim_{\mathbb{Q}_{+}\ni s\rightarrow t}\widetilde{q}^{(R)}(s,x,y)=\sum_{j=1}^{\infty}\exp(-\lambda_{j}^{(R)}t)\varphi_{j}^{(R)}(x)\varphi_{j}^{(R)}(y),
       \end{equation}
       and the series in \eqref{e.hs-r} convergent uniformly on $(t,x,y)\in [\delta,\infty)\times B_{\infty}(p_{\infty},R)\times B_{\infty}(p_{\infty},R)$ for any $\delta>0$.
           \item\label{lb.hsr2} The function $q^{(R)}(\cdot,\cdot,\cdot):(0,\infty)\times B_{\infty}(p_{\infty},R)\times B_{\infty}(p_{\infty},R)\to \mathbb{R}$ is continuous. For each $t>0$, $q^{(R)}(t,\cdot,\cdot)$ is H\"older continuous on $B_{\infty}(p_{\infty},R)\times B_{\infty}(p_{\infty},R)$.
\item\label{lb.hsr3} Along the same subsequence in Proposition \ref{p.rational}, there holds \begin{equation}\label{e.real}
        \sup_{x,y\in B_{\infty}(p_{\infty},R)}|{q}^{(R)}(t,x,y)-q^{(n_{k},R)}(t,x,y)|\rightarrow0\text{ as }k\rightarrow\infty\quad \text{for all }t>0,
    \end{equation}
\end{enumerate}
\end{proposition}
\begin{proof}
\begin{enumerate}[label=\textup{({\arabic*})},align=right,leftmargin=*,topsep=5pt,parsep=0pt,itemsep=2pt]
\item[\ref{lb.hsr1}]   Let $\delta>0$ and $(t,x,y)\in [\delta,\infty)\times B_{\infty}(p_{\infty},R)\times B_{\infty}(p_{\infty},R)$. By \eqref{e.unif}, {for $s\in[\delta/2,\infty)\cap \mathbb{Q}$,} the series $\sum_{j=1}^{\infty}\exp(-\lambda_{j}^{(R)}s)\varphi_{j}^{(R)}(x)\varphi_{j}^{(R)}(y)$ is dominated by the convergent series $C\sum_{j=1}^{\infty} \frac{\exp(-\lambda_{j}^{(R)}\delta/{4})}{V_{\infty}(\Psi_{\infty}^{-1}(\delta/{8}))}$. Thus we can apply dominated convergence theorem and conclude that \begin{align*}
        {q}^{(R)}(t,x,y)&:=\lim_{\mathbb{Q}_{+}\ni s\rightarrow t}\widetilde{q}^{(R)}(s,x,y)\\&= \lim_{\substack{\mathbb{Q}_{+}\ni s\rightarrow t\\ {s\in[\delta/2,\infty)\cap \mathbb{Q}}}}\sum_{j=1}^{\infty}\exp(-\lambda_{j}^{(R)}s)\varphi_{j}^{(R)}(x)\varphi_{j}^{(R)}(y)\\
        &=\sum_{j=1}^{\infty}\exp(-\lambda_{j}^{(R)}t)\varphi_{j}^{(R)}(x)\varphi_{j}^{(R)}(y).
    \end{align*}
which gives \eqref{e.hs-r}.
    \item[\ref{lb.hsr2}] The continuity of $q^{(R)}(\cdot,\cdot,\cdot)$ follows from the uniform convergent property of \eqref{e.hs-r}. The H\"older continuity of $q^{(R)}(t,\cdot,\cdot)$ follows from the H\"older continuity of $\widetilde{q}^{(R)}(t,\cdot,\cdot)$ in Proposition \ref{p.rational}.
\item[\ref{lb.hsr3}] Let $\delta\leq s\leq t$ and assume that $\delta\in(0,\Psi_{\infty}(\ell))$ is small. 
{\small
\begin{align}
&\phantom{=\ }\abs{\left(\exp(-\lambda_{j}^{(R)}s)-\exp(-\lambda_{j}^{(R)}t)\right)\varphi_{j}^{(R)}(x)\varphi_{j}^{(R)}(y)}\\
&=\left(\exp\left(-\lambda_{j}^{(R)}\left(s-\frac{1}{2}\delta\right)\right)-\exp\left(-\lambda_{j}^{(R)}\left(t-\frac{1}{2}\delta\right)\right)\right)\abs{\exp\left(-\lambda_{j}^{(R)}\delta/2\right)\varphi_{j}^{(R)}(x)\varphi_{j}^{(R)}(y)}\\
&\overset{\eqref{e.unif}}{\lesssim}\frac{\exp(-{\lambda_{j}^{(R)}\delta/4})}{V_{\infty}(\Psi_{\infty}^{-1}(\delta/8))}\left(\exp\left(-\lambda_{j}^{(R)}\left(s-\frac{1}{2}\delta\right)\right)-\exp\left(-\lambda_{j}^{(R)}\left(t-\frac{1}{2}\delta\right)\right)\right)\\
&= \frac{1}{V_{\infty}(\Psi_{\infty}^{-1}(\delta/8))}\left(\exp\left(-\lambda_{j}^{(R)}\left(s-\frac{1}{4}\delta\right)\right)-\exp\left(-\lambda_{j}^{(R)}\left(t-\frac{1}{4}\delta\right)\right)\right)\label{e.real2}
\end{align}
\normalsize}
By definition of $q^{(R)}$ and Mercer's theorem \cite[Proposition 5.6.9]{Dav07},
    \begin{align}
        &\phantom{=\ }\left|{q}^{(R)}(s,x,y)-{q}^{(R)}(t,x,y)\right|
        \\&=\left|\sum_{j=1}^{\infty}\left(\exp(-\lambda_{j}^{(R)}s)-\exp(-\lambda_{j}^{(R)}t)\right)\varphi_{j}^{(R)}(x)\varphi_{j}^{(R)}(y)\right|
        \\&\overset{\eqref{e.real2}}{\lesssim}\frac{1}{V_{\infty}(\Psi_{\infty}^{-1}(\delta/8))}\sum_{j=1}^{\infty}\left(\exp\left(-\lambda_{j}^{(R)}\left(s-\frac{1}{4}\delta\right)\right)-\exp\left(-\lambda_{j}^{(R)}\left(t-\frac{1}{4}\delta\right)\right)\right)
        \\&\leq \frac{1}{V_{\infty}(\Psi_{\infty}^{-1}(\delta/8))}\liminf_{k\to\infty}\sum_{j=1}^{\infty}\left(\exp\left(-\lambda_{n_{k},j}^{(R)}\left(s-\frac{1}{4}\delta\right)\right)-\exp\left(-\lambda_{n_{k},j}^{(R)}\left(t-\frac{1}{4}\delta\right)\right)\right)
        \\&=\liminf_{k\to\infty}\int_{B_{n_{k}}(p_{n_{k}},R)}\frac{\left(p^{(n_{k},R)}\left(s-\frac{1}{4}\delta,x,x\right)-p^{(n_{k},R)}\left(t-\frac{1}{4}\delta,x,x\right)\right)}{V_{\infty}(\Psi_{\infty}^{-1}(\delta/8))}\dif m_{n_{k}}(x)\label{e.real3}
    \end{align}
    By \cite[Corollary 5.7]{GT12},
    \begin{align}
   &\phantom{=\ } \abs{\int_{B_{n}(p_{n},R)}\left(p^{(n,R)}\left(s-\frac{1}{4}\delta,x,x\right)-p^{(n,R)}\left(t-\frac{1}{4}\delta,x,x\right)\right)\dif m_{n}(x)}\\
   &\leq \int_{B_{n}(p_{n},R)}\left(\int_{t-\frac{1}{4}\delta}^{s-\frac{1}{4}\delta}\abs{\frac{\partial}{\partial r}p^{(n,R)}\left(r,x,x\right)}\dif r\right)\dif m_{n}(x)\\
   &\lesssim \int_{B_{n}(p_{n},R)}\left(\int_{t-\frac{1}{4}\delta}^{s-\frac{1}{4}\delta}\frac{2}{r}\frac{1}{V_{n}(\Psi^{-1}_{n}(r))}\dif r\right)\dif m_{n}(x)\lesssim \frac{8}{3\delta}\frac{V_{n}(R)}{V_{n}(\Psi^{-1}_{n}(3\delta/4))}\abs{s-t}\label{e.real4}
    \end{align}
    
    Combine \eqref{e.real3} and \eqref{e.real4}, we see that 
    \begin{equation}\label{e.real5}
    \left|{q}^{(R)}(s,x,y)-{q}^{(R)}(t,x,y)\right|\lesssim \frac{1}{\delta}\frac{V_{\infty}(R)}{V_{\infty}(\Psi^{-1}_{\infty}(\delta/8))^{2}}\abs{s-t}.
    \end{equation}
    Similarly by \eqref{e.real4}, \begin{align}\label{e.real6}
        \left|q^{(n_{k},R)}(s,x,y)-q^{(n_{k},R)}(t,x,y)\right|\lesssim\frac{1}{\delta}\frac{V_{n_{k}}(R)}{V_{n_{k}}(\Psi^{-1}_{n_{k}}(\delta/8))^{2}}\abs{s-t}.
    \end{align}
    For each fixed $\delta>0$ and any $t>\delta>0$, if $t\in(0,\infty)\setminus\mathbb{Q}$, we choose $s\in(\delta,t)\cap\mathbb{Q}$, and by \eqref{e.real5} and \eqref{e.real6}, \begin{align*}
        &\left|{q}^{(R)}(t,x,y)-q^{(n_{k},R)}(t,x,y)\right|\\ \leq& \left|{q}^{(R)}(t,x,y)-{q}^{(R)}(s,x,y)\right|+\left|{q}^{(R)}(s,x,y)-q^{(n_{k},R)}(s,x,y)\right|\\
        &\qquad+\left|q^{(n_{k},R)}(s,x,y)-q^{(n_{k},R)}(t,x,y)\right|\\
        \lesssim& \frac{1}{\delta}\left(\frac{V_{n_{k}}(R)}{V_{n_{k}}(\Psi^{-1}_{n_{k}}(\delta/8))^{2}}+\frac{V_{\infty}(R)}{V_{\infty}(\Psi^{-1}_{\infty}(\delta/8))^{2}}\right) (t-s)+ \left|{q}^{(R)}(s,x,y)-q^{(n_{k},R)}(s,x,y)\right|,
    \end{align*}
    which implies 
    \begin{align}
    &\phantom{\leq\ }\lim_{k\to\infty}\sup_{x,y\in B_{\infty}(p_{\infty},R)}\left|{q}^{(R)}(t,x,y)-q^{(n_{k},R)}(t,x,y)\right|\\
    &\lesssim \frac{1}{\delta}\frac{V_{\infty}(R)}{V_{\infty}(\Psi^{-1}_{\infty}(\delta/8))^{2}} (t-s)+ \lim_{k\to\infty}\sup_{x,y\in B_{\infty}(p_{\infty},R)}\left|{q}^{(R)}(s,x,y)-q^{(n_{k},R)}(s,x,y)\right| \\&\overset{\eqref{e.ration1}}{\lesssim} \frac{1}{\delta}\frac{V_{\infty}(R)}{V_{\infty}(\Psi^{-1}_{\infty}(\delta/8))^{2}} (t-s)
    \end{align}
for any $s\in(\delta,t)\cap\mathbb{Q}$. Letting $\mathbb{Q}\ni s\to t$, we obtain \eqref{e.real}. 
    \end{enumerate}
\end{proof}
{\begin{remark}\label{r.Hold-r}
	For any bounded Borel measurable function $f\in L^{\infty}(B_{\infty}(p_{\infty},R),m_{\infty})$, we define \begin{equation}
		Q_{t}^{(R)}f(x):=\int_{B_{\infty}(p_{\infty},R)}q^{(R)}(t,x,y)f(y)\dif m_{\infty}(y),\ x\in B_{\infty}(p_{\infty},R). 
	\end{equation}
	Then by the H\"older continuity of $q^{(R)}(t,\cdot,\cdot)$ in Proposition \ref{p.hs-r}-\ref{lb.hsr2} and that \[m_{\infty}(B_{\infty}(p_{\infty},R))\lesssim V_{\infty}(R)<\infty,\] we see that $Q_{t}^{(R)}f$ is a H\"older continuous function on $B_{\infty}(p_{\infty},R)$. By the proof of Proposition \ref{p.conser} later, we have $Q_{t}^{(R)}\one_{B_{\infty}(p_{\infty},R)}\leq 1$ so $Q_{t}^{(R)}f$ is also bounded by $\norm{f}_{L^{\infty}(B_{\infty}(p_{\infty},R),m_{\infty})}$.
\end{remark}}
\section{Dirichlet form on limit space}\label{s.dflim}
In this section, we will show that the function $q^{(R)}$ given in Proposition \ref{p.hs-r} gives a heat kernel and a semigroup on $B_{\infty}(p_{\infty},R)$. By letting $R\to\infty$, we will obtain a heat kernel on $X_{\infty}$.
\begin{proposition}\label{p.orthn}
    For each $R>0$, the functions $\{\varphi_{j}^{(R)}\}_{j=1}^{\infty}$ are orthonormal in the Hilbert space $L^{2}(B_{\infty}(p_{\infty},R),m_{\infty})$.
\end{proposition}
\begin{proof}Let $1\leq j,k<\infty$. By weak convergence of measures and the H\"older continuity of $\varphi_{j}^{(R)}$ in Proposition \ref{p.rational},
    \begin{align*}
        &\phantom{=\ \ }\int_{B_{\infty}(p_{\infty},R)} \varphi_{j}^{(R)}(x)\varphi_{k}^{(R)}(x)\dif m_{\infty}(x)\\ &=\lim_{n\rightarrow\infty} \int_{B_{\infty}(p_{\infty},R)} \varphi_{j}^{(R)}(x)\varphi_{k}^{(R)}(x)\dif \left(f_{n}\right)_{\#}\left(m_{n}|_{\overline{B_{n}(p_{n},R_{n})}}\right)\\
        &\overset{\eqref{e.vague}}{=} \lim_{n\rightarrow\infty} \int_{\overline{B_{n}(p_{n},R_{n})}}{\one}_{B_{\infty}(p_{\infty},R)}(f_{n}(x))\varphi_{j}^{(R)}(f_{n}(x))\varphi_{k}^{(R)}(f_{n}(x))\dif m_{n}(x)\\
        &:= \lim_{n\rightarrow\infty} (I_{1,n}+I_{2,n})
    \end{align*}
    where\begin{align*}
    I_{1,n}&:=\int_{B_{n}(p_{n},R)}\varphi_{j}^{(R)}(f_{n}(x))\varphi_{k}^{(R)}(f_{n}(x))\dif m_{n}(x),\\
    I_{2,n}&:=\int_{X_{n}}\left({\one}_{A_{n}}(x)-{\one}_{C_{n}}(x)\right)\varphi_{j}^{(R)}(f_{n}(x))\varphi_{k}^{(R)}(f_{n}(x))\dif m_{n}(x),\\
        A_{n}&:=\left(\overline{B_{n}(p_{n},R_{n})}\cap f_{n}^{-1}\left(B_{\infty}(p_{\infty},R)\right)\right)\setminus B_{n}(p_{n},R)\subset  B_{n}(p_{n},R+2\epsilon_{n})\setminus B_{n}(p_{n},R),\\
        C_{n}&:= B_{n}(p_{n},R)\setminus \left(\overline{B_{n}(p_{n},R_{n})}\cap f_{n}^{-1}\left(B_{\infty}(p_{\infty},R)\right)\right)\subset  B_{n}(p_{n},R)\setminus B_{n}(p_{n},R-2\epsilon_{n}),
    \end{align*}
since $\left(\overline{B_{n}(p_{n},R_{n})}\cap f_{n}^{-1}\left(B_{\infty}(p_{\infty},R)\right)\right)=(B_{n}(p_{n},R)\setminus C_{n})\cup A_{n}$. As a consequence, by Lemma \ref{l.annu} we have that \[|I_{2,n}|\leq m_{n}\left(B_{n}(p_{n},R+2\epsilon_{n})\setminus B_{n}(p_{n},R-2\epsilon_{n})\right)\cdot \lVert \varphi_{j}^{(R)}\rVert_{L^{\infty}}\lVert \varphi_{k}^{(R)}\rVert_{L^{\infty}}\rightarrow 0\text{ as } n\rightarrow\infty.\] Therefore, 
\begin{align*}
    &\int_{B_{\infty}(p_{\infty},R)} \varphi_{j}^{(R)}(x)\varphi_{k}^{(R)}(x)\dif m_{\infty}(x)= \lim_{n\rightarrow\infty}I_{1,n}\\
    =& \lim_{n\rightarrow\infty}\int_{B_{n}(p_{n},R)}\varphi_{j}^{(R)}(f_{n}(x))\varphi_{k}^{(R)}(f_{n}(x))-\varphi_{n,j}^{(R)}(f_{n}(x))\varphi_{n,k}^{(R)}(f_{n}(x))\dif m_{n}(x)\\&\ + \lim_{n\rightarrow\infty}\int_{B_{n}(p_{n},R)}\varphi_{n,j}^{(R)}(f_{n}(x))\varphi_{n,k}^{(R)}(f_{n}(x))-\phi_{n,j}^{(R)}(x)\phi_{n,k}^{(R)}(x)\dif m_{n}(x)\\&\ +\lim_{n\rightarrow\infty}\int_{B_{n}(p_{n},R)}\phi_{n,j}^{(R)}(x)\phi_{n,k}^{(R)}(x)\dif m_{n}(x)
\end{align*}
The first term is zero as $\varphi_{n,j}^{(R)}\varphi_{n,k}^{(R)}\rightarrow \varphi_{j}^{(R)}\varphi_{k}^{(R)}$ uniformly by Proposition \ref{p.rational}-\ref{lb.ra3}, and $m_{n}(B_{n}(p_{n},R))$ is uniformly bounded. The second term is zero, as $\{\phi_{n,j}^{(R)}\phi_{n,k}^{(R)}\}_{1\leq n<\infty}$ is equi-continuous by \eqref{e.hold2} {and \eqref{e.eigfunbd}}, $\varphi_{n,j}^{(R)}\circ f_{n}=\phi_{n,j}^{(R)}\circ(g_{n}\circ f_{n})$ and $d_{n}((g_{n}\circ f_{n})(x),x)\leq \epsilon_{n}\rightarrow0$ as $n\rightarrow\infty$. The third term is $\delta_{jk}$ since the functions {$\{\phi_{n,j}^{(R)}\}_{j=1}^{\infty}$} are orthonormal in $L^{2}(B_{n}(p_{n},R),m_{n})$.
\end{proof}
\begin{proposition}\label{p.heat-r}
    $q^{(R)}$ is a heat kernel on $B_{\infty}(p_{\infty},R)$. To be precise, \begin{enumerate}[label=\textup{(\arabic*)},align=right,leftmargin=*,topsep=5pt,parsep=0pt,itemsep=2pt]
        \item\label{it.heat-r1} \textup{Measurability:} The function $q^{(R)}:(0,\infty)\times B_{\infty}(p_{\infty},R)\times B_{\infty}(p_{\infty},R)\to \mathbb{R}$ is measurable with respect to the Borel $\sigma$-algebra.
        \item\label{it.heat-r2}  \textup{Symmetry:} for any $x,y\in B_{\infty}(p_{\infty},R)$, $t>0$, $q^{(R)}(t,x,y)=q^{(R)}(t,y,x)\geq0$
        \item\label{it.heat-r3}  \textup{Semigroup property:} for all $t,s>0$ and $x,y\in B_{\infty}(p_{\infty},R)$,
        \begin{equation}\label{e.qheat1}
            q^{(R)}(t+s,x,y)=\int_{B_{\infty}(p_{\infty},R)}q^{(R)}(t,x,z)q^{(R)}(s,z,y)\dif m_{\infty}(z).
        \end{equation}
        \item\label{it.heat-r4}  \textup{Contraction property:} for all $f\in L^{2}(B_{\infty}(p_{\infty},R))$, \begin{equation}\label{e.qheat2}
            \left\lVert \int_{B_{\infty}(p_{\infty},R)}q^{(R)}(t,\cdot,y) f(y)\dif m_{\infty}(y)\right\rVert_{L^{2}(B_{\infty}(p_{\infty},R))}\leq \left\lVert f\right\rVert_{L^{2}(B_{\infty}(p_{\infty},R))}.
        \end{equation}
    \end{enumerate} 
\end{proposition}
\begin{proof}
\begin{enumerate}[label=\textup{(\arabic*)},align=right,leftmargin=*,topsep=5pt,parsep=0pt,itemsep=2pt]
       \item[\ref{it.heat-r1}] Since $q^{(R)}$ is continuous by Proposition \ref{p.hs-r}-\ref{lb.hsr2}, it is measurable. 
       \item[\ref{it.heat-r2}] This follows directly from \eqref{e.real}.
       \item[\ref{it.heat-r3}] By Proposition \ref{p.hs-r}-\ref{lb.hsr1}, Proposition \ref{p.orthn} and Fubini's theorem,  
       \begin{align*}
           &\int_{B_{\infty}(p_{\infty},R)}q^{(R)}(t,x,z)q^{(R)}(s,z,y)\dif m_{\infty}(z)\\=\ &\int_{B_{\infty}(p_{\infty},R)}\sum_{j,k=1}^{\infty}\exp(-\lambda_{k}^{(R)}t-\lambda_{j}^{(R)}s)\varphi_{k}^{(R)}(x)\varphi_{k}^{(R)}(z)\varphi_{j}^{(R)}(z)\varphi_{j}^{(R)}(y)\dif m_{\infty}(z)\\=\ & \sum_{j,k=1}^{\infty}\exp(-\lambda_{k}^{(R)}t-\lambda_{j}^{(R)}s)\left( \int_{B_{\infty}(p_{\infty},R)}\varphi_{k}^{(R)}(z)\varphi_{j}^{(R)}(z)\dif m_{\infty}(z)\right) \varphi_{k}^{(R)}(x)\varphi_{j}^{(R)}(y)\\
           =\ & \sum_{j,k=1}^{\infty}\exp(-\lambda_{k}^{(R)}t-\lambda_{j}^{(R)}s)\varphi_{k}^{(R)}(x)\varphi_{j}^{(R)}(y)\delta_{jk}\\
           =\ & \sum_{j=1}^{\infty} \exp(-\lambda_{j}^{(R)}(t+s))\varphi_{j}^{(R)}(x)\varphi_{j}^{(R)}(y)= q^{(R)}(t+s,x,y).
       \end{align*}
       \item[\ref{it.heat-r4}] For any $f\in L^{2}(B_{\infty}(p_{\infty},R))$, by Fubini's theorem we have
       \begin{align*}
           &\int_{B_{\infty}(p_{\infty},R)}\left(\int_{B_{\infty}(p_{\infty},R)}q^{(R)}(t,x,y)f(y)\dif m_{\infty}(y)\right)^{2}\dif m_{\infty}(x)\\ \leq&\  
             \int_{B_{\infty}(p_{\infty},R)}\left(\int_{B_{\infty}(p_{\infty},R)}q^{(R)}(t,x,y)\left|f(y)\right|\dif m_{\infty}(y)\right)^{2}\dif m_{\infty}(x)\\ 
           \overset{\eqref{e.hs-r}}{=}&\ \int_{B_{\infty}(p_{\infty},R)}\left(\sum_{j=1}^{\infty}e^{-\lambda^{(R)}_{j}t}\left\langle \varphi_{j}^{(R)},\left|f\right|\right\rangle_{L^{2}(B_{\infty}(p_{\infty},R))}\varphi_{j}^{(R)}(x)\right)^{2}\dif m_{\infty}(x)\\ 
           {=}&\ \sum_{j=1}^{\infty}e^{-2\lambda^{(R)}_{j}t}\left\langle \varphi_{j}^{(R)},\left|f\right|\right\rangle_{L^{2}(B_{\infty}(p_{\infty},R))}^{2} \text{ (by Proposition \ref{p.orthn}) } \\
           \leq&\ \sum_{j=1}^{\infty}\left\langle \varphi_{j}^{(R)},\left|f\right|\right\rangle_{L^{2}(B_{\infty}(p_{\infty},R))}^{2}\leq \left\lVert f\right\rVert_{L^{2}(B_{\infty}(p_{\infty},R))}^{2}\text{ (by Bessel's Inequality) }
       \end{align*}
       which gives \eqref{e.qheat2}.
   \end{enumerate}
\end{proof}
We set $q^{(R)}(t,x,y)$ to zero if $x$ or $y$ outside $B_{\infty}(p_{\infty},R)$. By the monotonicity of $p^{(n,R)}$ with respect to $R$ for each $n$ {(see, for example, \cite[Theorem 2.12-(b)]{GT12})}, we can see that $0\leq q^{(R_{1})}(t,x,y)\leq q^{(R_{2})}(t,x,y)$ for all $t>0$, $0<R_{1}\leq R_{2}$ and $x,y\in X_{\infty}$, so that the limit
\begin{equation}\label{e.defq}
q(t,x,y):=\lim_{R\rightarrow\infty}q^{(R)}(t,x,y)=\lim_{\mathbb{Q}_{+}\ni R\rightarrow\infty}q^{(R)}(t,x,y)=\lim_{j\rightarrow\infty}q^{(R_{j})}(t,x,y).
\end{equation}is monotone and thus exists. We first prove that function $q(t,x,y)$ is conservative. 
\begin{proposition}\label{p.conser} The function $q$ is conservative, in the sense that:
    \begin{equation}
    \int_{X_{\infty}}q(t,x,y)\dif m_{\infty}(y)=1, \text{ for any $x\in X_{\infty}$, $t>0$.}\label{e.cons}
    \end{equation}
\end{proposition}
\begin{proof}
Fix $x\in X_{\infty}$ and choose $R$ large enough so that $g_{n}(x)\in B_{n}(p_{n},R/3)$ for all $n$ sufficiently large, say $R>4 d_{\infty}(p_{\infty},x)+1$. As $q^{(R)}(t,x,\cdot)$ is continuous on $B_{\infty}(p_{\infty},R)$, we can use the property of weak convergence and estimates the measure of annulus as the proof in Proposition \ref{p.orthn} and show that 
\begin{align}
        &\phantom{\ \leq}\int_{B_{\infty}(p_{\infty},R)}q^{(R)}(t,x,y)\dif m_{\infty}(y)=\lim_{n\rightarrow\infty}\int_{B_{n}(p_{n},R)}q^{(R)}(t,x,f_{n}(y))\dif m_{n}(y)\\
        &= \lim_{n\rightarrow\infty}\int_{B_{n}(p_{n},R)}q^{(R)}(t,x,f_{n}(y))-p^{(n,R)}(t,g_{n}(x),g_{n}(f_{n}(y))) \dif m_{n}(y)\\
        & \qquad + \lim_{n\rightarrow\infty}\int_{B_{n}(p_{n},R)}p^{(n,R)}(t,g_{n}(x),g_{n}(f_{n}(y)))-p^{(n,R)}(t,g_{n}(x),y)\dif m_{n}(y)\\
        & \qquad + \lim_{n\rightarrow\infty}\int_{B_{n}(p_{n},R)} p^{(n,R)}(t,g_{n}(x),y)\dif m_{n}(y).\label{e.hklow}
    \end{align}The first two terms vanish by \eqref{e.real} and the same argument as Proposition \ref{p.orthn}. The rest is to estimate the third term. By \cite[Corollary 3.20]{GT12} (see also \cite[Theorem 7.2]{GK17} and the computation in \cite[(5.22) in p.~549]{GH14}), since $g_{n}(x)\in B_{n}(p_{n},R/3)$, we see that for some $c,\gamma\in(0,\infty)$, 
    \begin{align}
    1&\geq\int_{B_{n}(p_{n},R)} p^{(n,R)}(t,g_{n}(x),y)\dif m_{n}(y)=\mathbb{P}^{g_{n}(x)}(t<\tau_{B_{n}(p_{n},R)})\\
&{\phantom{\geq}\begin{dcases}	
	 \geq \mathbb{P}^{g_{n}(x)}(t<\tau_{B_{n}(g_{n}(x),R/2)})\geq 1-c\exp\left(-\gamma\left({\frac{\Psi_{n}(R)}{t}}\right)^{1/(\beta^{\prime}-1)}\right)\ \text{if $R<\diam(X_{n},d_{n})$} \\
	=1 \ \text{if $\diam(X_{n},d_{n})<\infty$ and  $R\geq\diam(X_{n},d_{n})$, since $B_{n}(p_{n},R)=X_{n}$}
\end{dcases}}
\\
    &\geq 1-c\exp\left(-\gamma\left({\frac{\Psi_{n}(R)}{t}}\right)^{1/(\beta^{\prime}-1)}\right) \text{ for all $R\in(0,\infty)$}
    \end{align}
   { Here and in the following, for an open set $D\subset X_{n}$, $\tau_{D}$ is the \emph{first exit time from $D$} of the Hunt process on $X_{n}$ associated with the regular Dirichlet form $(\mathcal{E}_{n},\mathcal{F}_{n})$, see \cite[p.~1227]{GT12}.} Let $n\to\infty$, we see from \eqref{e.hklow} that for any $R>4 d_{\infty}(p_{\infty},x)+1$, \begin{equation}\label{e.comp}
    1\geq\int_{B_{\infty}(p_{\infty},R)}q^{(R)}(t,x,y)\dif m_{\infty}(y)\geq1-c\exp\left(-\gamma\left({\frac{\Psi_{\infty}(R)}{t}}\right)^{1/(\beta^{\prime}-1)}\right).
    \end{equation}
    Letting $R\to\infty$, we obtain \eqref{e.cons} since $\Psi_{\infty}(R)\to\infty$.
\end{proof}

To estimate the bounds of $q$, we need the following result on the convergence of $\Phi_{n}$.
\begin{lemma}\label{l.phicon}
Assume \ref{lb.as2} and \ref{lb.as3}. Let $\Phi_{n}:=\Phi_{\Psi_{n}}$, $1\leq n\leq\infty$ be defined by \eqref{e.phi}. Then\[\Phi_{n}\to\Phi_{\infty}\text{ converges uniformly on compact sets.}\]
\end{lemma}
\begin{proof}
Since $\Psi_{n}(1)\to\Psi_{\infty}(1)$ as $n\to\infty$, there is a positive number $C>0$ such that $C^{-1}\leq\Psi_{n}(1)\leq C$ for all $1\leq n\leq\infty$ and therefore
\[(C_{0}C)^{-1}(r^{\beta}\wedge r^{\beta^{\prime}})\leq\Psi_{j}(r)\leq (C_{0}C)(r^{\beta}\vee r^{\beta^{\prime}}),\ 1\leq j\leq\infty,\ r>0.\] Fix $0<m<M$. We first show that $\Phi_{n}\to\Phi_{\infty}$ uniformly on $[m,M]$. We may assume that $M>C$. For any $s\in[m,M]$, \[\Phi_{n}(s)\geq\sup_{r>0}\left({\frac{m}{r}-\frac{1}{(C_{0}C)^{-1}(r^{\beta}\wedge r^{\beta^{\prime}})}}\right):=C_{2}=C_{2}(m,C,C_{0},\beta,\beta^{\prime}).\]
Let $C_{3}:=C_{3}(m,M,C,C_{0},\beta,\beta^{\prime})=MC_{2}^{-1}$, we have for all $s\in [m,M]$ that
\begin{align}
\sup_{r>C_{3}}\left({\frac{s}{r}-\frac{1}{\Psi_{n}(r)}}\right)\leq{\sup_{r>C_{3}}}\frac{M}{r}&<C_{0}C(\beta-1)\left({\frac{m(C_{0}C)^{-1}}{\beta}}\right)^{\beta/(\beta-1)}\\
&\leq\Phi_{n}(s)=\sup_{r>0}\left({\frac{s}{r}-\frac{1}{\Psi_{n}(r)}}\right).\label{e.phic1}
\end{align}
Let $C_{4}:=C_{4}(C_{0},M,\beta,\beta^{\prime})=(C_{0}CM)^{-1/(\beta-1)}\vee(C_{0}CM)^{-1/(\beta^{\prime}-1)}$, we have
\begin{equation}\label{e.phic2}
\sup_{0<r<C_{4}}\left(\frac{s}{r}-\frac{1}{\Psi_{n}(r)}\right)\leq {\sup_{0<r<C_{4}}}\left(\frac{M}{r}-\frac{1}{(C_{0}C)^{-1}(r^{\beta}\vee r^{\beta^{\prime}})}\right)<0.
\end{equation}
Hence \eqref{e.phic1} and \eqref{e.phic2} give \[\Phi_{n}(s)=\sup_{C_{4}\leq r\leq C_{3}}\left({\frac{s}{r}-\frac{1}{\Psi_{n}(r)}}\right),\ s\in[m,M],\ 1\leq n<\infty.\] Since $\Psi_{n}\to \Psi_{\infty}$ uniformly on $[C_{4},C_{3}]$, we conclude that $\Phi_{n}\to\Phi_{\infty}$ uniformly on $[m,M]$, for every $0<m<M$. Next, we show that $\Phi_{n}\to\Phi_{\infty}$ uniformly on $[0,M]$, for every $M>0$. In fact, by \cite[Lemma 2.10]{Mur20}, for any $\epsilon>0$, we can find a $m>0$ small enough, independent of $n$, such that $\abs{\Phi_{n}(s)}\lesssim \epsilon$ for all $s\in[0,m]$ and $1\leq n\leq \infty$. Then choose $N$ large enough such that $\abs{\Phi_{n}(s)-\Phi_{\infty}(s)}\lesssim \epsilon$ for all $s\in[m,M]$ and all $n\geq N$. Therefore, the uniform convergence holds on $[0,M]$ for all $M>0$.
\end{proof}

\begin{theorem}\label{t.qehke}
    Let $q(t,x,y)$ be defined in \eqref{e.defq}.
\begin{enumerate}[label=\textup{(\arabic*)},align=right,leftmargin=*,topsep=5pt,parsep=0pt,itemsep=2pt]
        \item\label{it.qehk1} \textup{Measurability:} $q$ is a measurable function on $(0,\infty)\times X_{\infty}\times X_{\infty}$. {For each $t>0$, $q(t,\cdot,\cdot):X_{\infty}\times X_{\infty}\to\mathbb{R}$ is continuous.}
        \item\label{it.qehk2} \textup{Markovian and conservative property:} for any $x,y\in X_{\infty}$, $t>0$, $q(t,x,y)\geq0$ and \begin{equation}
        \int_{X_{\infty}}q(t,x,y)\dif m_{\infty}(y)=1.
        \end{equation}
        \item\label{it.qehk3} \textup{Symmetry:} for any $x,y\in X_{\infty}$, $t>0$, $q(t,x,y)=q(t,y,x)$.
        \item\label{it.qehk4} \textup{Semigroup property:} for all $t, s > 0$,\begin{equation}\label{e.hk.sg}
q(t+s,x,y)=\int_{X_{\infty}}q(t,x,z)q(s,z,y)\dif m_{\infty}(z)
        \end{equation}
        \item\label{lb.qehke} {\textup{Two-sided bound:}  there exists
		$C_{1},c_{2},c_{3}, \delta\in(0,\infty)$ such that for all $(t,x,y)\in (0,\infty)\times X_{\infty}\times X_{\infty}$, \begin{equation}
        	 \frac{c\one_{\{d_{\infty}(x,y)\leq \delta \Psi_{\infty}^{-1}(t)\}}}{m_{\infty}(B_{\infty}(x,\Psi_{\infty}^{-1}(t)))}\leq q(t,x,y)\leq \frac{C_{1}}{m_{\infty}(B_{\infty}(x,\Psi_{\infty}^{-1}(t)))}\exp\left(-c_{2}t\Phi_{\infty}\left(c_{3}\frac{d_{\infty}(x,y)}{t}\right)\right).
        \end{equation}}
    \end{enumerate}
\end{theorem}
\begin{proof}
\begin{enumerate}[label=\textup{(\arabic*)},align=right,leftmargin=*,topsep=5pt,parsep=0pt,itemsep=2pt]
    \item[\ref{it.qehk1}] Since $q^{(R)}$ is measurable by Proposition \ref{p.hs-r}-\ref{lb.hsr1}, $q$ is also measurable as it is the pointwise limit of $q^{(R)}$. {By letting $n\to\infty$ in \eqref{e.raeq1} and then letting $R\to\infty$, the continuity of $q(t,\cdot,\cdot)$ follows as the constants in \eqref{e.raeq1} are independent of $n$ and $R$.}
    \item[\ref{it.qehk2}] The positivity of $q(t,\cdot,\cdot)$ is obvious. That $q$ being conservative has been proved in Proposition \ref{p.conser}.
    \item[\ref{it.qehk3}] The symmetry of $q$ follows from the symmetry of $q^{(R)}$ in Proposition \ref{p.heat-r}.
    \item[\ref{it.qehk4}] Taking $R\rightarrow\infty$ on both sides of \eqref{e.qheat1}, and using the monotone convergence theorem, we have the semigroup property \eqref{e.hk.sg} of $q$.
    \item[\ref{lb.qehke}] By the upper bound of $p^{(n)}$,
    \begin{align*}
        q^{(n,R)}(t,x,y)&=p^{(n,R)}(t,g_{n}(x),g_{n}(y))\leq p^{(n)} (t,g_{n}(x),g_{n}(y))\\ &\lesssim \frac{C_{1}}{m_{n}(B_{n}(g_{n}(x),\Psi_{n}^{-1}(t)))}\exp\left(-c_{2}t\Phi_{n}\left(c_{3}\frac{d_{n}(g_{n}(x),g_{n}(y))}{t}\right)\right)\\&\leq \frac{C_{1}}{V_{n}(\Psi_{n}^{-1}(t)\wedge \diam (X_{n},d_{n}))}\exp\left(-c_{2}t\Phi_{n}\left(c_{3}\frac{d_{\infty}(x,y)-\epsilon_{n}}{t}\right)\right) 
    \end{align*}
Letting $n\rightarrow\infty$, by \eqref{e.minfty}, Lemmas \ref{l.diam} and \ref{l.phicon} we see that \[q^{(R)}(t,x,y)\lesssim \frac{C_{1}}{m_{\infty}(B_{\infty}(x,\Psi_{\infty}^{-1}(t)))}\exp\left(-c_{2}t\Phi_{\infty}\left(c_{3}\frac{d_{\infty}(x,y)}{t}\right)\right) \text{ for all } R>1.\]
Let $R\rightarrow\infty$ and use \eqref{e.defq} we obtain the upper bound of $q$. 

To show the lower bound, let $0<\delta<\delta^{\prime}$ be given later. Let $t>0$ and $x,y\in X_{\infty}$ with $d_{\infty}(x,y)\leq \delta \Psi_{\infty}^{-1}(t)$. Therefore $d_{\infty}(x,y)<\delta^{\prime} \Psi_{\infty}^{-1}(t)$ and we may find $N=N(t,x,y)$ large enough such that if $n>N$ then $d_{n}(g_{n}(x),g_{n}(y))<\delta^{\prime} \Psi_{n}^{-1}(t)$. Recall that under our assumption, $\hyperlink{HKE}{\mathrm{HKE}(\Psi_{n})}$ quantitatively implies the \emph{local lower estimate} $\mathrm{w}\text{-}\mathrm{LLE}(\Psi_{n})$ \cite[Theorem 3.1]{BGK12} {(The heat kernel bound $\mathrm{w}$-$\mathrm{HKE}$ considered in \cite[Theorem 3.1]{BGK12} is slightly different from our version $\hyperlink{HKE}{\mathrm{HKE}}$ in Definition~\ref{d:HKE}. Nevertheless, it is straightforward to adapt the proof in \cite[Section~4.1]{BGK12} to show that our $\hyperlink{HKE}{\mathrm{HKE}}$ quantitatively implies their $\mathrm{w}\text{-}\mathrm{LLE}$. In fact, it suffices to verify that the value $M$ defined in \cite[(4.3)]{BGK12} satisfies the same estimate as in \cite[p.~1107]{BGK12}. This follows from our $\hyperlink{HKE}{\mathrm{HKE}}$, together with \cite[(5.13)]{GK17}, the exponential decay of $\exp(-t^{-1})$ as $t\downarrow0^+$, and an argument analogous to that on \cite[p.~1106]{BGK12}. Or one can use \cite[Theorem 4.5 and Remark 4.6]{KM23} to see that $\hyperlink{HKE}{\mathrm{HKE}(\Psi)}$ implies the \emph{parabolic Harnack inequality} $\mathrm{PHI}(\Psi)$, which implies $\mathrm{w}\text{-}\mathrm{LLE}(\Psi)$ by \cite[Theorem 3.1]{BGK12})}. So there exist constants $\epsilon\in (0,1)$ and $c > 0$ such that
\[p^{B_{n}(g_{n}(x),\epsilon\Psi_{n}^{-1}(t))}(t,g_{n}(x),g_{n}(y))\geq\frac{c}{m_{n}(B_{n}(g_{n}(x),\Psi_{n}^{-1}(t)))}\gtrsim \frac{c}{V_{n}(\Psi_{n}^{-1}(t)\wedge\diam(X_{n},d_{n}))}\]
for all $t>0$ and all $d_{n}(g_{n}(x),g_{n}(y))<\epsilon\Psi_{n}^{-1}(t)$. We choose $\delta^{\prime}=\epsilon$ and $\delta=\epsilon/2$. For $R$ large enough such that $R>d_{\infty}(p_{\infty},x)+2\delta^{\prime} \Psi_{\infty}^{-1}(t)+1$, we have by triangle inequality that $g_{n}(y)\in B_{n}(g_{n}(x),\delta^{\prime} \Psi_{n}^{-1}(t))\subset B_{n}(p_{n},R)$.  By the monotonicity of $p^{(n,R)}$ with respect to $R$, we have 
\begin{align}
	q^{(n,R)}(t,x,y)=p^{(n,R)}(t,g_{n}(x),g_{n}(y))&\geq p^{B_{n}(g_{n}(x),\epsilon\Psi_{n}^{-1}(t))}(t,g_{n}(x),g_{n}(y))\\
	&\geq\frac{c}{V_{n}(\Psi_{n}^{-1}(t)\wedge\diam(X_{n},d_{n}))}.
\end{align}

Using Proposition \ref{p.rational} and letting $n\to\infty$ along some subsequence, we have $q^{(R)}(t,x,y)\geq\frac{c}{V_{\infty}(\Psi_{\infty}^{-1}(t))}$. Letting $R\to\infty$, we have the lower bound \[q(t,x,y)\geq \frac{c}{V_{\infty}(\Psi_{\infty}^{-1}(t)\wedge\diam(X_{\infty},d_{\infty}))}\gtrsim \frac{c}{{m_{\infty}(B_{\infty}(x,\Psi_{\infty}^{-1}(t)))}}.\]
\end{enumerate}
\end{proof}
\begin{theorem}\label{t.proQ}
    For each $t>0$ and each $f\in L^{\infty}(X_{\infty},m_{\infty})$, define \begin{equation}   
    Q_{t}f(x):=\int_{X_{\infty}}q(t,x,y)f(y)\dif m_{\infty}(y),\ x\in X_{\infty}.
    \end{equation}Then the family of operators $\{Q_{t}\}_{t>0}$ satisfies the following:
\begin{enumerate}[label=\textup{(\arabic*)},align=right,leftmargin=*,topsep=5pt,parsep=0pt,itemsep=2pt]
        \item\label{lb.q1} Each $Q_{t}$ extends to a bounded symmetric operator on ${L^{2}(X_{\infty},m_{\infty})}$.
        \item\label{lb.q2} \textup{Semigroup property:} $Q_{t}Q_{s}=Q_{t+s}$, $t,s>0$.
        \item\label{lb.q3} \textup{Contraction property:} $\lVert Q_{t}f\rVert_{L^{2}(X_{\infty},m_{\infty})}\leq \lVert f\rVert_{L^{2}(X_{\infty},m_{\infty})}$, $t>0$, $f\in L^{2}(X_{\infty},m_{\infty})$.
        \item\label{lb.q4} \textup{Strongly continuity:} $\lVert Q_{t}f-f\rVert_{L^{2}(X_{\infty},m_{\infty})}\rightarrow0$ as $t\rightarrow0$, for all $f\in L^{2}(X_{\infty},m_{\infty})$
        \item\label{lb.q5} \textup{Conservativity:}  $Q_{t}\one_{X_{\infty}}=\one_{X_{\infty}}$ for any $t>0$.
        \item\label{lb.q6}  \textup{Markovian property}: $0\leq Q_{t}f\leq1$ $m_{\infty}$-a.e. whenever $f\in L^{2}(X_{\infty},m_{\infty})$, $0\leq f\leq1$ $m_{\infty}$-a.e.
    \end{enumerate}
\end{theorem}
\begin{proof}
\begin{enumerate}[label=\textup{(\arabic*)},align=right,leftmargin=*,topsep=5pt,parsep=0pt,itemsep=2pt]
    \item[\ref{lb.q1}] {That $Q_{t}$ extends to a bounded symmetric operator on ${L^{2}(X_{\infty},m_{\infty})}$}
    follows from the proof of contraction property, see \ref{lb.q3} below. The symmetry of $Q_{t}$ follows directly from the symmetry of $q$.
    \item[\ref{lb.q2}] The semigroup property follows from \eqref{e.hk.sg}.
    \item[\ref{lb.q3}] For any $f\in L^{2}(X_{\infty},m_{\infty})\cap L^{\infty}(X_{\infty},m_{\infty})$, by the monotone convergence theorem,
    \begin{align*}
        \lVert Q_{t}f\rVert^{2}_{L^{2}(X_{\infty},m_{\infty})}&=\int_{X_{\infty}}\left(\int_{X_{\infty}}q(t,x,y)f(y)\dif m_{\infty}(y)\right)^{2} \dif m_{\infty}(x)\\
        &\leq \int_{X_{\infty}}\left(\int_{X_{\infty}}q(t,x,y)\left|f(y)\right|\dif m_{\infty}(y)\right)^{2} \dif m_{\infty}(x)\\
        &=\lim_{R\rightarrow\infty}\int_{B_{\infty}(p_{\infty},R)}\left(\int_{B_{\infty}(p_{\infty},R)}q^{(R)}(t,x,y)\left|f(y)\right|\dif m_{\infty}(y)\right)^{2} \dif m_{\infty}(x)\\
        &\overset{\eqref{e.qheat2}}{\leq} \lVert f\rVert_{L^{2}(X_{\infty},m_{\infty})}^{2}.
    \end{align*}
   So we may extend $Q_{t}$ to a bounded (and contractive) symmetric operator on the Hilbert space ${L^{2}(X_{\infty},m_{\infty})}$.
    \item[\ref{lb.q4}] We first show that $\lVert Q_{t}f-f\rVert_{L^{2}(X_{\infty},m_{\infty})}\rightarrow0$ as $t\rightarrow0$ holds for all $f\in C_{c}(X_{\infty})$. We may assume that $\supp(f)\subset B_{\infty}(p_{\infty},R/2)\subset B_{\infty}(p_{\infty},R)$. Since $f$ is uniformly continuous, for each $\epsilon>0$, there exists a $\delta>0$ such that $\left|f(x)-f(y)\right|<\epsilon$ whenever $d_{\infty}(x,y)\leq\delta$. By the conservativeness of $q$ in Proposition \ref{p.conser} and the H\"{o}lder's inequality,
        \begin{align*}
   & \phantom{\leq\ }\int_{X_{\infty}}\left|Q_{t}f(x)-f(x)\right|^{2}\dif m_{\infty}(x)\\&\leq\int_{X_{\infty}}\left(\int_{X_{\infty}}q(t,x,y)\left|f(x)-f(y)\right|\dif m_{\infty}(y)\right)^{2}\dif m_{\infty}(x)\\
    &\leq\int_{X_{\infty}}\int_{X_{\infty}}q(t,x,y)\left|f(x)-f(y)\right|^{2}\dif m_{\infty}(y)\dif m_{\infty}(x)\leq\sum_{j=1}^{4}I_{j}(t)
    \end{align*}
    where 
    \begin{align*}
        I_{1}(t)&:= \int_{B_{\infty}(p_{\infty},R)}\int_{B_{\infty}(p_{\infty},R)\cap B_{\infty}(x,\delta)} q(t,x,y)\left|f(x)-f(y)\right|^{2}\dif m_{\infty}(y)\dif m_{\infty}(x),\\
        I_{2}(t) &:= \int_{B_{\infty}(p_{\infty},R)}\int_{B_{\infty}(p_{\infty},R)\setminus B_{\infty}(x,\delta)} q(t,x,y)\left|f(x)-f(y)\right|^{2}\dif m_{\infty}(y)\dif m_{\infty}(x),\\
        I_{3}(t) &:= \int_{B_{\infty}(p_{\infty},R)}\int_{X_{\infty}\setminus B_{\infty}(p_{\infty},R)} q(t,x,y)\left|f(x)-f(y)\right|^{2}\dif m_{\infty}(y)\dif m_{\infty}(x),\\
        I_{4}(t) &:= \int_{X_{\infty}\setminus B_{\infty}(p_{\infty},R)}\int_{B_{\infty}(p_{\infty},R)} q(t,x,y)\left|f(x)-f(y)\right|^{2}\dif m_{\infty}(y)\dif m_{\infty}(x).
    \end{align*}
    By the uniform continuity of $f$ and Proposition \ref{p.conser}, \begin{equation}\label{e.divid1}
        I_{1}(t)\leq \epsilon^{2}\int_{B_{\infty}(p_{\infty},R)}\left(\int_{B_{\infty}(p_{\infty},R)\cap B_{\infty}(x,\delta)} q(t,x,y)\dif m_{\infty}(y)\right)\dif m_{\infty}(x)\lesssim V_{\infty}(R)\epsilon^{2}.
    \end{equation}
        By \cite[Lemma 3.19]{GT12}, \[q(t,x,y)\lesssim \frac{1}{V_{\infty}(\Psi_{\infty}^{-1}(t)\wedge\diam(X_{\infty},d_{\infty}))}\exp(-c\Psi_{\infty}(\delta)^{1/(\beta^{\prime}-1)}t^{-1/(\beta^{\prime}-1)})\] if $d_{\infty}(x,y)\geq\delta$ and $t$ is small, we have 
        {\small \begin{equation}\label{e.divid2}
            \begin{aligned}
         &\phantom{\ \leq}   I_{2}(t)\\
         &\lesssim \frac{V_{\infty}(R\wedge\diam(X_{\infty},d_{\infty}))^{2}}{V_{\infty}(\Psi_{\infty}^{-1}(t)\wedge\diam(X_{\infty},d_{\infty}))}\exp(-c\frac{\Psi_{\infty}(\delta)^{\frac{1}{\beta^{\prime}-1}}}{t^{\frac{1}{\beta^{\prime}-1}}})\lVert f\rVert^{2}_{L^{\infty}(X_{\infty},m_{\infty})}\\
            &\lesssim V_{\infty}(R\wedge\diam(X_{\infty},d_{\infty}))\left(\frac{R\wedge\diam(X_{\infty},d_{\infty})}{\Psi_{\infty}^{-1}(t)\wedge\diam(X_{\infty},d_{\infty})}\right)^{\alpha^{\prime}}\exp(-c\frac{\Psi_{\infty}(\delta)^{\frac{1}{\beta^{\prime}-1}}}{t^{\frac{1}{\beta^{\prime}-1}}})\lVert f\rVert^{2}_{L^{\infty}(X_{\infty},m_{\infty})}\\
           & \to0\ (t\to 0^{+}).
    \end{aligned}
    \end{equation}}
    
By \cite[Lemma 2.10]{Mur20}, we have $t\Phi_{\infty}(s/t)\gtrsim (s^{\beta^{\prime}}/t)^{1/(\beta^{\prime}-1)}$ whenever $s\geq t$.  Therefore, if $t\leq\Psi_{\infty}(R/2){\wedge \Psi_{\infty}(\diam (X_{\infty},d_{\infty}))}$, 
    \begin{align}
   & \phantom{\lesssim }\int_{X_{\infty}\setminus B_{\infty}(x,R)} q(t,x,y)\dif m_{\infty}(y)\\
    &\lesssim \frac{1}{V_{\infty}(\Psi_{\infty}^{-1}(t))}\int_{B_{\infty}(x,R/2)^{c}} \exp\left({-ct\Phi_{\infty}\left({\frac{d_{\infty}(x,y)}{t}}\right)}\right)\dif m_{\infty}(y)\\
    &\lesssim\frac{1}{V_{\infty}(\Psi_{\infty}^{-1}(t))}\int_{B_{\infty}(x,R/2)^{c}} \exp\left(-c\left(\frac{d_{\infty}(x,y)^{\beta^{\prime}}}{t}\right)^{1/(\beta^{\prime}-1)}\right)\dif m_{\infty}(y)\\
    &=\frac{1}{V_{\infty}(\Psi_{\infty}^{-1}(t))}\sum_{j=0}^{\infty}\int_{B_{\infty}(x,2^{j}R)\setminus B_{\infty}(x,2^{j-1}R)}\exp\left(-c\left(\frac{d_{\infty}(x,y)^{\beta^{\prime}}}{t}\right)^{1/(\beta^{\prime}-1)}\right)\dif m_{\infty}(y)\\
    &\lesssim  \sum_{j=0}^{\infty}\left( \frac{2^{j-1}R}{t^{1/\beta^{\prime}}}\right)^{\alpha^{\prime}}\exp\left(-\frac{c}{2}\left(\frac{2^{j-1}R}{t^{1/\beta^{\prime}}}\right)^{\beta^{\prime}/(\beta^{\prime}-1)}\right)\lesssim \exp\left(-\frac{c}{2}\left(\frac{R^{\beta^{\prime}}}{t}\right)^{1/(\beta^{\prime}-1)}\right).\label{e.tail}
    \end{align}
Therefore, for $t\leq\Psi_{\infty}(R/2)\wedge \Psi_{\infty}(\diam (X_{\infty},d_{\infty}))$, 
since we have assumed that $f\in C_{c}(X_{\infty})$ and that $\supp(f)\subset B_{\infty}(p_{\infty},R/2)\subset B_{\infty}(p_{\infty},R)$, we have
{\small 
\begin{align}I_{3}(t)
&=\int_{B_{\infty}(p_{\infty},R)}\int_{X_{\infty}\setminus B_{\infty}(p_{\infty},R)} q(t,x,y)\abs{f(x)-f(y)}^{2}\dif m_{\infty}(y)\dif m_{\infty}(x)\\
  &=\int_{B_{\infty}(p_{\infty},R)}\int_{X_{\infty}\setminus B_{\infty}(p_{\infty},R)} q(t,x,y)\left|f(x)\right|^{2}\dif m_{\infty}(y)\dif m_{\infty}(x) \text{ (as $\restr{f}{X_{\infty}\setminus B_{\infty}(p_{\infty},R/2)}=0$)}  \\
  &=\int_{B_{\infty}(p_{\infty},R/2)}\left|f(x)\right|^{2}\int_{X_{\infty}\setminus B_{\infty}(p_{\infty},R)} q(t,x,y)\dif m_{\infty}(y)\dif m_{\infty}(x) \text{\ (as $\restr{f}{X_{\infty}\setminus B_{\infty}(p_{\infty},R/2)}=0$)}  \\
    &\leq \int_{B_{\infty}(p_{\infty},R/2)}\left|f(x)\right|^{2}\int_{X\setminus B_{\infty}(x,R/2)} q(t,x,y)\dif m_{\infty}(y)\dif m_{\infty}(x)\\
    &\qquad \text{\ (since $X_{\infty}\setminus B_{\infty}(p_{\infty},R)\subset X\setminus B_{\infty}(x,R/2)$ for any $x\in B_{\infty}(p_{\infty},R/2))$ }\\
    &\overset{\eqref{e.tail}}{\lesssim}  V_{\infty}(R/2) \lVert f\rVert^{2}_{L^{\infty}(X_{\infty},m_{\infty})}\exp\left(-c\left(\frac{R^{\beta^{\prime}}}{t}\right)^{1/(\beta^{\prime}-1)}\right)\to 0\ \text{(when $t\to0^{+}$)} .
    \label{e.divid3}
     \end{align}
     }
By Fubini's theorem, \begin{align}
        I_{4}(t)&\lesssim \int_{X_{\infty}\setminus B_{\infty}(p_{\infty},R)}\left(\int_{B_{\infty}(p_{\infty},R)} q(t,x,y)\left|f(y)\right|^{2}\dif m_{\infty}(y)\right)\dif m_{\infty}(x)\\
        &=\int_{B_{\infty}(p_{\infty},R)}\left|f(y)\right|^{2}\left(\int_{X_{\infty}\setminus B_{\infty}(p_{\infty},R)}q(t,x,y)\dif m_{\infty}(x)\right)\dif m_{\infty}(y)\\
        &\overset{\eqref{e.tail}}{\lesssim}  \lVert f\rVert^{2}_{L^{2}(X_{\infty},m_{\infty})}\exp\left(-c\left(\frac{R^{\beta^{\prime}}}{t}\right)^{1/(\beta^{\prime}-1)}\right)\to 0\ (t\to0^{+}).\label{e.divid4}
    \end{align}
    
    Combining \eqref{e.divid1}, \eqref{e.divid2}, \eqref{e.divid3}, \eqref{e.divid4} and letting $t\rightarrow0$ we see that \[\limsup_{t\rightarrow0}\lVert Q_{t}f-f\rVert_{L^{2}(X_{\infty},m_{\infty})}^{2}\lesssim V_{\infty}(R) \epsilon^{2}.\] Since $\epsilon$ is arbitrary, we conclude that $\lim_{t\rightarrow0}\lVert Q_{t}f-f\rVert_{L^{2}(X_{\infty},m_{\infty})}=0$.
    
    For general $f\in L^{2}(X_{\infty},m_{\infty})$, given any $\epsilon>0$, we can find $g\in C_{c}(X_{\infty})$ such that $\lVert f-g\rVert_{L^{2}(X_{\infty},m_{\infty})}^{2}\leq\epsilon$. By the contraction property, $\lVert Q_{t}f-Q_{t}g\rVert_{L^{2}(X_{\infty},m_{\infty})}^{2}\leq \lVert f-g\rVert_{L^{2}(X_{\infty},m_{\infty})}^{2}\leq\epsilon$. Thus by triangle inequality, $\limsup_{t\rightarrow0}\lVert Q_{t}f-f\rVert_{L^{2}(X_{\infty},m_{\infty})}^{2}\leq 2\epsilon+\limsup_{t\rightarrow0}\lVert Q_{t}g-g\rVert_{L^{2}(X_{\infty},m_{\infty})}^{2}=2\epsilon$. Again, since $\epsilon$ is arbitrary, we conclude that $\lim_{t\rightarrow0}\lVert Q_{t}f-f\rVert_{L^{2}(X_{\infty},m_{\infty})}=0$.
    \item[\ref{lb.q5}] This follows from Proposition \ref{p.conser}.
    \item[\ref{lb.q6}] This follows immediately from the fact that $q\geq0$ and the conservative property.
\end{enumerate}
\end{proof}
Now we can prove the existence of the Dirichlet form on the limit space:
\begin{proof}[Proof of Theorem \ref{t.main}-\ref{lb.hke}]
\begin{enumerate}[label=\textup{({\roman*})},align=right,leftmargin=*,topsep=5pt,parsep=0pt,itemsep=2pt]
\item By the general theory of Dirichlet forms \cite[Chapter 1]{FOT11}, the Markovian strongly continuous semigroup $\{Q_{t}\}_{t>0}$ corresponds to a unique Dirichlet form, which is denoted by $(\mathcal{E}_{\infty},\mathcal{F}_{\infty})$.
\item The Dirichlet form $(\mathcal{E}_{\infty},\mathcal{F}_{\infty})$ is symmetric and conservative since $\{Q_{t}\}$ is. The heat kernel bound $\hyperlink{HKE}{\mathrm{HKE}(\Psi_{\infty})}$ is given in Theorem \ref{t.qehke}-\ref{lb.qehke}.
\item Since $\{Q_{t}\}$ is conservative and admits two-sided  estimates by Theorem \ref{t.qehke}-\ref{lb.qehke}, we can apply the proof of \cite[Proposition 3.2]{Lie15} and see that the set $\mathcal{C}:=\{Q_{t}f:f\in C_{0}(X_{\infty})\cap L^{2}(X_{\infty},m_{\infty}), t>0\}$ is a \emph{core} of $(\mathcal{E}_{\infty}, \mathcal{F}_{\infty})$. In fact, {by Theorem \ref{t.qehke},} \cite[Proposition 3.2]{Lie15} and \cite[Lemma 1.3.3 (i) and (iii)]{FOT11}, we see that $\mathcal{C}\subset C_{0}(X_{\infty})\cap\mathcal{F}_{\infty}$, $\mathcal{C}$ is dense in $C_{0}(X_{\infty})$ in the uniform norm and dense in $C_{0}(X_{\infty})\cap \mathcal{F}_{\infty}$ in $\mathcal{E}_{\infty,1}$. We now prove that $C_{0}(X_{\infty})\cap \mathcal{F}_{\infty}$ is dense in $\mathcal{F}_{\infty}$ in the $\mathcal{E}_{\infty,1}$ norm, so $\mathcal{C}$ is also dense in $\mathcal{F}_{\infty}$ in the $\mathcal{E}_{\infty,1}$ norm. By \cite[Lemma 1.3.3 (iii)]{FOT11}, it suffices to prove that $Q_{t}(L^{2}(X_{\infty},m_{\infty}))\subset C_{0}(X_{\infty})\cap\mathcal{F}_{\infty}$ for each $t>0$. 

From \cite[Lemma 1.3.3 (i)]{FOT11}, we know that $Q_{t}(L^{2}(X_{\infty},m_{\infty}))\subset \mathcal{F}_{\infty}$. Fix $\epsilon>0$, $f\in L^{2}(X_{\infty},m_{\infty})$ and let $R>0$ large enough such that $\lVert f\rVert_{L^{2}(X_{\infty}\setminus B_{\infty}(p_{\infty},R),m_{\infty})}\leq\epsilon$, then 
\begin{align}
\abs{Q_{t}f(z)}&\leq \int_{B_{\infty}(p_{\infty},R)}\abs{f(y)}q(t,z,y)\dif m_{\infty}(y)+\int_{X_{\infty}\setminus B_{\infty}(p_{\infty},R)}\abs{f(y)}q(t,z,y)\dif m_{\infty}(y)\\
&\leq \int_{B_{\infty}(p_{\infty},R)}\abs{f(y)}q(t,z,y)\dif m_{\infty}(y)+\epsilon q(2t,z,z)^{1/2}
\end{align}
The first term can be made small if $d_{\infty}(p_{\infty},z)$ is large enough by $\hyperlink{HKE}{\mathrm{HKE}(\Psi_{\infty})}$, so $Q_{t}f\in C_{0}(X_{\infty})$ for each $t>0$. This completes the proof that $\mathcal{C}$ is a {core}.
\item 
Let $f, g\in \mathcal{F}_{\infty}$ that have compact support. Let $K$ be a compact set such that $\supp(g)\subset K$ and assume that $f$ is constant on an open set $U$ that contains $K$. By \cite[(4.5.7)]{FOT11} and the conservativeness of $\{Q_{t}\}$, we have
\begin{align}
\mathcal{E}_{\infty}(f,g)&=\lim_{t\to0}\frac{1}{2t}\int_{X_{\infty}}\int_{X_{\infty}}(f(x)-f(y))(g(x)-g(y))q(t,x,y)\dif m_{\infty}(x)\dif m_{\infty}(y)\\
&=\lim_{t\to0}\frac{1}{2t}\Big(\int_{X_{\infty}\setminus U}\int_{K}(f(x)-f(y))(g(x)-g(y))q(t,x,y)\dif m_{\infty}(x)\dif m_{\infty}(y)\\
&\quad+\int_{K}\int_{X_{\infty}\setminus U}(f(x)-f(y))(g(x)-g(y))q(t,x,y)\dif m_{\infty}(x)\dif m_{\infty}(y)\Big).\label{e.stronglocal1}
\end{align}
By the bound of $q$, we have for $\delta:=\dist(K,X_{\infty}\setminus U)$ that, for $t\in(0,1)$, \begin{align}
&\phantom{\leq\ }\int_{X_{\infty}\setminus U}\int_{K}(f(x)-f(y))(g(x)-g(y))q(t,x,y)\dif m_{\infty}(x)\dif m_{\infty}(y)\\
&\lesssim \frac{\exp(-\frac{c}{2}\left(\Psi_{\infty}(\delta)/{t}\right)^{\frac{1}{\beta^{\prime}-1}})}{V_{\infty}(\Psi_{\infty}^{-1}(t))}\int_{X_{\infty}\setminus U}\int_{K}\frac{(f(x)-f(y))(g(x)-g(y))}{\exp(c^{\prime}d_{\infty}(x,y)^{{\beta^{\prime}}/{(\beta^{\prime}-1})})}\dif m_{\infty}(x)\dif m_{\infty}(y)\\
&\to 0 \quad(t\downarrow0)
\end{align}
Similarly we can estimate the second term in \eqref{e.stronglocal1}. Therefore $\mathcal{E}_{\infty}(f,g)=0$. By \cite[Theorem 2.4.3]{CF12}, $(\mathcal{E}_{\infty},\mathcal{F}_{\infty})$ is strongly local.
\end{enumerate}
\end{proof}

\section{Mosco convergence of Dirichlet forms}\label{s.mosco}

In this section, we will show that along some subsequence, the Dirichlet form $\left({\mathcal{E}_{\infty}}, \mathcal{F}_{\infty}\right)$ is the Mosco limit of $\left(\mathcal{E}_{n},{\mathcal{F}_{n}}\right)$ as $n \rightarrow \infty$. For each $1\leq n\leq \infty$, we identify $(\mathcal{E}_{n},{\mathcal{F}_{n}})$ with the quadratic energy form $\mathcal{E}_{n}$ on $L^{2}(X_{n},m_{n})$ defined by
\begin{equation}
\mathcal{E}_{n}(u):=
\begin{dcases}
\mathcal{E}_{n}(u,u) &\text{if $u\in\mathcal{F}_{n}$,}\\
\infty& \text{if $u\in L^{2}(X_{n},m_{n})\setminus\mathcal{F}_{n}$.}
\end{dcases}
\end{equation}

The following definition, from \cite{KS03}, extends the $L^{2}$ convergence of functions to varying spaces.
\begin{definition}\label{d.l2}
Assume that the Hilbert spaces $L^{2}(X_{n},m_{n})$ converge to $L^{2}(X_{\infty},m_{\infty})$ (see Definition \ref{d.conhil}).
\begin{enumerate}[label=\textup{(\arabic*)},align=right,leftmargin=*,topsep=5pt,parsep=0pt,itemsep=2pt]
\item A sequence $\left\{u_{n}\right\}_{1 \leq n<\infty}$ of real functions with $u_{n} \in L^{2}\left(X_{n}, m_{n}\right)$ is said to \emph{$L^{2}$ converges} to a real function $u \in L^{2}\left(X_{\infty}, m_{\infty}\right)$, or \emph{$u_{n}\to u$ in $L^{2}$}, if there exists a sequence of functions $\left\{v_{j}\right\}_{1 \leq j<\infty} \subset  \mathcal{C}$ converges to $u$ in $L^{2}\left(X_{\infty}, m_{\infty}\right)$ such that\[\lim _{j \rightarrow \infty} \limsup _{n \rightarrow \infty}\left\|\upphi_{n} v_{j}-u_{n}\right\|_{L^{2}\left(X_{n}, m_{n}\right)}=0.\]

\item A sequence $\left\{u_{n}\right\}_{1 \leq n<\infty}$ of real functions with $u_{n} \in L^{2}\left(X_{n}, m_{n}\right)$ is said to \emph{$L^{2}$ weakly converges} to a real function $u \in L^{2}\left({X}_{\infty}, m_{\infty}\right)$ if\[\lim _{n \rightarrow \infty}\left\langle u_{n}, w_{n}\right\rangle_{L^{2}\left(X_{n}, m_{n}\right)}=\langle u, w\rangle_{L^{2}\left(X_{\infty}, m_{\infty}\right)}\] for any sequence $\left\{w_{n}\right\}$ converges to $w$ {in $L^{2}$}.
\item Let $A_{n}$ be a bounded linear operator on $L^{2}(X_{n},m_{n})$, $1\leq n\leq\infty$. We say $A_{n}\to A_{\infty}$ in \emph{strong operator topology} if $A_{n}u_{n}\to A_{\infty}u_{\infty}$ in $L^{2}$ whenever $u_{n}\to u_{\infty}$ in $L^{2}$.
\end{enumerate}
\end{definition}
The following definition of Mosco convergence was originally due to Mosco \cite[Definition 2.1.1]{Mos94} for a fixed space and developed to varying spaces by Kuwae and Shioya \cite[Definition 2.11]{KS03}. 
\begin{definition}[Mosco Convergence]\label{d.mc}
We say that quadratic forms $\mathcal{E}_{n}$ on $L^{2}\left(X_{n}, m_{n}\right)$ \emph{Mosco converge} to a quadratic form $\mathcal{E}_{\infty}$ on $L^{2}\left(X_{\infty}, m_{\infty}\right)$ if the following holds:
\begin{enumerate}[label=\textup{(\arabic*)},align=right,leftmargin=*,topsep=5pt,parsep=0pt,itemsep=2pt]
 \item If a sequence $\left\{u_{n}\right\}_{1 \leq n<\infty}$ of real functions with $u_{n} \in L^{2}\left(X_{n}, m_{n}\right)$ is $L^{2}$ weakly converge to $u_{\infty}\in L^{2}\left(X_{\infty}, m_{\infty}\right)$, then\[\mathcal{E}_{\infty}\left(u_{\infty}\right) \leq \liminf _{n \rightarrow \infty} \mathcal{E}_{n}\left(u_{n}\right)\]
 \item  For any function $u$ in $L^{2}\left(X_{\infty}, m_{\infty}\right)$, there exists $\left\{u_{n}\right\}_{1 \leq n<\infty}$ of real functions with $u_{n} \in L^{2}\left(X_{n}, m_{n}\right)$, which $L^{2}$ converges to $u$ and\[\mathcal{E}_{\infty}\left(u\right)=\lim _{n \rightarrow \infty} \mathcal{E}_{n}\left(u_{n}\right)\]
\end{enumerate}
\end{definition}

The following theorem of Mosco \cite[Corollary 2.6.1]{Mos94} and Kuwae--Shioya \cite[Theorem 2.4]{KS03} is used to characterize Mosco convergence in the convergence of semigroups. 
\begin{theorem}[Mosco, Kuwae--Shioya]\label{t.mks}
Let $\mathcal{E}_{n}$ be a quadratic form on $L^{2}(X_{n},m_{n})$ for each $1\leq n\leq \infty$. Let $\{P_{t}^{(n)}\}_{t>0}$, $1\leq n\leq \infty$ be their corresponding semigroups. Then $\mathcal{E}_{n}$ Mosco converges to $\mathcal{E}_{\infty}$ if and only if $P_{t}^{(n)}$ converges to $P_{t}^{(\infty)}$ in the strong operator topology, for each $t>0$.
\end{theorem}

\begin{proof}[Proof of Theorem \ref{t.main}-\ref{lb.mosco}]
Since we take only countably many subsequences on Section \ref{s.dflim} and in \eqref{e.defq}, we can find a common subsequence such that the estimates in Proposition \ref{p.hs-r} hold for all $R\in\mathbb{Q}_{+}\cup\{R_{j}\}_{j=1}^{\infty}$. We will show that along this subsequence, $P_{t}^{(n)} u_{n} \rightarrow Q_{t} u$ in $L^{2}$ for $u_{n} \rightarrow u$ in $L^{2}$ for every $t>0$, which gives Theorem \ref{t.main}-\ref{lb.mosco} by Theorem \ref{t.mks}. 

By {Definition \ref{d.l2} and Theorem \ref{t.conv-Hil}}, there exist some $v_{j} \in C_{c}\left(X_{\infty}, d_{\infty}\right)$ such that
\begin{equation}\label{e.mc3}
v_{j} \rightarrow u \text { in } L^{2}\left(X_{\infty}, m_{\infty}\right)
\end{equation}
and 
\begin{equation}\label{e.mc4}
\lim _{j \rightarrow \infty} \limsup _{n \rightarrow \infty}\left\|\upphi_{n} v_{j}-u_{n}\right\|_{L^{2}\left(X_{n}, m_{n}\right)}=0 .
\end{equation}
where $\upphi_{n}$ is defined in \eqref{e.upphi}. We may assume that the sequences in Proposition \ref{p.isom} satisfying $\{\epsilon_{n}\}\subset(0,1/8)$ and $\{R_{j}\}\subset[1,\infty)$. Moreover we may assume that $\supp\left(v_{j}\right) \subset B_{\infty}\left(p_{\infty}, R_{j}-1\right) \subset  B_{\infty}\left(p_{\infty}, R_{j}\right)$ so that $\supp(\upphi_{n}v_{j})\subset B_{n}(p_{n},R_{j})$. Let $\phi_{j} \in C_{c}\left(B_{\infty}\left(p_{\infty}, R_{j}\right)\right)$ be a cut-off function such that \[\text{$0 \leq \phi_{j} \leq 1$ and $\left.\phi_{j}\right|_{B_{\infty}\left(p_{\infty}, R_{j}-R_{j}^{-\delta}\right)}=1$,}\] where the constant $\delta>0$ will be chosen later. Define
\begin{equation}\label{e.mc5}
w_{j}(x):=\phi_{j}(x) \cdot Q_{t}^{\left(R_{j}\right)} v_{j}(x), \quad x \in X_{\infty} .
\end{equation}
Then $w_{j} \in C_{c}\left(X_{\infty}, d_{\infty}\right)$ {since $Q_{t}^{\left(R_{j}\right)} v_{j}$ is H\"{o}lder continuous on $B_{\infty}(p_{\infty}, R_{j})$ by Remark \ref{r.Hold-r}}. We claim that:
\begin{equation}\label{e.mc6}
w_{j} \rightarrow Q_{t} u \text { in } L^{2}\left(X_{\infty}, m_{\infty}\right)
\end{equation}
and
\begin{equation}\label{e.mc7}
\lim _{j \rightarrow \infty} \limsup _{n \rightarrow \infty}\left\|\upphi_{n} w_{j}-P_{t}^{(n)} u_{n}\right\|_{L^{2}\left(X_{n}, m_{n}\right)}=0.
\end{equation}
which will imply, by Definition \ref{d.l2}, that $P_{t}^{(n)} u_{n} \rightarrow Q_{t} u$ in $L^{2}$.
\begin{enumerate}
[label=\textup{({\roman*})}, labelindent=0pt]
\item \textit{Proof of  \eqref{e.mc6}.}

By triangle inequality,
\begin{equation}\label{e.mc6+}
\left\|w_{j}-Q_{t} u\right\|_{L^{2}\left(X_{\infty}, m_{\infty}\right)}= \left\|\phi_{j} \cdot Q_{t}^{\left(R_{j}\right)} v_{j}-Q_{t} u\right\|_{L^{2}\left(X_{\infty}, m_{\infty}\right)} \leq \sum_{k=1}^{4} I_{k}(j)
\end{equation}
where \begin{align}
I_{1}(j)&:= \left\|\phi_{j} \cdot Q_{t}^{\left(R_{j}\right)} v_{j}-Q_{t}^{\left(R_{j}\right)} v_{j}\right\|_{L^{2}\left(X_{\infty}, m_{\infty}\right)}\\
I_{2}(j)&:= \left\|Q_{t}^{\left(R_{j}\right)} v_{j}-Q_{t}^{\left(R_{j}\right)}\left(u {\one}_{B_{\infty}\left(p_{\infty}, R_{j}\right)}\right)\right\|_{L^{2}\left(X_{\infty}, m_{\infty}\right)}\\
I_{3}(j)&:=\left\|Q_{t}^{\left(R_{j}\right)}\left(u {\one}_{B_{\infty}\left(p_{\infty}, R_{j}\right)}\right)-Q_{t}\left(u {\one}_{B_{\infty}\left(p_{\infty}, R_{j}\right)}\right)\right\|_{L^{2}\left(X_{\infty}, m_{\infty}\right)} \\
I_{4}(j)&:=\left\|Q_{t}\left(u {\one}_{B_{\infty}\left(p_{\infty}, R_{j}\right)}\right)-Q_{t} u\right\|_{L^{2}\left(X_{\infty}, m_{\infty}\right)}.
\end{align}
Recall that we set $q^{(R)}(t, x, y)$ to be zero if either $x$ or $y$ is outside $B_{\infty}\left(p_{\infty}, R\right)$. 
Since $V_{u}$ is doubling, we can find a constant $\delta>0$ such that
\begin{equation}\label{e.volum0}
R^{-(\delta+1)\gamma}V_{u}(R)\to0\ \text{ as }R\to\infty,
\end{equation}
where $\gamma$ is the constant in Lemma \ref{l.annu}.
\begin{enumerate}
[align=right,leftmargin=*,topsep=5pt,parsep=0pt,itemsep=2pt]
\item By a direct estimate, 
\begin{align}
&\phantom{ \ \leq}I_{1}(j)\\
 & \leq \sup _{x \in B_{\infty}\left(p_{\infty}, R_{j}\right)}\left|Q_{t}^{\left(R_{j}\right)} v_{j}(x)\right| \cdot m_{\infty}\left(B_{\infty}\left(p_{\infty}, R_{j}\right) \setminus B_{\infty}\left(p_{\infty}, R_{j}-R_{j}^{-\delta}\right)\right)^{1/2} \\
& \lesssim R_{j}^{-(\delta+1)\gamma/2}V_{\infty}(R_{j})^{1/2} \\
&\qquad \times \sup _{x \in B_{\infty}\left(p_{\infty}, R_{j}\right)}\left|\int_{B_{\infty}\left(p_{\infty}, R_{j}\right)}q^{\left(R_{j}\right)}(t, x, \cdot)^{2} \dif  m_{\infty}\right|^{1 / 2}\left\|v_{j}\right\|_{L^{2}\left(B_{\infty}\left(p_{\infty}, R_{j}\right)\right)}\\
&\quad \text{(by Cauchy--Schwarz inequality and Proposition \ref{l.annu})}\\
& \lesssim R_{j}^{-(\delta+1)\gamma/2}V_{\infty}(R_{j})^{1/2}   \sup _{x \in B_{\infty}\left(p_{\infty}, R_{j}\right)} q^{\left(R_{j}\right)}(2 t, x, x)^{1/2}\lVert u\rVert_{L^{2}\left(X_{\infty}, m_{\infty}\right)}\\
&\quad \text{(by \eqref{e.mc3}, the symmetry and semigroup property in Proposition \ref{p.heat-r})}\\
&\lesssim \frac{1}{V_{\infty}(\Psi_{\infty}^{-1}(2t))^{1/2}}R_{j}^{-(\delta+1)\gamma/2}V_{u}(R_{j})^{1/2}\lVert u\rVert_{L^{2}\left(X_{\infty}, m_{\infty}\right)}\overset{\eqref{e.volum0}}{\to} 0 \quad(j\to\infty).\label{e.mc8}
\end{align}

\item For the second term, by the contraction property \eqref{e.qheat2} of $Q^{(R_{j})}$ and triangle inequality,
\begin{align}
I_{2}(j) & \leq\left\|v_{j}-u {\one}_{B_{\infty}\left(p_{\infty}, R_{j}\right)}\right\|_{L^{2}\left(X_{\infty}, m_{\infty}\right)} \\
& \leq\left\|v_{j}-u\right\|_{L^{2}\left(X_{\infty}, m_{\infty}\right)}+\left\|u-u {\one}_{B_{\infty}\left(p_{\infty}, R_{j}\right)}\right\|_{L^{2}\left(X_{\infty}, m_{\infty}\right)} \rightarrow 0\quad(j \rightarrow \infty) .\label{e.mc9}
\end{align}
\item By definition, for any $x\in X_{\infty}$, whenever $R_{j}\geq 4d_{\infty}(p_{\infty},x)+1$, we have {by the monotone convergence theorem that}
\begin{align}
&\phantom{\leq\ }\left|Q_{t}^{\left(R_{j}\right)}\left(u {\one}_{B_{\infty}\left(p_{\infty}, R_{j}\right)}\right)(x)-Q_{t}\left(u {\one}_{B_{\infty}\left(p_{\infty}, R_{j}\right)}\right)(x)\right|^{2} \\
&\leq \left(\int_{X_{\infty}}\left|q(t, x, y)-q^{\left(R_{j}\right)}(t, x, y) \right| \left|u(y)\right| \dif m_{\infty}(y)\right)^{2}\\
&\leq  \left(\int_{X_{\infty}}\left|q(t, x, y)-q^{\left(R_{j}\right)}(t, x, y) \right|^{2}  \dif m_{\infty}(y)\right)\lVert u\rVert_{L^{2}(X_{\infty},m_{\infty})}^{2}\\
&\rightarrow 0 \quad (j\to\infty)\label{e.mc10}.
\end{align}
Therefore the function $x\mapsto\left|Q_{t}^{\left(R_{j}\right)}\left(u {\one}_{B_{\infty}\left(p_{\infty}, R_{j}\right)}\right)(x)-Q_{t}\left(u {\one}_{B_{\infty}\left(p_{\infty}, R_{j}\right)}\right)(x)\right|$ goes to $0$ pointwisely. Meanwhile, it is dominated by $2 Q_{t}(|u|)$, which is in $L^{2}\left(X_{\infty}, m_{\infty}\right)$, so we may apply the dominated convergence theorem and conclude that $I_{3}(j) \rightarrow 0$ as $j \rightarrow \infty$. 
\item By the contractive property of $Q_{t}$ in Theorem \ref{t.proQ}, we have 
\begin{equation}\label{e.mc10+}
I_{4}(j) \leq\left\|u-u {\one}_{B_{\infty}\left(p_{\infty}, R_{j}\right)}\right\|_{L^{2}\left(X_{\infty}, m_{\infty}\right)} \rightarrow 0\quad (j \rightarrow \infty).
\end{equation} 
\end{enumerate}
Combining the estimates \eqref{e.mc8}, \eqref{e.mc9}, \eqref{e.mc10}, \eqref{e.mc10+} and \eqref{e.mc6+}, we complete the proof of \eqref{e.mc6}.
\item \textit{Proof of \eqref{e.mc7}.} 

By triangle inequality,
\begin{align}
&\phantom{\leq\ }\left\|\upphi_{n} w_{j}-P_{t}^{(n)} u_{n}\right\|_{L^{2}\left(X_{n}, m_{n}\right)}\\&\leq \left\|\upphi_{n} w_{j}-P_{t}^{(n)}\left(\upphi_{n} v_{j}\right)\right\|_{L^{2}\left(X_{n}, m_{n}\right)}+\left\|P_{t}^{(n)}\left(\upphi_{n} v_{j}\right)-P_{t}^{(n)} u_{n}\right\|_{L^{2}\left(X_{n}, m_{n}\right)} \\
&\leq  \sum_{k=1}^{4} I_{k}(n, j),\label{e.mc11}
\end{align}
where
\begin{align}
I_{1}(n,j)&:= \left\|\upphi_{n} w_{j}-\upphi_{n}\left(Q_{t}^{\left(R_{j}\right)} v_{j}\right)\right\|_{L^{2}\left(X_{n}, m_{n}\right)}\\
I_{2}(n,j)&:=\left\|\upphi_{n}\left(Q_{t}^{\left(R_{j}\right)} v_{j}\right)-P_{t}^{\left(n, R_{j}\right)}\left(\upphi_{n} v_{j}\right)\right\|_{L^{2}\left(X_{n}, m_{n}\right)} \\
I_{3}(n,j)&:=\left\|P_{t}^{\left(n, R_{j}\right)}\left(\upphi_{n} v_{j}\right)-P_{t}^{(n)}\left(\upphi_{n} v_{j}\right)\right\|_{L^{2}\left(X_{n}, m_{n}\right)} \\
I_{4}(n,j)&:= \left\|P_{t}^{(n)}\left(\upphi_{n} v_{j}\right)-P_{t}^{(n)} u_{n}\right\|_{L^{2}\left(X_{n}, m_{n}\right)}
\end{align}
We estimate $I_{k}(n,j)$, $1\leq k\leq 4$ term by term as follows. 
\begin{enumerate}
[wide, labelindent=0pt]
\item \textit{Estimate $I_{1}(n, j)$.} 

By definition of $\upphi_{n}$ and $w_{j}$, if $x \in \overline{B_{n}\left(p_{n}, R_{n}\right)}$,
\begin{equation}\label{e.mc12}
\left|\upphi_{n} w_{j}(x)-\upphi_{n}\left(Q_{t}^{\left(R_{j}\right)} v_{j}\right)(x)\right|=\left(1-\phi_{j}\left(f_{n}(x)\right)\right) \cdot\left(Q_{t}^{\left(R_{j}\right)} v_{j}\right)\left(f_{n}(x)\right)
\end{equation}
Since $Q_{t}^{\left(R_{j}\right)} v_{j}$ vanishes outside ${B_{\infty}\left(p_{\infty}, R_{j}\right)}$ and is {bounded and }H\"older continuous on ${B_{\infty}\left(p_{\infty}, R_{j}\right)}$, we conclude that, as $n \rightarrow \infty$,
\begin{align}
I_{1}(n,j)^{2}=&\int_{X_{n}}\left|\upphi_{n} w_{j}(x)-\upphi_{n}\left(Q_{t}^{\left(R_{j}\right)} v_{j}\right)(x)\right|^{2} \dif  m_{n}(x)\\ 
=&\ \int_{\overline{B_{n}\left(p_{n}, R_{n}\right)}}\left(1-\phi_{j}\left(f_{n}(x)\right)\right)^{2}\left(Q_{t}^{\left(R_{j}\right)} v_{j}\right)^{2}\left(f_{n}(x)\right) \mathrm{d} m_{n}(x) \\
=&\ \int_{B_{\infty}\left(p_{\infty}, R_{j}\right)}\left(1-\phi_{j}(y)\right)^{2}\left(Q_{t}^{\left(R_{j}\right)} v_{j}\right)^{2}(y) \dif\left(f_{n}\right)_{\#} \left(m_{n}\vert_{\overline{B_{n}\left(p_{n}, R_{n}\right)}}\right)(y) \\
 \overset{\eqref{e.vague}}{\rightarrow}&\ \int_{B_{\infty}\left(p_{\infty}, R_{j}\right)}\left(1-\phi_{j}(y)\right)^{2}\left(Q_{t}^{\left(R_{j}\right)} v_{j}\right)^{2}(y) \dif m_{\infty}(y)=I_{1}(j)^{2},\label{e.mc13}
\end{align}
which means
\begin{equation}
\limsup _{n \rightarrow \infty} I_{1}(n, j)= I_{1}(j).
\end{equation}
Therefore 
\begin{equation}\label{e.mc14}
\lim_{j\to\infty}\limsup _{n \rightarrow \infty} I_{1}(n, j)=\lim_{j\to\infty} I_{1}(j)\overset{\eqref{e.mc8}}{=}0.
\end{equation}
\item \textit{Estimate $I_{2}(n, j)$.}  

For $n$ sufficiently large, since the mass between $B_{n}\left(p_{n}, R_{j}+2\epsilon_{n} \right)\setminus B_{n}\left(p_{n}, R_{j}\right)$ is small, we have by triangle inequality that
\begin{align*}
I_{2}(n, j)^{2}\lesssim & \int_{B_{n}\left(p_{n}, R_{j}\right)}\left(Q_{t}^{\left(R_{j}\right)}{v_{j}}\left(f_{n}(x)\right)-P_{t}^{\left(n, R_{j}\right)}\left(\upphi_{n} v_{j}\right)(x)\right)^{2} \dif  m_{n}(x) \\
\lesssim &  \int_{B_{n}\left(p_{n}, R_{j}\right)}\left(Q_{t}^{\left(R_{j}\right)}{v_{j}}\left(f_{n}(x)\right)-P_{t}^{\left(n, R_{j}\right)}\left(\upphi_{n} v_{j}\right)\left(g_{n}\left(f_{n}(x)\right)\right)\right)^{2} \dif  m_{n}(x) \\
& + \int_{B_{n}\left(p_{n}, R_{j}\right)}\left(P_{t}^{\left(n, R_{j}\right)}\left(\upphi_{n} v_{j}\right)\left(g_{n}\left(f_{n}(x)\right)\right)-P_{t}^{\left(n, R_{j}\right)}\left(\upphi_{n} v_{j}\right)(x)\right)^{2} \dif  m_{n}(x) \\
= &  \int_{f_{n}\left(B_{n}\left(p_{n}, R_{j}\right)\right)}\left(Q_{t}^{\left(R_{j}\right)}{v_{j}}(y)-P_{t}^{\left(n, R_{j}\right)}\left(\upphi_{n} v_{j}\right)\left(g_{n}(y)\right)\right)^{2} \dif \left(f_{n}\right)_{\#} m_{n}(y) \\
& + \int_{B_{n}\left(p_{n}, R_{j}\right)}\left(P_{t}^{\left(n, R_{j}\right)}\left(\upphi_{n} v_{j}\right)\left(g_{n}\left(f_{n}(x)\right)\right)-P_{t}^{\left(n, R_{j}\right)}\left(\upphi_{n} v_{j}\right)(x)\right)^{2} \dif  m_{n}(x) \\
:= & I_{2,1}(n, j)+I_{2,2}(n, j)
\end{align*}
To estimates $I_{2,1}(n, j)$, 
{\small
\begin{align}
&\phantom{\ \leq} \left|Q_{t}^{\left(R_{j}\right)}{v_{j}}(y)-P_{t}^{\left(n, R_{j}\right)}\left(\upphi_{n} v_{j}\right)\left(g_{n}(y)\right)\right| \\
&=  \left|\int_{B_{\infty}\left(p_{\infty}, R_{j}\right)} q^{\left(R_{j}\right)}(t, y, z) v_{j}(z) \dif m_{\infty}(z)-\int_{B_{n}\left(p_{n}, R_{j}\right)} p^{\left(n, R_{j}\right)}\left(t, g_{n}(y), z\right) v_{j}\left(f_{n}(z)\right) \dif m_{n}(z)\right| \\
&\leq  \left|\int_{B_{\infty}\left(p_{\infty}, R_{j}\right)} q^{\left(R_{j}\right)}(t, y, z) v_{j}(z) \dif m_{\infty}(z)-\int_{B_{n}\left(p_{n}, R_{j}\right)} q^{\left(n, R_{j}\right)}\left(t, y, f_{n}(z)\right) v_{j}\left(f_{n}(z)\right) \dif m_{n}(z)\right| \\
 &+\left|\int_{B_{n}\left(p_{n}, R_{j}\right)}\left(p^{\left(n, R_{j}\right)}\left(t, g_{n}(y), g_{n}\left(f_{n}(z)\right)\right)-p^{\left(n, R_{j}\right)}\left(t, g_{n}(y), z\right)\right) v_{j}\left(f_{n}(z)\right) \dif m_{n}(z)\right| \\
&\leq  \left|\int_{B_{\infty}\left(p_{\infty}, R_{j}\right)} q^{\left(R_{j}\right)}(t, y, z) v_{j}(z) \dif m_{\infty}(z)-\int_{B_{n}\left(p_{n}, R_{j}\right)} q^{\left(R_{j}\right)}\left(t, y, f_{n}(z)\right) v_{j}\left(f_{n}(z)\right) \dif m_{n}(z)\right|\\
	&\quad +\left|\int_{B_{\infty}\left(p_{\infty}, R_{j}\right)} \left(q^{\left(R_{j}\right)}(t, y, z) -q^{\left(n, R_{j}\right)}\left(t, y, z\right)\right) v_{j}\left(z\right) \dif\left(f_{n}\right)_{\#}\left(m_{n}|_{\overline{B_{n}(p_{n},R_{n})}}\right)(z)\right|\\
	&\quad + \left|\int_{B_{n}\left(p_{n}, R_{j}\right)}\left(p^{\left(n, R_{j}\right)}\left(t, g_{n}(y), g_{n}\left(f_{n}(z)\right)\right)-p^{\left(n, R_{j}\right)}\left(t, g_{n}(y), z\right)\right) v_{j}\left(f_{n}(z)\right) \dif m_{n}(z)\right|.
\label{e.mc21}
\end{align}
}
The second and third term in \eqref{e.mc21} converge to $0$ uniformly on $y$, as $n \rightarrow \infty$, by the uniform convergence in \eqref{e.real} and the equi-continuity of $p^{\left(n, R_{j}\right)}(t, \cdot, \cdot)$, respectively. For the first term, we use the expression \eqref{e.hs-r} and see that for each $y$, the first term is less or equal to 
\begin{equation}
	\sum_{k=1}^{\infty}e^{-\lambda_{k}^{(R)}t}\abs{\varphi_{k}^{(R)}(y)}D_{k,j,n}\overset{\eqref{e.ration3}}{\lesssim }\frac{\sum_{k=1}^{\infty}e^{-\lambda_{k}^{(R)}t/2}D_{k,j,n}}{V_{\infty}(\Psi_{\infty}^{-1}(t/2){\wedge\diam(X_{\infty},d_{\infty}}))^{1/2}}.
\end{equation}
where \[D_{k,j,n}:=\left|\int_{B_{\infty}\left(p_{\infty}, R_{j}\right)}\varphi_{k}^{(R)}\cdot v_{j}\dif m_{\infty}- \int_{B_{\infty}\left(p_{\infty}, R_{j}\right)}\varphi_{k}^{(R)}\cdot v_{j}\dif\left(f_{n}\right)_{\#}\left(m_{n}|_{\overline{B_{n}(p_{n},R_{n})}}\right)\right|.\]
Since for each $k$ and $j$, $\lim_{n\to\infty}D_{k,j,n}=0$ by the convergence of measures in Proposition \ref{p.mea.tan}. By \eqref{e.ration3}, $\exp(-\lambda_{k}^{(R)}t/2)D_{k,j,n}$ is dominated by \[C\norm{v_{j}}_{L^{\infty}}V_{l}(R_j){V_{\infty}(\Psi_{\infty}^{-1}(t/4)\wedge\diam(X_{\infty},d_{\infty}))^{-1/2}}e^{-\lambda_{k}^{(R)}t/4}\]which is summable with respective to $k$ by Mercer's theorem on trace \cite[Proposition 5.6.9]{Dav07} (similar to \eqref{e.Mercer}). By the dominated convergence theorem, \[\lim_{n\to\infty}\sum_{k=1}^{\infty}e^{-\lambda_{k}^{(R)}t/2}D_{k,j,n}=0.\] Therefore the first term also converges to $0$ uniformly on $y$, as $n \rightarrow \infty$. Consequently, we conclude that $\lim \sup _{n \rightarrow \infty} I_{2,1}(n, j)=0$. It is easy to see that, by the equi-continuity of $p^{\left(n, R_{j}\right)}(t, \cdot, \cdot)$ and dominated convergence theorem, $\lim \sup _{n \rightarrow \infty} I_{2,2}(n, j)= 0$. Thus
\begin{equation}\label{e.mc22}
\limsup _{n \rightarrow \infty} I_{2}(n, j)=0 .
\end{equation}
\item \textit{Estimate $I_{3}(n, j)$.}

We decompose the integration in $I_{3}(n,j)$ into two parts: \begin{align}
I_{3,1}(n,j)&:=\left\|P_{t}^{\left(n, R_{j}\right)}\left(\upphi_{n} v_{j}\right)-P_{t}^{(n)}\left(\upphi_{n} v_{j}\right)\right\|_{L^{2}\left(B_{n}\left(p_{n}, R_{j}\right), m_{n}\right)}^{2}\\
\text{and }I_{3,2}(n,j)&:=\left\|P_{t}^{\left(n, R_{j}\right)}\left(\upphi_{n} v_{j}\right)-P_{t}^{(n)}\left(\upphi_{n} v_{j}\right)\right\|_{L^{2}\left(B_{n}\left(p_{n}, R_{j}\right)^{c}, m_{n}\right)}^{2},
\end{align}
so that $I_{3}(n,j)^{2}= I_{3,1}(n,j)+I_{3,2}(n,j)$. Given any $\epsilon>0$, we fix a $j_{\epsilon} \in \mathbb{N}$ such that $\left\|v_{j}-v_{j_{\epsilon}}\right\|_{L^{2}\left(X_{\infty}, m_{\infty}\right)}<\epsilon$ for all $j \geq j_{\epsilon}$, as $\left\{v_{j}\right\}$ is a Cauchy sequence in $L^{2}\left(X_{\infty}, m_{\infty}\right)$. 

Since $\upphi_{n} v_{j}$ is supported in $B_{n}\left(p_{n}, R_{j}\right)$ when $n$ is large enough, we have by triangle inequality that
{\footnotesize \begin{align}
& \phantom{\leq\ }I_{3,1}(n,j)=\left\|P_{t}^{\left(n, R_{j}\right)}\left(\upphi_{n} v_{j}\right)-P_{t}^{(n)}\left(\upphi_{n} v_{j}\right)\right\|_{L^{2}\left(B_{n}\left(p_{n}, R_{j}\right), m_{n}\right)}^{2} \\
&=  \int_{B_{n}\left(p_{n}, R_{j}\right)}\left(\int_{B_{n}\left(p_{n}, R_{j}\right)}\left(p^{(n)}(t, x, y)-p^{(n,R_{j})}(t, x, y)\right) \upphi_{n} v_{j}(y) \dif m_{n}(y)\right)^{2} \dif  m_{n}(x) \\
&\lesssim   \int_{B_{n}\left(p_{n}, R_{j}\right)}\left(\int_{B_{n}\left(p_{n}, R_{j}\right)}\left(p^{(n)}(t, x, y)-p^{(n,R_{j})}(t, x, y)\right) \abs{\upphi_{n} v_{j}(y)-\upphi_{n} v_{j_{\epsilon}}(y)} \dif m_{n}(y)\right)^{2} \dif  m_{n}(x)\\
& \quad+ \int_{B_{n}\left(p_{n}, R_{j}\right)}\left(\int_{B_{n}\left(p_{n}, R_{j}\right)}\left(p^{(n)}(t, x, y)-p^{(n,R_{j})}(t, x, y)\right)\abs{ \upphi_{n} v_{j_{\epsilon}}(y)} \dif m_{n}(y)\right)^{2} \dif  m_{n}(x) \label{e.mc23-}
\end{align}
\normalsize}

The first term in \eqref{e.mc23-} is less than or equal to \begin{equation}
\left\lVert P^{(n)}_{t} \abs{\upphi_{n} v_{j}-\upphi_{n} v_{j_{\epsilon}}}\right\rVert_{L^{2}\left(X_{n}, m_{n}\right)}^{2}\leq \lVert\upphi_{n} v_{j}-\upphi_{n} v_{j_{\epsilon}}\rVert_{L^{2}\left(X_{n}, m_{n}\right)}^{2}
\end{equation} 
which tends to $\lVert v_{j}-v_{j_{\epsilon}}\rVert_{L^{2}(X_{\infty},m_{\infty})}^{2}\leq\epsilon^{2}$ as $n\to\infty$ by {Theorem \ref{t.conv-Hil}.} If $R_{j}<2\diam(X_{n},d_{n})$, then by the upper bound of $p^{(n)}$, we have for $y\in B_{n}\left(p_{n}, R_{j_{\epsilon}}\right)$,\begin{align}
	&0\leq \int_{B_{n}\left(p_{n}, R_{j}\right)}\left(p^{(n)}(t, x, y)-p^{(n,R_{j})}(t, x, y)\right)^{2} \dif m_{n}(x)\label{e.mc23-+-}\\
	&\leq \int_{B_{n}\left(p_{n}, R_{j}\right)}p^{(n)}(t, x, y)\cdot \left(p^{(n)}(t, x, y)-p^{(n,R_{j})}(t, x, y)\right) \dif m_{n}(x) \\
	&\qquad \text{(monotonicity of heat kernel)}\\
	&\lesssim \frac{\int_{B_{n}\left(p_{n}, R_{j}\right)}\left(p^{(n)}(t, x, y)-p^{(n,R_{j})}(t, x, y)\right) \dif m_{n}(x)}{V_{n}(\Psi_{n}^{-1}(t){\wedge\diam(X_{n},d_{n})})} \ \text{(on-diagonal upper bound)}\\
	&\leq \frac{1}{V_{n}(\Psi_{n}^{-1}(t){\wedge\diam(X_{n},d_{n})})}\left(1- \int_{B_{n}\left(p_{n}, R_{j}\right)}p^{(n,R_{j})}(t, x, y)\dif m_{n}(x) \right)\\
	&=\frac{\left(1- \mathbb{P}^{y}(t<\tau_{B_{n}\left(p_{n}, R_{j}\right)}) \right)}{V_{n}(\Psi_{n}^{-1}(t){\wedge\diam(X_{n},d_{n})})}=\frac{\mathbb{P}^{y}(\tau_{B_{n}\left(p_{n}, R_{j}\right)}\leq t)}{V_{n}(\Psi_{n}^{-1}(t){\wedge\diam(X_{n},d_{n})})} \\
	&\leq \frac{\mathbb{P}^{y}(\tau_{B_{n}\left(y, R_{j}-R_{j_{\epsilon}}\right)}\leq t)}{V_{n}(\Psi_{n}^{-1}(t){\wedge\diam(X_{n},d_{n})})}\ \\
	&\qquad {\text{(as $B_{n}\left(y, R_{j}-R_{j_{\epsilon}}\right)\subset B_{n}\left(p_{n}, R_{j}\right)$ so $\tau_{B_{n}\left(y, R_{j}-R_{j_{\epsilon}}\right)}\leq \tau_{B_{n}\left(p_{n}, R_{j}\right)}$)}}\\
	&\leq \frac{1}{V_{n}(\Psi_{n}^{-1}(t){\wedge\diam(X_{n},d_{n})})}\exp\left(-{\gamma}\left({\frac{\Psi_{n}(R_{j}-R_{j_{\epsilon}})}{t}}\right)^{1/(\beta^{\prime}-1)}\right)\label{e.mc23-+}
\end{align}where in the last inequality we use \cite[Corollary 3.20]{GT12} (or we can apply \cite[Theorem 7.2 (1)'$\Rightarrow$(7)]{GK17} and use the calculation in \cite[(5.52) in p.~549]{GH14} to obtain \cite[Theorem 7.2 (1)']{GK17}). If $\diam(X_{n},d_{n})<\infty$ and $R_{j}>\diam(X_{n},d_{n})$, then $B_{n}(p_{n},R_{j})=X_{n}$ so the integrand in \eqref{e.mc23-+-} vanishes and the upper bound \eqref{e.mc23-+} holds automatically. Consequently, the estimate \eqref{e.mc23-+} holds for all $j\in\mathbb{N}$. Since $\upphi_{n} v_{j_{\epsilon}}$ is supported in $B_{n}\left(p_{n}, R_{j_{\epsilon}}\right)$, we have by Minkowski inequality that, whenever $R_{j}\geq 8R_{j_{\epsilon}}$, the second term in \eqref{e.mc23-} is less than or equal to
{\small \begin{align}
&\phantom{\leq\ }\left(\int_{B_{n}\left(p_{n}, R_{j_{\epsilon}}\right)}\left(\int_{B_{n}\left(p_{n}, R_{j}\right)}\left(p^{(n)}(t, x, y)-p^{(n,R_{j})}(t, x, y)\right)^{2} \dif m_{n}(x)\right)^{\frac{1}{2}} \left|\upphi_{n}v_{j_{\epsilon}}(y)\right|\dif  m_{n}(y)\right)^{2}\\
&\overset{\eqref{e.mc23-+}}{\lesssim} {\frac{1}{V_{n}(\Psi_{n}^{-1}(t){\wedge\diam(X_{n},d_{n})})}\left(\int_{B_{n}\left(p_{n}, R_{j_{\epsilon}}\right)}\exp\left(-\frac{\gamma}{2}{\frac{\Psi_{n}(R_{j}-R_{j_{\epsilon}})^{\frac{1}{\beta^{\prime}-1}}}{t^{\frac{1}{\beta^{\prime}-1}}}}\right) \left|\upphi_{n}v_{j_{\epsilon}}\right|\dif  m_{n}\right)^{2}}\\
&\lesssim \frac{V_{n}(R_{j_{\epsilon}}{\wedge\diam(X_{n},d_{n})})}{V_{n}(\Psi_{n}^{-1}(t){\wedge\diam(X_{n},d_{n})})}\exp\left(-{\gamma}\left({\frac{\Psi_{n}(R_{j}-R_{j_{\epsilon}})}{t}}\right)^{1/(\beta^{\prime}-1)}\right)\lVert \upphi_{n}v_{j_{\epsilon}}\rVert_{L^{2}(X_{n},m_{n})}^{2}\\
&\to  \frac{V_{\infty}(R_{j_{\epsilon}}{\wedge\diam(X_{\infty},d_{\infty})})}{V_{\infty}(\Psi_{\infty}^{-1}(t){\wedge\diam(X_{\infty},d_{\infty})})}\exp\left(-{\gamma}\left({\frac{\Psi_{\infty}(R_{j}-R_{j_{\epsilon}})}{t}}\right)^{1/(\beta^{\prime}-1)}\right)\lVert v_{j_{\epsilon}}\rVert_{L^{2}(X_{\infty},m_{\infty})}^{2} \\
&\quad\text{as $n\to\infty$},
\end{align}} 
where we use the Cauchy--Schwarz inequality in the second inequality, and the uniform continuity of $v_{j_{\epsilon}}$ in the third line. Letting $j\to\infty$, we see that the second term in \eqref{e.mc23-} vanishes and therefore
\begin{equation}\label{e.mc23}
\limsup_{j\to\infty}\limsup_{n\to\infty}I_{3,1}(n,j)\lesssim\epsilon^{2}.
\end{equation}
On the other hand, since $P_{t}^{\left(n, R_{j}\right)}\left(\upphi_{n} v_{j}\right)=0$ on $B_{n}\left(p_{n}, R_{j}\right)^{c}$, by triangle inequality and Minkowski inequality, whenever $R_{j}\geq 8R_{j_{\epsilon}}$,
{\small
\begin{align}
&\phantom{\leq\ } \limsup _{n \rightarrow \infty}I_{3,2}(n,j)=\limsup _{n \rightarrow \infty}\left\|P_{t}^{\left(n, R_{j}\right)}\left(\upphi_{n} v_{j}\right)-P_{t}^{(n)}\left(\upphi_{n} v_{j}\right)\right\|_{L^{2}\left(B_{n}\left(p_{n}, R_{j}\right)^{c}, m_{n}\right)}^{2} \\
&\lesssim \limsup _{n \rightarrow \infty} \left\|P_{t}^{(n)}\left(\upphi_{n} v_{j}-\upphi_{n} v_{j_{\epsilon}}\right)\right\|_{L^{2}\left(X_{n}, m_{n}\right)}^{2}+\limsup _{n \rightarrow \infty} \left\|P_{t}^{(n)}\left(\upphi_{n} v_{j_{\epsilon}}\right)\right\|_{L^{2}\left(B_{n}\left(p_{n}, R_{j}\right)^{c}, m_{n}\right)}^{2} \\
&\lesssim  \epsilon^{2}+\limsup _{n \rightarrow \infty}\int_{B_{n}\left(p_{n}, R_{j}\right)^{c}}\left(\int_{B_{n}\left(p_{n}, R_{j_{\epsilon}}\right)} p^{(n)}(t, x, y) v_{j_{\epsilon}}\left(f_{n}(y)\right) \dif m_{n}(y)\right)^{2} \dif  m_{n}(x) \\
&\lesssim   \epsilon^{2}+\limsup _{n \rightarrow \infty}\left(\int_{B_{n}\left(p_{n}, R_{j_{\epsilon}}\right)}\left\|p^{(n)}(t, \cdot, y)\right\|_{L^{2}\left(B_{n}\left(y, R_{j}-R_{j_{\epsilon}}\right)^{c}, m_{n}\right)}\left|v_{j_{\epsilon}}\left(f_{n}(y)\right)\right| \dif m_{n}(y)\right)^{2} \\
&\lesssim   \epsilon^{2}+\frac{V_{\infty}(R_{j_{\epsilon}}{\wedge\diam(X_{\infty},d_{\infty})})\lVert v_{j_{\epsilon}}\rVert_{L^{2}(X_{\infty},m_{\infty})}^{2}}{V_{\infty}(\Psi_{\infty}^{-1}(t){\wedge\diam(X_{\infty},d_{\infty})})} \exp\left(-c\left({\frac{\Psi_{\infty}(R_{j}-R_{j_{\epsilon}})}{t}}\right)^{\frac{1}{\beta^{\prime}-1}}\right). \label{e.mc24}
\end{align}
}
Combining \eqref{e.mc23} and \eqref{e.mc24}, we conclude that \[\limsup_{j \rightarrow \infty} \limsup _{n \rightarrow \infty} I_{3}(n, j) \lesssim \epsilon^{2}.\] Since $\epsilon$ is arbitrary, we have
\begin{equation}\label{e.mc25}
\lim _{j \rightarrow \infty} \limsup _{n \rightarrow \infty} I_{3}(n, j)=0 .
\end{equation}
\item \textit{Estimate $I_{4}(n,j)$.}

Since $P^{(n)}_{t}$ is contractive, we have $I_{4}(n, j) \leq\left\|\upphi_{n} v_{j}-u_{n}\right\|_{L^{2}\left(X_{n}, m_{n}\right)}$. Using \eqref{e.mc4}, we have
\begin{equation}\label{e.mc26}
\lim _{j \rightarrow \infty} \limsup _{n \rightarrow \infty} I_{4}(n, j)=0
\end{equation}
\end{enumerate}
Combining \eqref{e.mc14}, \eqref{e.mc22}, \eqref{e.mc25}, \eqref{e.mc26} and \eqref{e.mc11}, we finally have \eqref{e.mc7}.
\end{enumerate}
\end{proof}

\section{Some examples}\label{s.example}
In this section, we will reconstruct conservative, regular, and strongly local Dirichlet forms on the unbounded Sierpi\'{n}ski carpet and the unbounded Sierpi\'{n}ski gasket in Examples \ref{ex.sc} and \ref{ex.SG}, respectively, as applications of Theorem \ref{t.main}. The primary components of this reconstruction are the reflected Dirichlet forms on pre-carpets and the Brownian motion on Sierpi\'{n}ski gasket cable systems. In Example \ref{ex.inverse}, we present an alternative proof of the existence of diffusions with prescribed heat kernel estimates featuring power scaling functions, which constitutes a special case of a recent result by Murugan \cite{Mur25}. {It is worth noting that the limit measure constructed in Examples \ref{ex.sc}, \ref{ex.SG}, and \ref{ex.inverse} is equivalent to the Hausdorff measure on the limit space, owing to the Ahlfors regularity \cite[Exercise 8.11]{Hei01}.}
\begin{example}\label{ex.sc}
Let $F_{0}=[0,1]^{2}\subset \mathbb{R}^{2}$ be the unit square, $l=3$, $N=8$ and $S=\{1,\ldots, N\}$. Define $q_{k}$, $k\in S$ by 
\begin{alignat*}{4}
q_{1}&:=(0,0),\ &&q_{2}:=(1/3,0),\ &&q_{3}:=(2/3,0),\ &&q_{4}:=(2/3,1/3),\\
q_{5}&:=(2/3,2/3),\ &&q_{6}:=(1/3,2/3),\ &&q_{7}:=(0,2/3),\ &&q_{8}:=(0,1/3).
\end{alignat*} 
Let $h_{k}:\mathbb{R}^{2}\to \mathbb{R}^{2}$, $k\in S$ denote the similitude $h_{k}(x)=l^{-1}x+q_{k}$. We define a sequence of decreasing compact set $\{F_{n}\}$ inductively by
\begin{equation}
F_{n+1}:=\bigcup_{k\in S}h_{k}(F_{n}),\ n\geq0.
\end{equation}
Let $F_{\infty}:=\bigcap_{n=0}^{\infty}F_{n}$ be the standard planar Sierpi\'{n}ski carpet and let \begin{equation}
X_{0}:=\bigcup_{n\geq0}l^{n}F_{n}\text{ and }X_{\infty}:=\bigcup_{n\geq0}l^{n}F_{\infty}
\end{equation}
We call $X_{0}$ the \emph{pre-carpet} and $X_{\infty}$ the unbounded Sierpi\'{n}ski carpet. Obviously, $X_{\infty}\subsetneq  X_{0}$.
\begin{enumerate}[label=\textup{(\arabic*)},align=right,leftmargin=*,topsep=5pt,parsep=0pt,itemsep=2pt]
\item\label{lb.sc1} We denote $d_{\mathbb{R}^{2}}$ be a metric on $\mathbb{R}^{2}$ (and on its subsets) defined by $d_{\mathbb{R}^{2}}((x_{1},y_{1}),(x_{2},y_{2}))=|x_{1}-x_{2}|+|y_{1}-y_{2}|\asymp \max(|x_{1}-x_{2}|,|y_{1}-y_{2}|)$. Let $d_{0}$ (resp. $d_{\infty}$) be the intrinsic distance of $X_{0}$ (resp. $X_{\infty}$) with respect to $d_{\mathbb{R}^{2}}$, i.e. $d_{0}(x,y)$ (resp. $d_{\infty}(x,y)$) is the length of the shortest path in $X_{0}$ (resp. $X_{\infty}$) connecting $x$ and $y$. It is known that $(X_{0},d_{0})$ and $(X_{\infty},d_{\infty})$ are complete separable geodesic metric spaces. Moreover, $d_{0}$ and ${d_{\mathbb{R}^{2}}}|_{X_{0}}$, $d_{\infty}$ and ${d_{\mathbb{R}^{2}}}|_{X_{\infty}}$ are bi-Lipschitz equivalent \cite{Bar13}. 

\item\label{lb.sc2} Let \[\alpha_{{\mathrm{SC}}}:=\frac{\log8}{\log 3}\text{ and }\beta_{\mathrm{SC}}:=\frac{\log (8\rho)}{\log3}\]
where $\rho$ is the `resistance scaling factor' of standard Sierpi\'{n}ski carpet so that $\beta_{\mathrm{SC}}$ is its \emph{walk dimension} \cite{BB92}. 

\item\label{lb.sc3} Let $m_{0}$ be the Lebesgue measure on $\mathbb{R}^{2}$. We define $V_{0}({r})=r^{2}\wedge r^{\alpha_{{\mathrm{SC}}}}$, $r>0$. By \cite[(2.2)]{BB00}, \begin{equation}\label{e.m0}
m_{0}(B_{0}(x,r))\asymp V_{0}(r),\ x\in X_{0},\ r>0.
\end{equation}
where $B_{0}(x,r)=\{y\in X_{0}:d_{0}(x,y)<r\}$.

\item\label{lb.sc4} It is known from \cite[Theorem 5.3]{BB00} that the heat kernel of the reflected Dirichlet form $(\mathcal{E}_{0},\mathcal{F}_{0})$ on $L^{2}(X_{0},m_{0})$ defined by
\[\mathcal{E}_{0}(u,v):=\int_{X_{0}}\nabla u\cdot\nabla v\dif m_{0},\ u,v\in\mathcal{F}_{0}:=W^{1,2}(X_{0})\]
satisfies $\hyperlink{HKE}{\mathrm{HKE}(\Psi_{0})}$, where $\Psi_{0}(r)=r^{2}\vee r^{\beta_{\mathrm{SC}}}$, $r>0$.
\item\label{lb.sc5} We define a sequence of metric measure spaces $(X_{n},d_{n},m_{n})$, $n\geq1$, by 
\begin{equation}
X_{n}:=X_{0},\ d_{n}:=l^{-n}d_{0}\ \text{and}\ m_{n}:=l^{-\alpha_{\mathrm{SC}}n}m_{0}.
\end{equation}
By \eqref{e.m0},  \begin{equation}\label{e.mn}
m_{n}(B_{n}(x,r))\asymp V_{n}(r)\ \text{for all $x\in X_{n}=X_{0}$, all $r>0$ and all $n\geq0$}.
\end{equation}
where $B_{n}(x,r):=\{y\in X_{n}:d_{n}(x,y)<r\}$ and \begin{equation}\label{e.vn}
V_{n}(r):=l^{-\alpha_{\mathrm{SC}}n}V_{0}(l^{n}r).
\end{equation}
We define a sequence of Dirichlet forms $(\mathcal{E}_{n},\mathcal{F}_{n})$ on $L^{2}(X_{n},m_{n})$ by
\begin{equation}
\mathcal{E}_{n}(u,v):=l^{(\beta_{\mathrm{SC}}-\alpha_{\mathrm{SC}})n}\mathcal{E}_{0}(u,v),\ u,v\in \mathcal{F}_{n}:=\mathcal{F}_{0}.
\end{equation}
It is easy to verify that the heat kernel $p^{(n)}(t,x,y)$ of $(\mathcal{E}_{n},\mathcal{F}_{n})$ with respect to $m_{n}$ is related to the heat kernel $p(t,x,y)$ of $(\mathcal{E}_{0},\mathcal{F}_{0})$ with respect to $m_{0}$ in the following way:
\begin{equation}\label{e.ptn}
p^{(n)}(t,x,y)=l^{\alpha_{\mathrm{SC}}n}p(l^{\beta_{\mathrm{SC}}n}t,x,y),\ x,y\in X_{n}=X_{0},\ t>0\ \text{and}\ n\geq0.
\end{equation}
Let \begin{align}\label{e.psin}
\Psi_{n}(r):=l^{-\beta_{\mathrm{SC}}n}\Psi_{0}(l^{n}r),\ r\geq0,\ n\geq1.
\end{align}
so we have 
\begin{align}
\Psi_{n}^{-1}(t)&=l^{-n}\Psi_{0}^{-1}(l^{\beta_{\mathrm{SC}}n}t)\label{e.psi-1n},\\
\Phi_{n}(s)=\Phi_{\Psi_{n}}(s)&=l^{\beta_{\mathrm{SC}}n}\Phi_{0}(l^{(1-\beta_{\mathrm{SC}})n}s).\label{e.phin}
\end{align}
Combining \eqref{e.mn}, \eqref{e.vn}, \eqref{e.ptn}, \eqref{e.psin}, \eqref{e.psi-1n} and \eqref{e.phin} we see that the metric measure Dirichlet space $(X_{n},d_{n},m_{n},\mathcal{E}_{n},\mathcal{F}_{n})$ satisfies $\hyperlink{HKE}{\mathrm{HKE}(\Psi_{n})}$ with constants independent of $n$. It is easy to verify that conditions \ref{lb.as1}-\ref{lb.as5} hold with $\beta=2$, $\beta^{\prime}=\beta_{\mathrm{SC}}$, $V_{l}(r)=V_{0}(r)$, $V_{u}(r)=V_{\infty}(r)=r^{\alpha_{\mathrm{SC}}}$ and $\Psi_{\infty}(r)=r^{\beta_{\mathrm{SC}}}$.

\item\label{lb.sc6} We will prove that $(X_{n},d_{n},q_{1})$ pointed Gromov--Hausdorff converge to $(X_{\infty},d_{\infty},q_{1})$.  Define \[\widetilde{X}_{n}:=l^{-n}X_{0}=\{l^{-n}x:x\in X_{0}\}\subset X_{0},\ n\geq0,\]which is called the unbounded scaled Sierpi\'{n}ski carpet in \cite{BB00}, so that $\bigcap_{n=1}^{\infty}X_{n}=X_{\infty}$, see Figure \ref{f.sc}. Define a metric on $\widetilde{X}_{n}$: \[ \widetilde{d}_{n}(x,y):=l^{-n}d_{0}(l^{n}x,l^{n}y),\ \forall x,y\in \widetilde{X}_{n}.\] Hence the map $x\mapsto l^{-n}x$ is a bijective isometry between $(X_{n},d_{n})$ and $(\widetilde{X}_{n},\widetilde{d}_{n})$. It is easy to see that $\widetilde{d}_{n}$ is a geodesic metric on $\widetilde{X}_{n}$ and $\widetilde{d}_{n}(x,y)$ is the length (with respect to $d_{\mathbb{R}^{2}}$) of shortest path of in $\widetilde{X}_{n}$ connecting $x$ and $y$. It suffices to show that $(\widetilde{X}_{n},\widetilde{d}_{n},q_{1})$ pointed Gromov--Hausdorff converge to $(X_{\infty},d_{\infty},q_{1})$,  in other words, the standard planar Sierpi\'{n}ski carpet $X_{\infty}$ is an \emph{asymptotic cone} \cite[Definition 8.2.7]{BBI01} of the pre-carpet $X_{0}$ at $q_{1}$. By joining geodesic in $\widetilde{d}_{n}$ for each $n$ and using compactness, we see that \begin{equation}\label{e.neardist}
\sup_{x,y\in X_{\infty}}|d_{\infty}(x,y)-\widetilde{d}_{n}(x,y)|\lesssim l^{-n}.
\end{equation}
In fact, it is obvious that $\widetilde{d}_{n}(x,y)\leq d_{\infty}(x,y)$ for all $x,y\in X_{\infty}$ since $X_{\infty}\subset \widetilde{X}_{n}$; on the other hand, if we define $H_{n}$ to be the union of boundaries of the closed squares of side $l^{-n}$ with corners in $l^{-n}\mathbb{Z}^{2}$ that are contained in $\widetilde{X}_{n}$, then $H_{n}\subset X_{\infty}$ for all $n\geq1$. For any $x\in X_{\infty}$, let $v_{n}(x)\in H_{n}$ denote the lower-left hand corner of one of the squares containing $x$ so that $d_{\infty}(x,v_{n}(x))\leq C l^{-n}$ for some $C$ by \cite[(7.4)]{BB92}. For $x,y\in X_{\infty}$, $d_{\infty}(v_{n}(x),v_{n}(y))$ is less than the length of the shortest path in $H_{n}$ connecting $v_{n}(x)$ and $v_{n}(y)$, which equals $\widetilde{d}_{n}(v_{n}(x),v_{n}(y))$, since the shortest path in $\widetilde{X}_{n}$ can be chosen to be piecewise segments that are piecewise parallel to the shortest path in $H_{n}$. Therefore \begin{align}
d_{\infty}(x,y)&\leq d_{\infty}(x,v_{n}(x))+d_{\infty}(v_{n}(x),v_{n}(y))+d_{\infty}(v_{n}(y),y)\\ &\leq \widetilde{d}_{n}(v_{n}(x),v_{n}(y))+2Cl^{-n}\\&\leq  \widetilde{d}_{n}(x,y)+4Cl^{-n}
\end{align}
which gives \eqref{e.neardist}. 

It is easy to calculate that the Hausdorff distance between $\widetilde{X}_{n}$ and $X_{\infty}$ is
$d_{\mathrm{H}}(\widetilde{X}_{n},X_{\infty})\lesssim l^{-n}$. For each $x\in \widetilde{X}_{n}$, we choose $f_{n}(x)\in X_{\infty}$ such that $d_{n}(x,f_{n}(x))=\min_{y\in X_{\infty}}d_{n}(x,y)$. Obviously, $f_{n}(q_{1})=q_{1}$ for all $n\geq1$. For $x,y\in \widetilde{X}_{n}$, we have by triangle inequality that \begin{align}
&\phantom{\leq\ }\abs{d_{\infty}(f_{n}(x),f_{n}(y))-\widetilde{d}_{n}(x,y)}\\ &\leq \abs{d_{\infty}(f_{n}(x),f_{n}(y))-\widetilde{d}_{n}(f_{n}(x),f_{n}(y))}+\abs{\widetilde{d}_{n}(f_{n}(x),x)-\widetilde{d}_{n}(f_{n}(y),y)}\lesssim l^{-n}
\end{align}
Moreover, $B_{\infty}(q_{1},r)\subset N_{Cl^{-n}}(f_{n}(B_{n}(q_{1},r)))$ for $C$ large enough. Therefore by Definition \ref{d.pgh}, $(\widetilde{X}_{n},\widetilde{d}_{n},q_{1})$ pointed Gromov--Hausdorff converge to $(X_{\infty},d_{\infty},q_{1})$.
\begin{figure}[htbp]
\subfloat
{
\includegraphics[scale=8.3]{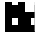}
}
\subfloat{
\includegraphics[scale=8.3]{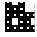}
}
\subfloat{
\includegraphics[scale=8.3]{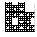}
}
\caption{Part of $(\widetilde{X}_{n},\widetilde{d}_{n})$, $1\leq n\leq3$.}
\label{f.sc}
\end{figure}

\item\label{lb.sc7} We can now apply Theorem \ref{t.main} and obtain a conservative, regular and strongly local Dirichlet form on the unbounded Sierpi\'{n}ski carpet $X_{\infty}$ satisfying $\hyperlink{HKE}{\mathrm{HKE}(r^{\beta_{\mathrm{SC}}})}$ and is the Mosco limit of a subsequence of $\{\mathcal{E}_{n}\}$. Compared to \cite{BB89}, this gives another construction of diffusion process on Sierpi\'{n}ski carpet. This argument can be adapted to construct diffusion on generalized Sierpi\'{n}ski carpet, as long as the heat kernel bounds on reflected Dirichlet spaces are known.
\end{enumerate}
\end{example}
\begin{example}
\label{ex.SG}
In this example, we re-construct a diffusion process on unbounded Sierpi\'{n}ski gasket from the diffusion process on the unbounded Sierpi\'{n}ski gasket cable system. 

Let $p_{0}:=(0,0)$, $p_{1}:=(1,0)$, $p_{2}:=(1/2,\sqrt{3}/2)$, $V_{0}=\{p_{k}\}_{k=0}^{2}$ and let $G_{0}$ be the triangle with vertices $V_{0}$. Define the similitudes $g_{k}:\mathbb{R}^{2}\to\mathbb{R}^{2}$ by $g_{k}(x):=x/2+p_{k}$, $0\leq k\leq2$. Let $V_{n+1}:=\bigcup_{k=0}^{2}g_{k}(V_{n})$, $G_{n+1}:=\bigcup_{k=0}^{2}g_{k}(G_{n})$ for $n\geq1$. Define $G:=\bigcap_{n=0}^{\infty}G_{n}$, $W_{\infty}:=\bigcup_{n=0}^{\infty}2^{n}G$. We call $W_{\infty}$ the \emph{unbounded Sierpi\'{n}ski gasket}. Let $d$ be the standard Euclidean distance on $\mathbb{R}^{2}$ and $d_{\infty}$ be the geodesic distance on $W_{\infty}$ with respect to $d$. Define $V_{*}:=\bigcup_{n=0}^{\infty}2^{n}V_{n}$ and $E:=\{(x,y)\in V_{*}\times V_{*}:d(x,y)=1\}$. Then $(V_{*}, E)$ is an infinite, locally finite, connected graph. By replacing each edge in $E$ by an isometric copy of the line segment $[0,1]$ (called \emph{cable}) and gluing them in an obvious way at the vertices, we obtain an unbounded connected closed set $W_{0}\subset \mathbb{R}^{2}$, called the corresponding \emph{cable system} of $(V_{*}, E)$  (see {Figure~\ref{fig8.2}}). We define a metric $d_{0}$ on $W_{0}$ by using Euclidean distance on each cable and extend $d_{0}$ to a metric on $W_{0}$ which agrees with the graph distance on $V_{*}$. It is easy to see that $d_{0}$ is a geodesic distance. Let $m_{0}$ be the Hausdorff measure on $(W_{0},d_{0})$, which coincides with the one-dimensional Lebesgue measure on each cable. Let 
\[\alpha_{\mathrm{SG}}:=\frac{\log3}{\log2} \text{ and }\beta_{\mathrm{SG}}:=\frac{\log5}{\log2}.\]
We define $V_{0}(r):=r\vee r^{\alpha_{\mathrm{SG}}}$ and $\Psi_{0}(r):=r^{2}\vee r^{\beta_{\mathrm{SG}}}$. It is known that \[m_{0}(B_{0}(x,r))\asymp V_{0}(r),\ r\geq0;\] see \cite[p.15562]{DRY23} and \cite[Lemma 2.1]{Bar98}. In \cite{BB04} (also \cite[Section 6]{BM18}) a conservative, strongly local regular Dirichlet form $(\mathcal{E}_{0},\mathcal{F}_{0})$ on $L^{2}(W_{0},m_{0})$ is defined, and the corresponding continuous Markov process is called the \emph{cable process}. It is known from \cite[Theorem 2.4]{Bar03} and \cite{BB04} (see also \cite[p.15563]{DRY23}) that the metric measure Dirichlet space $(W_{0},d_{0},m_{0},\mathcal{E}_{0},\mathcal{F}_{0})$ satisfies $\hyperlink{HKE}{\mathrm{HKE}(\Psi_{0})}$. We define a sequence of metric measure Dirichlet spaces $(W_{n},d_{n},m_{n},\mathcal{E}_{n},\mathcal{F}_{n})$, $n\geq1$, by 
\begin{equation}
W_{n}:=W_{0},\ d_{n}:=2^{-n}d_{0}\ \text{and}\ m_{n}:=2^{-\alpha_{\mathrm{SG}}n}m_{0}
\end{equation}
and
\begin{equation}
\mathcal{E}_{n}(u,v):=2^{(\beta_{\mathrm{SG}}-\alpha_{\mathrm{SG}})n}\mathcal{E}_{0}(u,v),\ u,v\in \mathcal{F}_{n}:=\mathcal{F}_{0}.
\end{equation}
Similar to Example \ref{ex.sc}-\ref{lb.sc5}, we see that $(W_{n},d_{n},m_{n},\mathcal{E}_{n},\mathcal{F}_{n})$ satisfies $\hyperlink{HKE}{\mathrm{HKE}(\Psi_{n})}$ with uniform constants, where  \begin{align}\label{e.psisgn}
\Psi_{n}(r):=2^{-\beta_{\mathrm{SG}}n}\Psi_{0}(2^{n}r),\ r\geq0,\ n\geq1,
\end{align}
and conditions \ref{lb.as1}-\ref{lb.as5} hold with $\beta=2$, $\beta^{\prime}=\beta_{\mathrm{SG}}$, $V_{l}(r)=V_{0}(r)$, $V_{u}(r)=V_{\infty}(r)=r^{\alpha_{\mathrm{SG}}}$ and $\Psi_{\infty}(r)=r^{\beta_{\mathrm{SG}}}$. Similar to Example \ref{ex.sc}-\ref{lb.sc6}, it is easy to verify that $(W_{n},d_{n},p_{0})$ pointed Gromov--Hausdorff converge to $(W_{\infty},d_{\infty},p_{0})$. By Theorem \ref{t.main}, we obtain a conservative, regular and strongly local Dirichlet form on the unbounded Sierpi\'{n}ski gasket $W_{\infty}$ satisfying $\hyperlink{HKE}{\mathrm{HKE}(r^{\beta_{\mathrm{SG}}})}$ and is the Mosco limit of a subsequence of $\{\mathcal{E}_{n}\}$.

\begin{figure}[htbp]
\subfloat
{
\includegraphics[scale=0.9]{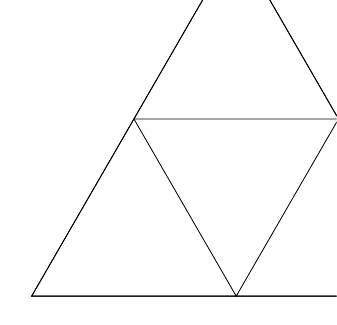}
}
\subfloat{
\includegraphics[scale=0.9]{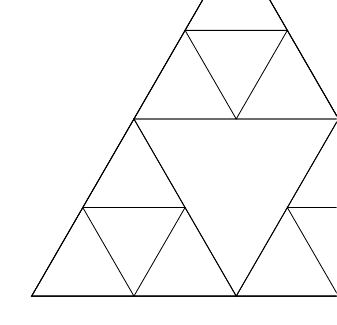}
}
\subfloat{
\includegraphics[scale=0.9]{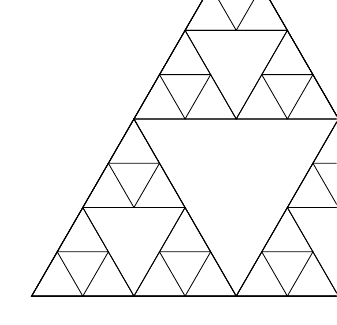}
}
\caption{Parts of the cable systems.}
\label{fig8.2}
\end{figure}
\end{example}

\begin{remark}
We note that Cao and Qiu \cite[Theorem 9.8 and Remark 9.9]{CQ23} have constructed conservative, strongly local, and regular Dirichlet forms on unconstrained Sierpi\'{n}ski carpets in $\mathbb{R}^3$ via approximations by cable systems and proved uniform parabolic Harnack inequalities and uniform heat kernel estimates. Their construction is conceptually similar to the approximation scheme used in Examples~\ref{ex.sc}, \ref{ex.SG}, and in the proof of Theorem~\ref{t.main}-\ref{lb.hke}, where the convergence of (global) heat kernels plays a crucial role. Unconstrained Sierpi\'{n}ski carpets thus provide another natural class of spaces to which Theorem~\ref{t.main} applies, as shown by the convergence of the space in \cite[Lemma 9.7-(a)]{CQ23} and the uniform heat kernel bounds in \cite[Theorem 9.4]{CQ23}. In particular, the Mosco convergence of the energy forms holds by Theorem~\ref{t.main}-\ref{lb.mosco}.
\end{remark}

\begin{example}\label{ex.inverse}
Barlow in \cite{Bar04} shows that, the existence of an infinite locally finite weighted graph that satisfying $\mathrm{V}_{\alpha}$, $\mathrm{EHI}$ and $\mathrm{E}_{\beta}$ (see \cite{Bar04} for definitions), which is equivalent to the random walk satisfying the following type heat kernel estimates: for some $C_{1},c_{2}\geq1$ and all $n\geq1$,
\begin{equation}\label{e.hkegraph}
\begin{aligned}
p_{n}(x,y)&\leq \frac{C_{1}}{n^{\alpha/\beta}}\exp\left(-c_{2}\left(\frac{d(x,y)^{\beta}}{n}\right)^{1/(\beta-1)}\right)\\
\text{and \quad}p_{n}(x, y) + p_{n+1}(x, y) &\geq \frac{C^{-1}_{1}}{n^{\alpha/\beta}}\exp\left(-c^{-1}_{2}\left(\frac{d(x,y)^{\beta}}{n}\right)^{1/(\beta-1)}\right),
\end{aligned}
\end{equation}
is equivalent to \begin{equation}\label{e.alphabeta}
\alpha\geq1\text{ and }2\leq\beta\leq\alpha+1.
\end{equation} 
In a recent paper by Murugan \cite{Mur25}, the analogous result on diffusion on metric measure space is proved: the existence of a metric measure Dirichlet space satisfying the full heat kernel bounds $\hyperlink{HKEf}{\mathrm{HKE}_{\mathrm{f}}(\Psi)}$ with $m(B(x,r))\asymp V(r)$ is equivalent to \begin{equation}\label{e.alphabeta1}
\frac{R^{2}}{r^{2}}\lesssim \frac{\Psi(R)}{\Psi(r)}\lesssim \frac{RV(R)}{rV(r)} \quad\forall\ 0<r\leq R;
\end{equation}see Remark \ref{r.main}-\ref{lb.rm1} for definition of $\hyperlink{HKEf}{\mathrm{HKE}_{\mathrm{f}}(\Psi)}$. In the special case of $V(r)=r^{\alpha}$ and $\Psi(r)=r^{\beta}$, \eqref{e.alphabeta1} is equivalent to \eqref{e.alphabeta}. In this example, we give an independent proof of a special case that \eqref{e.alphabeta} implying the existence of such metric measure Dirichlet space, as an application of Theorem \ref{t.main}. The method we use here is to take the pointed Gromov--Hausdorff limit of the cable systems of the graphs constructed by Barlow in \cite{Bar04}.
\begin{theorem}[{cf. \cite[Theorem 2.2]{Mur25}}]\label{t.inverse}
For any $\alpha\geq1\text{ and }2\leq\beta\leq\alpha+1$, there is a metric measure space $(X,d,m)$ and a conservative, regular and strongly local symmetric Dirichlet form on $L^{2}(X,m)$ that satisfies $\hyperlink{HKEf}{\mathrm{HKE}_{\mathrm{f}}(\Psi)}$ with $\Psi(r)=r^{\beta}$ and $m(B(x,r))\asymp r^{\alpha}$.
\end{theorem}
\begin{proof}
\begin{enumerate}[label=\textup{({\arabic*})},align=right,leftmargin=*,topsep=5pt,parsep=0pt,itemsep=2pt]
\item Fix $\alpha$ and $\beta$ such that $\alpha\geq1\text{ and }2\leq\beta\leq\alpha+1$. By \cite{Bar04} and \cite[Theorem 3.1]{GT02}, there is an infinite locally finite weighted graph $\Gamma=(V,E,\mu)$ satisfying \eqref{e.hkegraph}. It can be directly checked that $\Gamma$ also satisfies the $p_{0}$-condition stated in \cite[(1.5)]{BB04} (or use the heat kernel lower bound in \eqref{e.hkegraph}). Let us denote the cable system corresponding to $\Gamma$ by $\Gamma^{(c)}$. Similar to Example \ref{ex.SG}, we equip $\Gamma^{(c)}$ a metric $d$ by extending the Euclidean distance on each cable, and let $m$ be the Hausdorff measure on $\Gamma^{(c)}$, so that $(\Gamma^{(c)},d,m)$ is a proper metric measure space. The Dirichlet form corresponding to the cable process on $(\Gamma^{(c)},d,m)$ is denoted by $(\mathcal{E},\mathcal{F})$. 

\item We denote $V(r):=r\vee r^{\alpha}$ and $\Psi(r):=r^{2}\vee r^{\beta}$. By \cite[Lemma 2.1]{BB04},\[m(B(x,r))\asymp V(r),\ r\geq0.\] The elliptic Harnack inequality $\mathrm{EHI}$ holds on $(\mathcal{E},\mathcal{F})$ by \cite[Corollary 2.5]{BB04}. The mean exit time estimate $\mathrm{E}_{\beta}$ on cable process holds by \cite[Lemma 2.6]{BB04}. Therefore we can apply \cite[Theorem 5.15]{GT12} and conclude that $(\mathcal{E},\mathcal{F})$ on $L^{2}(\Gamma^{(c)},m)$ satisfies $\hyperlink{HKE}{\mathrm{HKE}(\Psi)}$. Fix any sequence of positive real numbers $\{\varepsilon_{n}\}_{n=1}^{\infty}$ with $\varepsilon_{1}=1$ and $\varepsilon_{n}\downarrow 0$. We define a sequence of metric measure Dirichlet spaces $(\Gamma^{(c)}_{n},d_{n},m_{n},\mathcal{E}_{n},\mathcal{F}_{n})$, $n\geq1$, by 
\begin{equation}
\Gamma^{(c)}_{n}:=\Gamma^{(c)},\ d_{n}:=\varepsilon_{n}d\ \text{and}\ m_{n}:=\varepsilon_{n}^{\alpha}m
\end{equation}
and
\begin{equation}
\mathcal{E}_{n}(u,v):=\varepsilon_{n}^{\beta-\alpha}\mathcal{E}(u,v),\ u,v\in \mathcal{F}_{n}:=\mathcal{F}.
\end{equation}
Define $V_{n}(r):=\varepsilon_{n}^{\alpha}V(\varepsilon_{n}^{-1}r)$ and $\Psi_{n}(r):=\varepsilon_{n}^{\beta}\Psi(\varepsilon_{n}^{-1}r)$ for $r\geq0$ and $n\geq1$. Then the metric measure Dirichlet space $(\Gamma^{(c)}_{n},d_{n},m_{n},\mathcal{E}_{n},\mathcal{F}_{n})$ satisfies \ref{lb.as1}-\ref{lb.as5} with $V_{\infty}(r)=r^{\alpha}$, $\Psi_{\infty}(r)=r^{\beta}$. 
\item Fix $p\in \Gamma^{(c)}=\Gamma^{(c)}_{n}$. We claim that $(\Gamma^{(c)}_{n},d_{n}, p)$ is \emph{pointed totally bounded} \cite[p.321]{HKST15}: there is a function $N : (0, \infty)\times (0, \infty) \to (0, \infty)$ such that, for each $0 < r < R$ and $n\geq1$, the closed ball $B_{n}(p,R)$ in $(\Gamma^{(c)}_{n},d_{n})$ contains an $r$-net of cardinality at most $N(r,R)$. Since in the sense of sets, $B_{n}(p,R)$ is identical to $B_{1}(p,\varepsilon_{n}^{-1}R)$, which, by Zorn's Lemma, contains an $\varepsilon_{n}^{-1}r$-net of cardinality at most $C(R/r)^{\alpha}$ by the existence of doubling measure $m$ and \cite[Exercise 10.17]{Hei01} for $C\geq1$ independent of $n$. We thus may take $N(r,R)=C(R/r)^{\alpha}$.
\item Since $(\Gamma^{(c)}_{n},d_{n}, p)$ are proper length and pointed totally bounded, we conclude by Gromov's compactness theorem \cite[Theorem 11.3.16]{HKST15} that there is a proper length pointed metric space $(\Gamma^{(c)}_{\infty},d_{\infty}, p_{\infty})$ such that \[(\Gamma^{(c)}_{n},d_{n}, p)\xrightarrow{\mathrm{p-GH}}(\Gamma^{(c)}_{\infty},d_{\infty}, p_{\infty})\]
along some subsequence. Therefore \ref{lb.as6} holds along this subsequence.
\item Now it is time to apply Theorem \ref{t.main} and Remark \ref{r.main}-\ref{lb.rm1}, and conclude that there is a measure $m_{\infty}$ on $(\Gamma^{(c)}_{\infty},d_{\infty})$ that is Ahlfors $\alpha$-regular, i.e., $m_{\infty}(B_{\infty}(x,r))\asymp r^{\alpha}$, and there is a conservative, regular and strongly local symmetric Dirichlet form on $L^{2}(\Gamma^{(c)}_{\infty},m_{\infty})$ that satisfies heat kernel bounds $\hyperlink{HKEf}{\mathrm{HKE}_{\mathrm{f}}(r^{\beta})}$.
\end{enumerate}
\end{proof}
\end{example}
\begin{remark}

We conclude this paper with several remarks on the scope and limitations of Theorem \ref{t.main}.
\begin{enumerate}[label=\textup{({\arabic*})},align=right,leftmargin=*,topsep=5pt,parsep=0pt,itemsep=2pt]
\item In certain situations, the convergence of the walk dimension (Assumption \ref{lb.as3}) is not necessary for the convergence of Brownian motions. For example, in recent work \cite{CHK25}, it was proved that the Brownian motions on the $l$-level Sierpi\'{n}ski gasket $\mathrm{SG}(l)$, under suitable time changes, converge to the standard reflected Brownian motion on a tetrahedron as $l\to\infty$. However, Theorem \ref{t.main} does not cover this situation, since the convergence of the walk dimension required in Assumption \ref{lb.as3} fails in this setting; see \cite[Remark 1.4-(b)]{CHK25}.
\item In Theorem \ref{t.main}-\ref{lb.as4}, the volume function $(x,r)\mapsto m_{j}(B_{j}(x,r))$ is assumed to be uniformly comparable to a function $V_{j}(r)$ that is independent of the center $x$. In particular, the volume growth is required to be essentially uniform in $x$. This assumption is essential for obtaining the upper bound for the eigenvalues in Proposition \ref{p.weyl}. Nevertheless, it restricts the scope of the result, and it would be desirable to obtain a similar stability result under weaker volume regularity conditions.
\item The proof of Theorem \ref{t.main} relies heavily on the use of the heat kernel and its two-sided estimates. Such tools are generally unavailable for \emph{$p$-energy forms} on metric measure spaces when $p\neq2$. Extending stability results of this type to the setting of $p$-energy forms therefore requires substantially different techniques, and remains an interesting direction for future research.
\end{enumerate}
\end{remark}
\vspace{5pt}

\noindent \textbf{Acknowledgments.} 
I am very grateful to Mathav Murugan for introducing me to the problem of constructing a Dirichlet form on limit space, for the suggestion to apply the main result to Example \ref{ex.inverse}, and for many helpful discussions regarding this topic. I thank Jiaxin Hu for his helpful comments and interest in this paper. The author is very grateful to the anonymous referees for their careful reading of this manuscript and many helpful suggestions.

\vspace{11pt}
\noindent Department of Mathematical Sciences, Tsinghua University, Beijing 100084, China

\vspace{3pt}
\noindent \texttt{aobochen.math@hotmail.com}

\vspace{-2pt}
\noindent \texttt{cab21@mails.tsinghua.edu.cn}
\end{document}